\documentclass[a4paper]{amsart}

\usepackage{etex}
\usepackage{bbm}
\usepackage{pbox}

\newcommand{\bk}{\mathbbm{k}}

\newcommand{\cC}{\mathcal{C}}\newcommand{\cD}{\mathcal{D}}

\newcommand{\cI}{\mathcal{I}}\newcommand{\cJ}{\mathcal{J}}
\newcommand{\cL}{\mathcal{L}}
\newcommand{\cM}{\mathcal{M}}

\newcommand{\bN}{\mathbb{N}}

\newcommand{\bQ}{\mathbb{Q}}\newcommand{\bR}{\mathbb{R}}

\newcommand{\bZ}{\mathbb{Z}}

\usepackage{lmodern}
\usepackage{booktabs}
\usepackage[dvipsnames,svgnames,x11names,hyperref]{xcolor}
\usepackage{mathtools,url,verbatim,amssymb,enumerate,stmaryrd,nicefrac}
\usepackage[pagebackref,colorlinks,citecolor=Mahogany,linkcolor=Mahogany,urlcolor=Mahogany,filecolor=Mahogany]{hyperref}
\usepackage{microtype}
\usepackage[margin=1.5in]{geometry}

\usepackage{tikz, tikz-cd}
\usetikzlibrary{matrix,calc,arrows}

\newtheorem{theorem}{Theorem}[section]
\newtheorem{lemma}[theorem]{Lemma}
\newtheorem{proposition}[theorem]{Proposition}
\newtheorem{corollary}[theorem]{Corollary}
\newtheorem*{corollary*}{Corollary}

\newtheorem{atheorem}{Theorem}

\newtheorem{acorollary}[atheorem]{Corollary}

\theoremstyle{definition}
\newtheorem{definition}[theorem]{Definition}

\newtheorem{convention}[theorem]{Convention}
\theoremstyle{remark}
\newtheorem{example}[theorem]{Example}
\newtheorem{remark}[theorem]{Remark}

\newcommand{\half}{\nicefrac{1}{2}}
\newcommand{\cat}[1]{\mathsf{#1}}
\newcommand{\mr}[1]{{\rm #1}}

\newcommand{\ul}[1]{\underline{#1}}
\newcommand{\fS}{\mathfrak{S}}

\newcommand{\lra}{\longrightarrow}

\newcommand{\Sp}{\mathbf{Sp}}
\newcommand{\OO}{\mathbf{O}}
\newcommand{\SO}{\mathbf{SO}}
\newcommand{\GG}{\mathbf{G}}

\DeclareMathOperator*{\colim}{colim}

\AtBeginDocument{%
	\def\MR#1{}
}

\title{The cohomology of Torelli groups is algebraic}

\author{Alexander Kupers}
\email{kupers@math.harvard.edu}
\address{Department of Mathematics \\
	One Oxford Street \\
	Cambridge MA, 02138 \\USA}

\author{Oscar Randal-Williams}
\email{o.randal-williams@dpmms.cam.ac.uk}
\address{Centre for Mathematical Sciences\\
	Wilberforce Road\\
	Cambridge CB3 0WB\\
	UK}

\date{\today}

\begin{document}
	
	\begin{abstract}The Torelli group of $W_g = \#^g S^n \times S^n$ is the subgroup of the diffeomorphisms of $W_g$ fixing a disc which act trivially on $H_n(W_g;\bZ)$. The rational cohomology groups of the Torelli group are representations of an arithmetic subgroup of $\mr{Sp}_{2g}(\bZ)$ or $\mr{O}_{g,g}(\bZ)$. In this paper we prove that for $2n \geq 6$ and $g \geq 2$, they are in fact algebraic representations. Combined with previous work, this determines the rational cohomology of the Torelli group in a stable range. We further prove that the classifying space of the Torelli group is nilpotent.\end{abstract}
	
	\maketitle  

\section{Introduction} Let $W_{g}$ denote the $2n$-dimensional manifold $\#^g S^n \times S^n$, and $\mr{Diff}(W_g,D^{2n})$ denote the topological group of diffeomorphisms of $W_{g}$ fixing an open neighbourhood of a disc $D^{2n} \subset W_g$ pointwise, in the $C^\infty$-topology. Only certain automorphisms of $H_n(W_g;\bZ)$ can be realised by diffeomorphisms: in particular, they must respect the intersection form, giving a homomorphism
\[\alpha_g \colon \mr{Diff}(W_g,D^{2n}) \lra G_g \coloneqq \begin{cases} \mr{Sp}_{2g}(\bZ) & \text{if $n$ is odd,}\\
\mr{O}_{g,g}(\bZ) & \text{if $n$ is even,}\end{cases}\]
whose image we denote by $G'_g$. By work of Kreck, in dimension $2n \geq 6$ the only additional constraint is that the automorphism preserves a certain quadratic refinement of the intersection form \cite{kreckisotopy}. Thus $G'_g$ is a finite index subgroup of $G_g$ and hence an arithmetic subgroup associated to the algebraic group $\GG \in \{\Sp_{2g},\OO_{g,g}\}$. The kernel of $\alpha_g$ is called the \emph{Torelli group} and denoted by $\mr{Tor}(W_g,D^{2n})$; it is equipped with an outer action of $G'_g$.

In \cite{kupersdisk} the first author proved that the rational cohomology groups of $B\mr{Tor}(W_g,D^{2n})$ are finite-dimensional in each degree as long as $2n \geq 6$. It is then a consequence of Margulis super-rigidity that they are \emph{almost algebraic representations} of $G'_g$, i.e.\ there is a finite index subgroup of $G'_g$ such that the restriction of the representation to this subgroup extends to a rational representation of the algebraic group $\GG$ \cite[1.3.(9)]{serrearithmetic}. The purpose of this paper is to show that the rational cohomology groups are not just almost algebraic, but \emph{algebraic}: no restriction to a finite index subgroup is necessary.

\begin{atheorem}\label{thm:main} For $2n \geq 6$ and $g \geq 2$, the  $H^i(B\mr{Tor}(W_g,D^{2n});\bQ)$  are algebraic representations of $G'_g$.\end{atheorem}

\begin{remark}
	Of course in the case $g=0$, such a statement is trivial as $G'_g = \{e\}$. For $g=1$ and $n$ odd, extensions of algebraic representations need no longer split and our techniques only imply that these rational cohomology groups have a filtration whose associated graded consists of algebraic representations. For $g=1$ and $n$ even, we do not obtain any information.
\end{remark}

Theorem \ref{thm:main} may be used in conjunction with the main result of \cite{KR-WTorelli}: for $2n \geq 6$ that paper computes the maximal algebraic subrepresentation of $H^*(B\mr{Tor}(W_g,D^{2n});\bQ)$ in a stable range; by Theorem \ref{thm:main} this is in fact the entire cohomology.

If $G_g'' \leq G'_g$ is a finite index subgroup, let us write $\mr{Diff}^{G''_g}(W_g, D^{2n}) \coloneqq \alpha_g^{-1}(G_g'')$.  Theorem \ref{thm:main} implies that its cohomology with coefficients in any algebraic $G_g'$-representation (as usual, these are representations over the rationals) is independent of the choice of finite index subgroup in a stable range:

\begin{acorollary}\label{cor:main}
Let $2n \geq 6$, $g \geq 2$, $V$ be an algebraic $G'_g$-representation, and if $n$ is even then assume that $G''_g$ is not entirely contained in $\SO_{g,g}(\bZ)$. Then the natural map
\[H^*(B\mr{Diff}(W_g, D^{2n});V) \lra H^*(\mr{Diff}^{G''_g}(W_g, D^{2n});V),\]
which is a split injection by transfer, is an isomorphism in degrees $* < g-e$, with $e=0$ if $n$ is odd and $e=1$ if $n$ is even.
\end{acorollary}

Our techniques can also be used to prove a second property of Torelli groups. Recall that a based path-connected topological space is \emph{nilpotent} if its fundamental group is nilpotent and acts nilpotently on all higher homotopy groups.

\begin{atheorem}\label{thm:nilp} For $2n \geq 6$, the spaces $B\mr{Tor}(W_g,D^{2n})$ are nilpotent.
\end{atheorem}

The spaces $B\mr{Diff}(W_g, D^{2n})$ classify smooth fibre bundles with fibre $W_g$ containing a trivialised disc bundle, and can be considered as moduli spaces of such manifolds. In Section \ref{sec:tangential} we prove the natural generalisations of Theorem \ref{thm:main}, Corollary \ref{cor:main}, and Theorem \ref{thm:nilp} to moduli spaces of manifolds with certain tangential structures (such as framings).

\subsection*{Acknowledgements}
The authors are grateful to Manuel Krannich for his comments. AK was supported by the Danish National Research Foundation through the Centre for Symmetry and Deformation (DNRF92), the European Research Council (ERC) under the European Union's Horizon 2020 research and innovation programme (grant agreement No.\ 682922), and NSF grant DMS-1803766. ORW was supported by the ERC under the European Union's Horizon 2020 research and innovation programme (grant agreement No.\ 756444), and by a Philip Leverhulme Prize from the Leverhulme Trust.

\tableofcontents

\section{Algebraicity} \label{sec:algebraicity} We start by proving some qualitative results about algebraicity of representations, with the goal of passing such properties through spectral sequences and long exact sequences. A particular role is played by the following well-known groups. Define groups $\mr{O}_{g,g}(\bQ)$, resp.~$\mr{Sp}_{2g}(\bQ)$, as those automorphisms of $\bQ^{2g}$ preserving the symmetric, resp.~anti-symmetric, form with matrix
\[\begin{bsmallmatrix} 0 & \mr{id}_g \\
\mr{id}_g & 0\end{bsmallmatrix}, \qquad \text{resp.} \quad\begin{bsmallmatrix} 0 & \mr{id}_g \\
-\mr{id}_g & 0\end{bsmallmatrix}.\] These are the $\bQ$-points of algebraic groups $\OO_{g,g}$ and $\Sp_{2g}$, respectively. The former has two connected components, and we let $\SO_{g,g}$ denote that containing the identity.

In this section we shall take as given a short exact sequence of groups
\begin{equation} \label{eqn:setup} 1 \lra J \lra \Gamma \lra G \lra 1,\end{equation}
for $G$ an \emph{arithmetic subgroup} of an algebraic group $\GG \in \{\Sp_{2g},\OO_{g,g},\SO_{g,g}\}$, which in this paper shall mean a finite index subgroup of $\GG(\bZ)$ which is Zariski dense in $\GG(\bQ)$ (in contrast with \cite{serrearithmetic}, which imposes no such condition). For $\GG \in \{\Sp_{2g},\SO_{g,g}\}$ any such $G \leq \GG(\bZ)$ of finite index is Zariski dense: $\SO_{g,g}$ for $g \geq 2$ and $\Sp_{2g}$ for $g \geq 1$, are connected semisimple algebraic groups defined over $\bQ$ without compact factors, so by \cite[Theorem 7.8]{BHC} $G$ is a lattice in $\GG(\bR)$, and by the Borel Density Theorem \cite{BorelDensity} $G$ is Zariski dense in $\GG(\bR)$, hence also in $\GG(\bQ)$. For $\GG = \OO_{g,g}$, a subgroup $G \leq \GG(\bZ)$ of finite index is Zariski dense if and only if it is \emph{not} contained in $\mr{SO}_{g,g}(\bZ)$. As $\SO_{1,1}$ fails to be semisimple, we will exclude it from now on:

\begin{convention}\label{conv:eveng2} If $n$ is even, we will assume $g \geq 2$.\end{convention}

\subsection{Some representation theory}

\subsubsection{Algebraic representations} Let $G$ be an arithmetic subgroup of the algebraic group $\GG \in \{\Sp_{2g},\OO_{g,g},\SO_{g,g}\}$ as above. A representation $\phi \colon G \to \mr{GL}_n(\bQ)$ is said to be \emph{algebraic} if it is the restriction of a finite-dimensional representation of the algebraic group $\GG$ in the sense that there is a morphism of algebraic groups $\varphi \colon \GG \to \mathbf{GL}_n$ which on taking $\bQ$-points and restricting to $G$ yields $\phi$. An action of $G$ on an $n$-dimensional $\bQ$-vector space $V$ is then said to be \emph{algebraic} if $V$ admits a basis such that the resulting representation $G \to \mr{GL}_n(\bQ)$ is algebraic. We usually denote a representation $(\phi, V)$ by $V$, leaving the action of $G$ on $V$ implicit. Properties (a), (b) and (c) of algebraic representations listed below are obtained in Section 2.1 of \cite{KR-WTorelli} by combining several results in the literature, and properties (d) and (e) are direct consequences of the definition.

\begin{theorem}\label{thm:representations} The class of algebraic $G$-representations is closed under the following operations:
	\begin{enumerate}[\indent (a)]
		\item subrepresentations,
		\item quotients,
		\item extensions when $g \geq 2$,
		\item duals,
		\item tensor products.
	\end{enumerate}
\end{theorem}

\subsubsection{$gr$-algebraic representations}

The vector spaces that show up in this paper will often \emph{not} be $G$-representations, let alone be algebraic. Instead they will be $\Gamma$-representations with the following property:

\begin{definition}A $\Gamma$-representation $V$ is \emph{$gr$-algebraic} if it admits a finite length filtration 
	\[0 \subset F_0(V) \subset F_1(V) \subset \cdots \subset F_p(V) = V\]
of subrepresentations such that each $F_i(V)/F_{i-1}(V)$ is the restriction to $\Gamma$ of an algebraic $G$-representation.
\end{definition}

\begin{remark}This class of representations has appeared before in the work of Hain (see e.g.~\cite[\S 5]{HainRelMalcev}); they play a role in his theory of relative unipotent completion.
\end{remark}

By definition a $gr$-algebraic $\Gamma$-representation is finite-dimensional. The class of $gr$-algebraic $\Gamma$-representations has all the closure properties of the algebraic $G$-representations themselves.

\begin{lemma}\label{lem:filteredalgebraic} The class of $gr$-algebraic $\Gamma$-representations is closed under the following operations:
	\begin{enumerate}[\indent (a)]
		\item subrepresentations,
		\item quotients,
		\item extensions,
		\item duals,
		\item tensor products.
\end{enumerate}\end{lemma}

\begin{proof}For part (a) and (b), suppose that $V$ is a $gr$-algebraic $\Gamma$-representation with filtration $\{F_i(V)\}_{i=0}^p$, and $W \subset V$ is a subrepresentation. Firstly, $W$ inherits a filtration $F_i(W) \coloneqq W \cap F_i(V)$. Each filtration quotient $F_i(W)/F_{i-1}(W)$ is a subrepresentation of $F_i(V)/F_{i-1}(V)$; this guarantees that the $\Gamma$-action factors over $G$ and then by Theorem \ref{thm:representations} (a) the resulting $G$-representation is algebraic. Secondly, $V/W$ inherits a filtration $F_i(V/W) \coloneqq \mr{im}(F_i(V) \to V \to V/W)$. The filtration quotient $F_i(V/W)/F_{i-1}(V/W)$ is the quotient of $F_i(V)/F_{i-1}(V)$ by $F_i(W)/F_{i-1}(W)$. The $\Gamma$-action thus factors over $G$ and then by Theorem \ref{thm:representations} (b) the resulting $G$-representation is algebraic.

For part (c), suppose that $V$ is a $gr$-algebraic $\Gamma$-representation with filtration $\{F_i(V)\}_{i=0}^q$ and $W$ is a $gr$-algebraic $\Gamma$-representation with filtration $\{F_j(W)\}_{j=0}^p$, and there is a short exact sequence of $\Gamma$-representations
\[0 \lra W \overset{\iota}\lra U \overset{\pi}\lra V \lra 0.\]
Then $U$ admits a filtration $\{F_k(U)\}_{k=0}^{p+q+1}$ by stitching together the filtrations of $W$ and $V$. That is, we set $F_k(U) \coloneqq \iota(F_j(W))$ for $k \leq p$ and $F_k(U) \coloneqq \pi^{-1}(F_{k-p-1}(V))$ for $k>p$. The filtration quotient $F_k(U)/F_{k-1}(U)$ is $F_k(W)/F_{k-1}(W)$ for $k \leq p$ and $F_{k-p-1}(V)/F_{k-p-2}(V)$ for $k>p$. In particular, the $\Gamma$-action factors over $G$ and as such is algebraic.

For part (d), suppose $V$ is a $gr$-algebraic $\Gamma$-representation with filtration $\{F_i(V)\}_{i=0}^p$. Then $V^\vee$ is a filtered $\Gamma$-representation by $F_i(V^\vee) \coloneqq \mr{ann}(F_{p-i}(V))$, with associated graded $F_i(V^\vee)/F_{i-1}(V^\vee) = (F_{p-i}(V)/F_{p-i-1}(V))^\vee$. Thus the $\Gamma$-action factors over $G$, and as such is algebraic.

For part (e), suppose $V$ is a $gr$-algebraic $\Gamma$-representation with filtration $\{F_i(V)\}_{i=0}^p$ and $W$ is a $gr$-algebraic $\Gamma$-representation with filtration $\{F_j(W)\}_{j=0}^q$. Then we may filter $V \otimes W$ by $F_k(V \otimes W) \coloneqq \sum_{i+j = k} F_i(V) \otimes F_j(W)$, so that the filtration quotient is $F_k(V \otimes W)/F_{k-1}(V \otimes W) = \bigoplus_{i+j=k} F_i(V)/F_{i-1}(V) \otimes F_j(W)/F_{j-1}(W)$. This means that the $\Gamma$-action factors over $G$ and as such is algebraic.
\end{proof}

\begin{remark}
Note that in distinction with Theorem \ref{thm:representations}, case (c) does not require the assumption $g \geq 2$. By Convention \ref{conv:eveng2}, this is only relevant when $n$ is odd.
\end{remark}

If $V$ is a $\Gamma$-representation then each cohomology group $H^i(J;V)$ is a $\Gamma$-representation, via the action $\gamma \cdot (j, v) = (\gamma j \gamma^{-1}, \gamma v)$ of $\Gamma$ on the object $(J,V)$ in the category of groups equipped with a module, and functoriality of group cohomology on this category. As inner automorphisms act trivially on group cohomology, this action of $\Gamma$ descends to an action of $G$ on $H^i(J;V)$, see \cite[Corollary 8.2]{Brown}. The following lemma gives a condition under which such a $G$-representation is algebraic.

\begin{lemma}\label{lem:filt-j-cohom} 
Suppose that $g \geq 2$, that each $G$-representation $H^i(J;\bQ)$ is algebraic, and that $V$ is a $gr$-algebraic $\Gamma$-representation. Then each $G$-representation $H^i(J;V)$ is algebraic.
\end{lemma}

\begin{proof}For any filtered $\Gamma$-representation $V$ with filtration $\{F_p(V)\}$, there is a spectral sequence of $G$-representations
	\[E_1^{p,q} = H^{p+q}(J;F_p(V)/F_{p-1}(V)) \Longrightarrow H^{p+q}(J;V).\]
If the action of $\Gamma$ on $F_p(V)/F_{p-1}(V)$ factors over $G$, then $J$ acts trivally on $F_p(V)/F_{p-1}(V)$ and so we can identity $E_1^{p,q}$ with $H^{p+q}(J;\bQ) \otimes F_p(V)/F_{p-1}(V)$ as a $G$-representation.

The hypotheses imply that the $E_1$-page consists of algebraic $G$-representations. Using Theorem \ref{thm:representations} (a) and (b) the $E_\infty$-page consists of algebraic $G$-representations, and using Theorem \ref{thm:representations} (c) and the assumption that $g \geq 2$ so does the abutment.\end{proof}

\subsection{Equivariant Serre classes} \label{sec:equiv-serre} We will use a version of Serre's mod $\cC$ theory for spaces with actions of a group $\Gamma$. 

\begin{definition}Fix a localisation $\bk$ of $\bZ$. A \emph{Serre class} $\cC$ of $\bk[\Gamma]$-modules is a collection of $\bk[\Gamma]$-modules satisfying the following properties:
	\begin{enumerate}[\indent (i)]
		\item \label{enum:equiv-serre-1} For every short exact sequence of $\bk[\Gamma]$-modules
		\[0 \lra A \lra B \lra C \lra 0,\]
		$A,C \in \cC$ if and only if $B \in \cC$.
		\item \label{enum:equiv-serre-2} $A,B \in \cC$ implies $A \otimes_\bk B \in \cC$ and $\mr{Tor}_1^\bk(A,B) \in \cC$.
		\item \label{enum:equiv-serre-3} $A \in \cC$ implies $H_p(K(A,n);\bk) \in \cC$ for all $n \geq 1$ and $p \geq 0$.
	\end{enumerate}
\end{definition}

Property (iii) can be weakened to 
\begin{enumerate}[\indent (i')]
\setcounter{enumi}{2}
\item \label{enum:equiv-serre-3p} $A \in \cC$ implies that $H_p(K(A,1);\bk) \in \cC$ for all $p \geq 0$,
\end{enumerate}
by the following lemma.

\begin{lemma}\label{lem:iii-prime} Given properties \eqref{enum:equiv-serre-1} and \eqref{enum:equiv-serre-2}, property \emph{(\ref{enum:equiv-serre-3p}')} implies property \eqref{enum:equiv-serre-3}.\end{lemma}

\begin{proof}
We will prove that $H_p(K(A,n);\bk) \in \cC$ by double induction over $n \geq 1$ and $p \geq 0$, the initial cases $n=1$ being (\ref{enum:equiv-serre-3p}') and $p=0$ following from $H_0(K(A,n);\bk) = \bk = H_0(K(A,1);\bk)$ which lies in $\cC$ by (\ref{enum:equiv-serre-3p}').
	
	For the case $(n,p)$ with $n > 1$ and $p >0$, we assume the result is proven for $n-1$ and all $p$, and for $n$ and all $p'<p$. We will use the Serre spectral sequence for the homotopy fibration sequence
	\[K(A,n-1) \lra * \lra K(A,n),\]
	which takes the form 
	\[E^2_{s,t} = H_s(K(A,n);H_t(K(A,n-1);\bk)) \Longrightarrow H_{s+t}(\ast;\bk).\]
	By the universal coefficient theorem, the group $E^2_{s,t}$ is an extension of 
	\[H_s(K(A,n);\bk) \otimes_{\bk} H_t(K(A,n-1);\bk) \quad \text{and}\] 
	\[\mr{Tor}_1^\bk(H_{s-1}(K(A,n);\bk), H_t(K(A,n-1);\bk)).\]
	Thus the inductive hypothesis and property \eqref{enum:equiv-serre-2} imply that $E^2_{p',q} \in \cC$ for $p' < p$ and for all $q$. Since $E^r_{p',q}$ is obtained from this by taking subquotients, by property \eqref{enum:equiv-serre-1} it also lies in $\cC$.
	
	We wish to prove that $E^2_{p,0} \in \cC$. Using the exact sequences
	\[0 \lra E^{r+1}_{p,0} \lra E^r_{p,0} \overset{d^r}\lra E^r_{p-r,r-1}\]
	and the fact that $E^r_{p-r,r-1}$ lies in $\cC$ so the image of $d^r$ does too, we see that $E^r_{p,0}$ lies in $\cC$ as long as $E^{r+1}_{p,0}$ does. Since $E^p_{p,0} = E^\infty_{p,0} = 0 \in \cC$, this concludes the proof of the induction step.
\end{proof}

Property \eqref{enum:equiv-serre-1} implies that belonging to the class $\cC$ passes through spectral sequences in the following sense:

\begin{lemma}\label{lem:serre-class-ss} Suppose we have a spectral sequence $\{E^r_{p,q}\}$ of $\bk[\Gamma]$-modules such that 
	\begin{enumerate}[\indent (a)]
		\item each $(p,q)$ has only finitely many non-zero differentials into and out of it,
		\item for each $n \in \bZ$ only finitely many entries $E^\infty_{p,q}$ with $p+q=n$ are non-zero,
		\item for each $(p,q)$ there exists an $r \geq 1$ such that $E^r_{p,q} \in \cC$.
	\end{enumerate}
Then the abutment consists of $\bk[\Gamma]$-modules which lie in $\cC$.\end{lemma}

\begin{proof}We first note that property \eqref{enum:equiv-serre-1} implies that being in $\cC$ is preserved by passing to subquotients. Using (c) it follows that each $E^{r'}_{p,q}$ for $r' > r$ also lies in $\cC$. By property (a), for each $(p,q)$ there exists an $r'$ such that $E^{r'}_{p,q} = E^\infty_{p,q}$. Thus each $E^\infty_{p,q}$ lies in $\cC$. 

Finally, the abutment in total degree $n$ has a filtration with associated graded given by the terms $E^\infty_{p,q}$ with $p+q=n$, which all lie in $\cC$. This filtration is finite by (b), and using property \eqref{enum:equiv-serre-1} a number of times we conclude that the abutment also lies in $\cC$.	
\end{proof}

Recall that a \emph{local system} on a space $X$ is a functor $\Pi(X) \to \cat{Ab}$, where $\Pi(X)$ denotes the fundamental groupoid of $X$. If $\pi \colon E \to B$ is a fibration then there is a local system $\ul{H}_q(\pi;\bk)$ on $B$ given by
\begin{align*} 
\ul{H}_q(\pi;\bk) \colon \Pi(B) &\lra \cat{Mod}_\bk \\
b_0 &\longmapsto H_q(\pi^{-1}(b_0);\bk).
\end{align*}
If the fibres of $\pi$ are typically called $F$, sometimes one writes $\ul{H}_q(F;\bQ)$ for this local system. Dress' construction \cite{Dress} of the homology Serre spectral sequence for a fibration $\pi \colon E \to B$, is given by
\[E^2_{p,q} = H_p(B;\ul{H}_q(\pi;\bk)) \Longrightarrow H_{p+q}(E;\bk)\]
and does not use a choice of basepoint in $B$. Traditionally, for $B$ path-connected one chooses a basepoint $b_0 \in B$ and replaces this local system with the $\bk[\pi_1(B,b_0)]$-module $H_q(\pi^{-1}(b_0);\bk)$. At that point the spectral sequence is only functorial in based maps of fibrations. The advantage of Dress' formulation is that it is natural in \emph{all} maps of fibrations: any commutative diagram
\[\begin{tikzcd} E  \rar{g}  \dar{\pi} & E' \dar{\pi'} \\
B \rar{f} & B' \end{tikzcd}\]
induces a map of Serre spectral sequence given on the $E_2$-page by
\[ H_p(B;\ul{H}_q(\pi;\bk))\overset{\ul{g}_*}\lra H_p(B;f^*\ul{H}_q(\pi';\bk))\overset{f_*}\lra H_p(B';\ul{H}_q(\pi';\bk)) \]
where $\ul{g}_* \colon f^*\ul{H}_q(\pi;\bk) \to \ul{H}_q(\pi';\bk)$ is the map of local systems induced by restricting $g$ to fibres.

We will often want to transfer results about rational homotopy groups to results about rational cohomology, and vice versa, which will only be possible if the action of the fundamental group on higher homotopy groups is under control. 

For each $i \geq 2$, there is a local system
\begin{equation}\label{eq:LocSysPiStar} 
\begin{aligned}
\ul{\pi_i}(X) \colon \Pi(X) &\lra \cat{Ab} \\
x_0 &\longmapsto \pi_i(X,x_0),
\end{aligned}
\end{equation}
and each continuous map $f \colon X \to Y$ induces a natural transformation $\ul{\pi_i}(X) \to f^* \ul{\pi_i}(Y)$. Recall that a path-connected space $X$ is called \emph{simple} if its fundamental group is abelian and acts trivially on higher homotopy groups (this is true for any basepoint if it is true for a single basepoint). If $X$ is simple, not only does \eqref{eq:LocSysPiStar} makes sense for $i=1$ as well, but each local system \eqref{eq:LocSysPiStar} for $i \geq 1$ has the following property: the isomorphism $\pi_i(X,x_0) \to \pi_i(X,x_1)$ is independent of the choice of morphism from $x_0$ to $x_1$ in $\Pi(X)$. Equivalently, $\ul{\pi_i}(X)$ is naturally isomorphic to a constant functor on an abelian group. We can make a canonical choice of such a group by
\[\pi_i(X) \coloneqq \colim \ul{\pi_i}(X).\]
By definition of the colimit $\pi_i(X)$ receives a natural map from $\pi_i(X,x_0)$ for any basepoint $x_0 \in X$, and  this is an isomorphism. If $f \colon X \to Y$ is a map between simple spaces, we therefore obtain a homomorphism $f_* \colon \pi_i(X) \to \pi_i(Y)$. 

Suppose now we have a $\Gamma$-action up to homotopy, given by a homomorphism $\Gamma \to [X,X]$ to the monoid of homotopy classes of maps. This induces an action of $\Gamma$ on the homology groups $H_i(X;\bk)$, and if $X$ is simple it also induces an action of $\Gamma$ on $\pi_i(X) \otimes \bk$. We can ask for either of these to lie in $\cC$. The following will be used in Section \ref{sec:proofs}.

\begin{lemma}\label{lem:filtered-algebraic-homotopy-to-homology}
	Let $X$ be a path-connected simple space with an action of $\Gamma$ up to homotopy. If all $\bk[\Gamma]$-modules $\pi_i(X) \otimes \bk$ lie in $\cC$, then so do all $H_i(X;\bk)$.\end{lemma}

\begin{proof}We use Postnikov towers, which can be produced functorially in all maps by first replacing $X$ by the naturally weakly equivalent space $|\mr{Sing}(X)|$ and letting the $n$th stage $\tau_{\leq n} X$ be the $n$th coskeleton $|\mr{cosk}_n(\mr{Sing}(X))|$, e.g.~ \cite{DwyerKanObstr}. The result is a natural tower of maps 
	\[\cdots \xrightarrow{q_{n+1}} {\tau_{\leq n} X} \xrightarrow{q_{n}} {\tau_{\leq n-1} X} \xrightarrow{q_{n-1}} \cdots \xrightarrow{q_1} \tau_{\leq 0} X\]
	under $X$. The homotopy groups $\pi_i(\tau_{\leq n} X)$ vanish for $i>n$ and the map $X \to \tau_{\leq n} X$ induces an isomorphism on $\pi_i$ for $i \leq n$. Given a point $x_0 \in X$, we obtain a point in $\tau_{n-1} X$ which we shall denote the same. The homotopy fibre of $q_{n}$ over this point in $\tau_{\leq n-1} X$ can be identified with the Eilenberg--Mac~Lane space $K(\pi_{n}(X,x_0),n)$.
	
	We apply the above version of the Serre spectral sequence to the fibration $q_{n}$ for $n \geq 1$, and get a spectral sequence of $\Gamma$-modules
	\[E^2_{p,q} = H_p(\tau_{\leq n-1} X;\ul{H}_q(q_{n};\bk)) \Longrightarrow H_{p+q}(\tau_{\leq n} X;\bk),\]
	where the coefficients are taken in the local system 
	\begin{align*} \ul{H}_q(q_{n};\bk) \colon \Pi(\tau_{\leq n-1} X) &\lra \cat{Mod}_\bk \\
	x_0 &\longmapsto H_q(q_{n}^{-1}(x_0);\bk).
	\end{align*}
	As $q_{n}^{-1}(x_0) \simeq K(\pi_{n}(X,x_0),n)$, homotopy classes of maps between these fibres are determined by their effect on $\pi_n$, so as $X$ is simple the fibre transport map from $q_{n}^{-1}(x_0)$ to $q_{n}^{-1}(x_1)$ is independent of the choice of path from $x_0$ to $x_1$. Thus $\ul{H}_q(q_{n};\bk)$ is canonically isomorphic to the trivial coefficient system on the abelian group
	$$\colim \ul{H}_q(q_{n};\bk) \cong H_q(K(\pi_n(X),n);\bk).$$
Thus the Serre spectral sequence simplifies to 
	\[E^2_{p,q} = H_p(\tau_{\leq n-1} X;H_q(K(\pi_n(X),n);\bk)) \Longrightarrow H_{p+q}(\tau_{\leq n} X;\bk).\]
	The universal coefficient theorem says that the $\Gamma$-module 
	\[H_p(\tau_{\leq n-1} X;H_q(K(\pi_n(X),n);\bk))\] is naturally an extension of the $\Gamma$-modules \[\mr{Tor}_1^\bk(H_{p-1}(\tau_{\leq n-1} X),H_q(K(\pi_n(X),n);\bk))\,\, \text{and}\,\, H_p(\tau_{\leq n-1} X;\bk) \otimes_\bk H_q(K(\pi_n(X),n);\bk).\] Under this identification, the $\Gamma$-action is given by the diagonal action of the evident action on $H_p(\tau_{\leq n-1} X;\bk)$ and the action on $H_q(K(\pi_n(X),n);\bk)$ induced by the action of $\Gamma$ on $\pi_n(X)$.
	
	\vspace{.5 em} 
	
After this preparation we now prove the proposition. We will show by induction over $n$ that each of the homology groups $H_*(\tau_{\leq n} X;\bk)$ lies in $\cC$. The initial case $n=1$ follows from the identification $\tau_{\leq 1} X \simeq K(\pi_1(X),1)$ as a space with $\Gamma$-action. The homology groups $H_q(K(\pi_1(X),1);\bk)$ lie in $\cC$ by property \eqref{enum:equiv-serre-3} of equivariant Serre classes.
	
	For the induction step, we use the above Serre spectral sequence and again use property \eqref{enum:equiv-serre-3} of equivariant Serre classes to see that $H_q(K(\pi_n(X),n);\bk)$ lies in $\cC$. By property \eqref{enum:equiv-serre-2} the $E^2$-page of the Serre spectral sequence also lies in $\cC$, so by Lemma \ref{lem:serre-class-ss} so does its abutment.
\end{proof}

We will only use the full strength of the previous lemma in Section \ref{sec:tangential}. In all other applications $X$ is in fact $1$-connected, and in this case any action of $\Gamma$ on $X$ up to homotopy can be replaced by a based action up to homotopy, i.e.~ a homomorphism $\Gamma \to [X, X]_\ast$ to the monoid of based homotopy classes of based maps. In this case one may ignore the technical discussion about local systems and use the ordinary Serre spectral sequence instead. It also allows one to apply the following converse result, which will be used in Sections \ref{sec:homotopy-cohomology-conf} and \ref{sec:nilpotent}.

\begin{lemma}\label{lem:filtered-algebraic-homology-to-homotopy}
	Let $X$ be a based $1$-connected space with a based action of $\Gamma$ up to homotopy. If all $\Gamma$-modules $H_i(X;\bk)$ lie in $\cC$, so do all $\pi_i(X) \otimes \bk$.\end{lemma}

\begin{proof}We will again use a Postnikov tower, and will prove by induction over $n$ that each of the homotopy groups $\pi_q(\tau_{\leq n} X) \otimes \bk$ lies in $\cC$. Since $\pi_n(X) \cong \pi_n(\tau_{\leq n} X)$ this in particular proves the lemma.
	
For induction step, we suppose that the homotopy groups $\pi_i(\tau_{\leq n-1} X) \otimes \bk$ lie in $\cC$ for all $i$. By the previous lemma it follows that each homology group $H_i(\tau_{\leq n-1} X;\bk)$ lies in $\cC$. Form the long exact sequence of homotopy groups for the pair $(\tau_{\leq n-1} X,X)$ we obtain an isomorphism of abelian groups
\[\pi_{n+1}(\tau_{\leq n-1} X,X) \overset{\cong}{\lra} \pi_n(X),\]
and since $X \to \tau_{\leq n-1} X$ is an $n$-connected map between $1$-connected spaces, the Hurewicz theorem gives us an isomorphism
\[\pi_{n+1}(\tau_{\leq n-1} X,X) \otimes \bk  \overset{\cong}\lra H_{n+1}(\tau_{\leq n-1} X,X;\bk).\]
It thus suffices proves that $H_{n+1}(\tau_{\leq n-1} X,X;\bk)$ lies in $\cC$. This follows from the long exact sequence of a pair
\[\cdots \to H_{n+1}(X;\bk) \to H_{n+1}(\tau_{\leq n-1} X;\bk) \to H_{n+1}(\tau_{\leq n-1} X,X;\bk) \to H_{n}(X;\bk) \to \cdots,\]
by property \eqref{enum:equiv-serre-1} of equivariant Serre casses and the fact that the homology groups $H_*(X;\bk)$ and $H_*(\tau_{\leq n-1} X;\bk)$ lie in $\cC$.
\end{proof}

If one is interested in cohomology instead of homology, one should impose a further axiom on the Serre class $\cC$: that it is closed under $\bk$-linear duals in a derived sense. Specifically, we impose that:
\begin{enumerate}[\indent (i)]
\setcounter{enumi}{3}
\item \label{enum:equiv-serre-4} $A \in \cC$ if and only if $\mr{Hom}_\bk(A,\bk) \in \cC$ and $\mr{Ext}^1_\bk(A,\bk) \in \cC$.
\end{enumerate}
The following will also be used in Sections \ref{sec:homotopy-cohomology-conf} and \ref{sec:nilpotent}.

\begin{lemma}\label{lem:homology-to-cohomology}
Let $\cC$ be an equivariant Serre class which also satisfies property \eqref{enum:equiv-serre-4}. Let $X$ be a space with an action of $\Gamma$ up to homotopy. Then all $\bk[\Gamma]$-modules $H_i(X;\bk)$ lie in $\cC$ if and only if all $H^i(X;\bk)$ do.\end{lemma}

\begin{proof}The universal coefficient theorem gives a natural short exact sequence
	\[0 \lra \mr{Ext}^1_\bk(H_{i-1}(X;\bk),\bk) \lra H^i(X;\bk) \lra \mr{Hom}_\bk(H_i(X;\bk),\bk) \lra 0.\]
The result follows from property \eqref{enum:equiv-serre-1} of equivariant Serre classes along with the additional property \eqref{enum:equiv-serre-4}.\end{proof}

\subsubsection{$gr$-algebraic representations}\label{sec:algebr-cohomology} The first equivariant Serre class is that of $gr$-algebraic representations:

\begin{lemma}\label{lem:gr-algebraic-serre} The $\bQ[\Gamma]$-modules which are $gr$-algebraic representations form an equivariant Serre class satisfying the additional property \eqref{enum:equiv-serre-4}.\end{lemma}

\begin{proof}Property \eqref{enum:equiv-serre-1} is Lemma \ref{lem:filteredalgebraic} (a), (b) and (c), and property \eqref{enum:equiv-serre-4} is (d).  Property \eqref{enum:equiv-serre-2} is Lemma \ref{lem:filteredalgebraic} (e). Property \eqref{enum:equiv-serre-3} follows by dualization, which is allowed by property \eqref{eq:FibSeq}, from the identification of the cohomology $H^*(K(A,n);\bQ)$ with the free graded-commutative algebra on $(A \otimes \bQ)^\vee[n]$, as an algebra in $\bQ[\Gamma]$-modules.
\end{proof}

Thus the hypotheses of Lemmas \ref{lem:serre-class-ss}, \ref{lem:filtered-algebraic-homotopy-to-homology}, and  \ref{lem:filtered-algebraic-homology-to-homotopy} hold when we take $\bk = \bQ$ and $\cC$ to be the $gr$-algebraic representations. Since $gr$-algebraic representations are closed under duals, the hypothesis of Lemma \ref{lem:homology-to-cohomology} holds too.

One way in which a space with a $\Gamma$-action can arise is as the universal cover $Y\langle e \rangle$ of a based path-connected space $Y$ having $\pi_1(Y, y_0) \cong \Gamma$. The Torelli spaces we will be studying arise similarly as covers $Y \langle J \rangle$ corresponding to a normal subgroup $J \lhd \Gamma$ given by the kernel of a homomorphism $\Gamma \to G$, i.e.~the homotopy fibre of the composition $Y \to B\Gamma \to BG$.

\begin{lemma}\label{lem:i-connective} 
	In this situation, suppose that $g \geq 2$, each $G$-representation $H^i(J;\bQ)$ is algebraic, and each $\Gamma$-representation $H^i(Y \langle e \rangle;\bQ)$ is $gr$-algebraic. Then each $G$-representation $H^{i}(Y \langle J \rangle;\bQ)$ is algebraic.
\end{lemma}
\begin{proof}
	Naturality of the Serre spectral sequence implies that
	\[E_2^{p,q} = H^p(J;\ul{H}^q(Y \langle e \rangle)) \Longrightarrow H^{p+q}(Y \langle J \rangle)\]
	is a spectral sequence of $G$-representations. As long as $g \geq 2$, we may apply Lemma \ref{lem:filt-j-cohom} and Theorem \ref{thm:representations} to obtain the conclusion.
\end{proof}

\subsubsection{Nilpotent modules} \label{sec:serre-nilp} The second equivariant Serre class we will consider is that of nilpotent $\Gamma$-modules which are finitely-generated as abelian groups.

\begin{definition}
For a group $\Gamma$, a $\Gamma$-module (i.e.\ $\bZ[\Gamma]$-module) $M$ is said to be \emph{nilpotent} if it has a finite filtration by sub-$\Gamma$-modules whose associated graded is a trivial $\Gamma$-module (i.e.\ has trivial $\Gamma$-action).
\end{definition}

\begin{lemma}\label{lem:NilpIsSerre}
The $\bZ[\Gamma]$-modules which are finitely-generated as abelian groups and nilpotent as $\Gamma$-modules form an equivariant Serre class.
\end{lemma}

\begin{proof}Property \eqref{enum:equiv-serre-1} is straightfoward. 
	
Now consider property \eqref{enum:equiv-serre-2}. Let $0 \subset F_0(A) \subset F_1(A) \subset \cdots \subset F_p(A) = A$ be a finite filtration of $A$ by $\bZ[\Gamma]$-modules such that each $F_i(A)/F_{i-1}(A)$ has a trivial $\Gamma$-action. Choose a resolution $B_* \to B$ by $\bZ[\Gamma]$-modules which are free as abelian groups, and form the filtered chain complex $\{F_i(A) \otimes B_*\}$. This gives a spectral sequence of $\bZ[\Gamma]$-modules 
\[{}^I E^1_{p,q} = \mr{Tor}_{p+q}^\bZ(F_q(A)/F_{q-1}(A),B) \Rightarrow \mr{Tor}^\bZ_{p+q}(A,B)\]
Similarly, letting $0 \subset F_0(B) \subset F_1(B) \subset \cdots \subset F_r(B) = B$ be a finite filtration whose filtration quotients have trivial $\Gamma$-action, and  $Q_* \to F_q(A)/F_{q-1}(A)$ be a resolution by free $\bZ$-modules, gives a spectral sequence of $\bZ[\Gamma]$-modules
\[{}^{II} E^1_{s,t} =  \mr{Tor}^\bZ_{s+t}(F_q(A)/F_{q-1}(A), F_t(B)/F_{t-1}(B)) \Rightarrow \mr{Tor}^\bZ_{s+t}(F_q(A)/F_{q-1}(A),B).\]
Now the groups ${}^{II} E^1_{s,t}$ are finitely-generated and have trivial $\Gamma$-action, so by property \eqref{enum:equiv-serre-1} each ${}^I E^1_{p,q}$ is finitely-generated and has a nilpotent $\Gamma$-action, and by property \eqref{enum:equiv-serre-1} again each $\mr{Tor}^\bZ_{i}(A,B)$ is finitely-generated and has a nilpotent $\Gamma$-action, as required. 
	
For property \eqref{enum:equiv-serre-3}, we instead prove property (\ref{enum:equiv-serre-3p}') and invoke Lemma \ref{lem:iii-prime}. We prove this by induction over the length $m$ of the filtration $0 \subset F_0(A) \subset \cdots F_m(A) = A$. In the initial case $m=0$, the action of $\Gamma$ on $A$ is trivial and the result follows. For the induction step, we apply the Serre spectral sequence to the fibration sequence
\[K(F_{m-1}(A),1) \lra K(A,1) \lra K(A/F_{m-1}(A),1).\]
The $E^2$-page is given by $E^2_{p,q} = H_p(K(A/F_{m-1}(A),1);H_q(K(F_{m-1}(A),1));\bZ)$, which are naturally an extension of the $\Gamma$-modules \[\mr{Tor}_1^\bk(H_{p-1}(K(A/F_{m-1}(A),1);\bZ),H_q(K(F_{m-1}(A),1);\bk))\,\, \text{and}\]
\[H_p(K(A/F_{m-1}(A),1) \otimes_\bk H_q(K(F_{m-1}(A),1);\bZ).\] These are in $\cC$ by property \eqref{enum:equiv-serre-2}, and hence so is the abutment by property \eqref{enum:equiv-serre-1}.
\end{proof}

Thus the hypotheses of Lemmas \ref{lem:serre-class-ss}, \ref{lem:filtered-algebraic-homotopy-to-homology}, and  \ref{lem:filtered-algebraic-homology-to-homotopy} hold when we take $\bk = \bZ$ and $\cC$ to be the nilpotent $\Gamma$-modules which are finitely generated as abelian groups. Furthermore, the hypothesis of Lemma \ref{lem:homology-to-cohomology} holds by the following:

\begin{lemma}\label{lem:NilpIsDualClosed}
Let $\cC$ be the class of $\bZ[\Gamma]$-modules which are finitely-generated as abelian groups and nilpotent as $\Gamma$-modules. Then property \eqref{enum:equiv-serre-4} holds.
\end{lemma}

\begin{proof}We need to verify that $A \in \cC$ if and only if $\mr{Hom}_\bZ(A,\bZ) \in \cC$ and $\mr{Ext}^1_\bZ(A,\bZ) \in \cC$. Suppose first that $A \in \cC$. There is a natural short exact sequence
\[0 \lra \mr{tors}(A) \lra A \lra A/\mr{tors}(A) \lra 0,\]
of $\bZ[\Gamma]$-modules from which one obtains natural isomorphisms
\[\mr{Hom}_\bZ(A,\bZ) \overset{\cong}\lra \mr{Hom}_\bZ(A/\mr{tors}(A),\bZ),\qquad \mr{Ext}^1_\bZ(A,\bZ) \overset{\cong}\lra \mr{Ext}^1_\bZ(\mr{tors}(A),\bZ)\]
Thus it suffices to prove the result separately for the $\bZ[\Gamma]$-modules which are torsion or are free as abelian groups.

First suppose $A$ is torsion. We will use the fact that the functor $\mr{Ext}^1_\bZ(-,\bZ)$ is exact on torsion abelian groups. A finite filtration $0 \subset F_0(A) \subset F_1(A) \subset \cdots \subset F_p(A) = A$ by $\bZ[\Gamma]$-modules such that each $F_i(A)/F_{i-1}(A)$ has a trivial $\Gamma$-action gives a collection of short exact sequences
\[0 \to \mr{Ext}^1_\bZ(F_i(A)/F_{i-1}(A),\bZ) \to \mr{Ext}^1_\bZ(F_i(A),\bZ) \to \mr{Ext}^1_\bZ(F_{i-1}(A),\bZ) \to 0,\]
where the left term has a trivial $\Gamma$-action. By applying property \eqref{enum:equiv-serre-1} we can then inductively prove that $\mr{Ext}^1_\bZ(F_i(A),\bZ)$ is a nilpotent $\Gamma$-module, so $\mr{Ext}^1_\bZ(A,\bZ) = \mr{Ext}^1_\bZ(F_p(A),\bZ)$ is too.

Next suppose $A$ is free and finitely-generated. We will use the fact that $\mr{Hom}_\bZ(-,\bZ)$ is exact on free and finitely-generated abelian groups. Applying $\mr{Hom}_\bZ(-,\bZ)$ to the filtration as before, we get short exact sequences
\[0 \to \mr{Hom}_\bZ(F_i(A)/F_{i-1}(A),\bZ) \to \mr{Hom}_\bZ(F_i(A),\bZ) \to \mr{Hom}_\bZ(F_{i-1}(A),\bZ) \to 0\]
where the left term has a trivial $\Gamma$-action. This again shows inductively that $\mr{Hom}(F_i(A),\bZ)$ is a nilpotent $\Gamma$-module, so $\mr{Hom}_\bZ(A,\bZ) = \mr{Hom}_\bZ(F_p(A),\bZ)$ is too.

For the reverse direction, supposing $\mr{Hom}_\bZ(A,\bZ) \in \cC$ and $\mr{Ext}^1_\bZ(A,\bZ) \in \cC$, it suffices to prove that $\mr{tors}(A)$ and $A/\mr{tors}(A)$ lie in $\cC$. Now use the natural isomorphisms 
\[\mr{tors}(A) \cong \mr{Ext}^1_\bZ(\mr{Ext}^1_\bZ(A,\bZ),\bZ), \qquad A/\mr{tors}(A) \cong \mr{Hom}_\bZ(\mr{Hom}_\bZ(A,\bZ),\bZ),\]
and apply the same arguments as above.\end{proof}

\section{Spaces of self-embeddings} 

Diffeomorphisms of the manifold $W_g := \#_g S^n \times S^n$ relative to an open neighbourhood of a disc $D^{2n}$ are the same as diffeomorphisms of $W_{g,1} = \#_g S^n \times S^n \setminus \mr{int}(D^{2n})$ relative to an open neighbourhood of its boundary $\partial W_{g,1} \cong S^{2n-1}$. We shall study diffeomorphisms of $W_{g,1}$ by thinking of them as self-embeddings.

\subsection{The Weiss fibration sequence} A diffeomorphism of $W_{g,1}$ fixing a neighbourhood of the boundary pointwise is the same as a self-embedding $W_{g,1} \hookrightarrow W_{g,1}$ fixing a neighbourhood of the boundary pointwise. We shall relax this boundary condition, and to do so fix an embedded disc $D^{2n-1} \subset \partial W_{g,1}$. 

\begin{definition}Let $\smash{\mr{Emb}^{\cong}_{\half \partial}}(W_{g,1})$ be the group-like topological monoid of embeddings $W_{g,1} \hookrightarrow W_{g,1}$ which fix a neighbourhood of $D^{2n-1} \subset \partial W_{g,1}$ pointwise and are isotopic through such embeddings to a diffeomorphism of $W_{g,1}$ fixing a neighbourhood of $\partial W_{g,1}$ pointwise. These are topologised in the $C^\infty$-topology.\end{definition}

The \emph{Weiss fibration sequence}, implicit in \cite{weissdalian}, takes the form
\begin{equation}\label{eqn:weiss}
B \mr{Diff}_\partial(D^{2n}) \lra B\mr{Diff}_\partial(W_{g,1}) \lra B\mr{Emb}^{\cong}_{\half \partial}(W_{g,1}).
\end{equation}
In \cite[Theorem 4.17]{kupersdisk} it was proven that this fibration sequence may be delooped once. In other words, \eqref{eqn:weiss} is a principal $B \mr{Diff}_\partial(D^{2n})$-bundle.  Its base is then $B^2\mr{Diff}_\partial(D^{2n})$, obtained by delooping $B\mr{Diff}_\partial(D^{2n})$ using its $E_{2n}$-algebra structure given by boundary connect-sum (see Remark \ref{rem:e2n-algebra} below). 

We will prove algebraicity properties for $B\mr{Diff}_\partial(W_{g,1})$ by first proving them for $B\mr{Emb}^{\cong}_{\half \partial}(W_{g,1})$ and leveraging \eqref{eqn:weiss}. In this section we set up the background necessary to implement this strategy, which will then be done in the remaining sections of this paper.

\begin{remark}\label{rem:e2n-algebra}
This $E_{2n}$-structure is well known: one construction is given in \cite[Lemma 6.1]{Krannich}. We will have use for a version of this $E_{2n}$-algebra structure for moduli spaces of manifolds with tangential structure, so we outline a construction. 

Our preferred model for $B\mr{Diff}_\partial(D^{2n})$ is the moduli space $\cM_\partial(D^{2n})$ of submanifolds of $D^{2n} \times \bR^{\infty}$ which (i) coincide with $D^{2n} \times \{0\}$ on an open neighbourhood of $\partial D^{2n} \times \bR^\infty$ and (ii) are diffeomorphic to $D^{2n}$ rel boundary (see \cite[Section 2]{grwmonoids} for details on the topology). Our preferred model for the $E_{2n}$-operad is the little $2n$-discs operad $\cD_{2n}$, whose space $\cD_n(r)$ of $r$-ary operations consists of ordered $r$-tuples $e = (e_1,\ldots,e_r)$ of embeddings $D^{2n} \hookrightarrow D^{2n}$ with disjoint interior that are each a composition of translation and dilation.

The structure maps of the $\cD_{2n}$-algebra structure on $\cM_\partial(D^{2n})$ are then given as follows: given $(e;X_1,\ldots,X_r) \in \cD_{2n}(r) \times \cM_{\partial}(D^{2n})^r$ we define a new element $X$ of $\cM_\partial(D^{2n})$ by ``inserting $X_i$ on the image of the $i$th disc.'' That is, the submanifold determined by
	\begin{align*}X \cap (e_i(D^{2n}) \times \bR^\infty) &= (e_i \times \mr{id}_{\bR^\infty})(X_i) \qquad \text{for $i =1,\ldots,r$,} \\
	X \cap (D^{2n} \setminus \cup_{i=1}^r e_i(D^{2n}) \times \bR^\infty) & = (D^{2n} \times \{0\}) \cap (D^{2n} \setminus \cup_{i=1}^r e_i(D^{2n}) \times \bR^\infty).\end{align*}
\end{remark}

\subsection{The group of path components}\label{sec:emb-path-components} We start with a computation of the group of path components of $\mr{Emb}^{\cong}_{\half \partial}(W_{g,1})$. This will take the form of a short exact sequence of groups as in \eqref{eqn:setup}. Recall from the introduction the homomorphism
\[\alpha_g \colon \mr{Diff}_\partial(W_{g,1}) \lra G_g \coloneqq \begin{cases} \mr{Sp}_{2g}(\bZ) & \text{if $n$ is odd,}\\
\mr{O}_{g,g}(\bZ) & \text{if $n$ is even,}\end{cases}\]
recording the action of a diffeomorphism on the middle-dimensional homology group $H_n(W_{g,1};\bZ)$. This lands in the symplectic or orthogonal group because the middle homology is equipped with a nondegenerate $(-1)^n$-symmetric intersection form. The intersection form is also equipped with a quadratic refinement, given by counting self-intersections of embedded spheres representing $n$-dimensional homology classes: see e.g.~\cite[Theorem 5.2]{wallscm} or \cite[Section 5]{grwstab1} for details of its construction. This quadratic refinement contains no further information unless $n$ is odd but not $1,3,7$, in which case in terms of the standard hyperbolic basis $e_1, f_1, e_2, f_2, \ldots, e_g, f_g$ of $H_n(W_{g,1};\bZ) \cong \bZ^{2g}$ it is given by 
\begin{equation}\label{eq:StdQuadForm}
\sum_{i=1}^g X_i e_i + Y_i f_i \longmapsto \sum_{i=1}^g X_i Y_i \in \bZ/2.
\end{equation}
 It follows from the work of Kreck \cite{kreckisotopy}, Theorem \ref{thm:kreck} below, that for $2n \geq 6$, the image $G'_g$ of $\alpha_g$ is given by
\[G'_g = \begin{cases} \mr{Sp}_{2g}(\bZ) & \text{if $n$ is 1, 3, or 7,} \\
\mr{Sp}_{2g}^q(\bZ) & \text{if $n$ is odd but not 1, 3, or 7,} \\
\mr{O}_{g,g}(\bZ) & \text{if $n$ is even,}\end{cases}\]
where $\mr{Sp}_{2g}^q(\bZ) \leq \mr{Sp}_{2g}(\bZ)$ is the proper subgroup of symplectic matrices which preserve the quadratic form \eqref{eq:StdQuadForm}. (This was presumably known earlier and, as the referee pointed out, may be deduced by combining \cite[Lemma 10]{wall2} with \cite[Corollaire  1]{cerf} and the fact that $W_g$'s have trivial inertia groups for $n \geq 3$ \cite[Theorem]{WallAction} \cite[Corollary 3.2]{Kosinski}.) We will write $H_n \coloneqq H_n(W_{g,1};\bZ)$. 

The kernel of $\alpha_g$ was also determined by Kreck \cite[Theorem 2]{kreckisotopy}:
\begin{theorem}[Kreck] \label{thm:kreck} For $2n \geq 6$, the mapping class group $\Gamma_g \coloneqq \pi_0(\mr{Diff}_\partial(W_{g,1}))$ is described by the pair of extensions
	\[1 \lra I_{g} \lra \Gamma_g \overset{\alpha_g}\lra G'_g \lra 1,\]
	\[1 \lra \Theta_{2n+1} \lra I_{g} \overset{\chi}\lra \mr{Hom}(H_n,S\pi_n(SO(n))) \lra 1.\]
\end{theorem}

Let us explain the groups and homomorphisms in Theorem \ref{thm:kreck} (the reference for the following discussion is \cite[\S 2]{kreckisotopy}, but see \cite{KrannichMCG} for a similar explanation as well as further information about these extensions). Recall that $W_{g,1}$ is given by the connected sum $\#_g S^n \times S^n \setminus \mr{int}(D^{2n})$, and hence has a standard handle decomposition with a single $0$-handle and $2g$ $n$-handles. Let us introduce terminology for the cores of the $n$-handles. Writing $S^n = D^n/S^{n-1}$, we may assume that the connected sums are performed along discs in $S^n \times S^n$ avoiding the subsets $S^n \times \{0\}$ and $\{0\} \times S^n$. Similarly, we may assume that the disc which is removed from $\#_g S^n \times S^n$ is disjoint from these subsets. We call the $2g$ subsets of $W_{g,1}$ obtained from the subsets $S^n \times \{0\}$ and $\{0\} \times S^n$ the \emph{standard cores}. These come in pairs, which intersect transversally in a single point, and we pick $g$ disjoint embedded arcs connecting these intersection points to $\partial W_{g,1}$. Up to isotopy, we may assume without loss generality that all diffeomorphisms are the identity on a neighbourhood of each of these arcs.

As a consequence of a result of Haefliger \cite{Haefliger}, each element $f \in I_g$ is represented by a diffeomorphism which fixes pointwise the $2g$ standard cores. Let us choose orientation-preserving trivialisations $\tau_i \colon \nu_i \oplus \bR \cong C_i \times \bR^{n+1}$, $1 \leq i \leq 2g$, of the once-stabilised normal bundles of each of these cores. The derivative of $f$ gives $2g$ elements 
\[[\tau_i \circ (Df|_{\nu_i} \oplus \mr{id}) \circ \tau_i^{-1}] \in \pi_n(SO(n+1)),\]
each of which is in the image of $\pi_n(SO(n))$ under stabilisation. Since the cores represent a basis $a_1,\ldots,a_{2g}$ of $H_n$, we can record this data as an element of $\mr{Hom}(H_n,S\pi_n(SO(n)))$, with $S\pi_n(SO(n)) \coloneqq \mr{im}(\pi_n(SO(n)) \to \pi_n(SO(n+1)))$ as given in Table \ref{tab.spino}. 
Kreck shows that this is independent of the choices of trivialisations $\tau_i$ and gives a homomorphism
\[\chi \colon I_{g} \lra \mr{Hom}(H_n,S\pi_n(SO(n))).\]

\begin{table}[h]
	\centering
	\caption{(\cite[p.~644]{kreckisotopy}) The abelian groups $S\pi_n(SO(n))$ for $n \geq 2$, with the exception that $S\pi_6(SO(6)) = 0$. }
	\label{tab.spino}
	\begin{tabular}{lcccccccc}
		\toprule
		$n \pmod 8$ & 0          & 1          & 2 & 3     & 4 & 5 & 6 & 7     \\ \midrule
		$S\pi_n(SO(n))$  & $(\bZ/2)^2$ & $\bZ/2$ & $\bZ/2$ & $\bZ$ & $\bZ/2$ & 0 & $\bZ/2$ & $\bZ$ \\ \bottomrule
	\end{tabular}
\end{table}

An element in the kernel of $\chi$ can be represented by a diffeomorphism which is the identity on an open neighbourhood of each of the standard cores in addition to open neighbourhoods of each of the aforementioned arcs. Thus it is supported in a disc and hence represented by an element of $\pi_0(\mr{Diff}_\partial(D^{2n}))$, which is identified with the group $\Theta_{2n+1}$ of homotopy $(2n+1)$-spheres. Finally, Kreck proved that the homomorphism $\Theta_{2n+1} \to I_g$ is injective.

We use this to study the group $\Lambda_g \coloneqq \pi_0(\mr{Emb}^{\cong}_{\half \partial}(W_{g,1}))$. By definition, the inclusion $\mr{Diff}_\partial(W_{g,1}) \to \smash{\mr{Emb}^{\cong}_{\half \partial}}(W_{g,1})$ is surjective on $\pi_0$ and there is a commutative diagram
\[\begin{tikzcd}\Gamma_g \arrow[two heads]{rr} \arrow[two heads]{rd}[swap]{\alpha_g} & & \Lambda_g \arrow{ld}{\beta_g} \\[-2pt]
& G'_g. & \end{tikzcd}\]
We conclude that the homomorphism $\Lambda_g \to G'_g$ is surjective. By \eqref{eqn:weiss}, the kernel of the homomorphism $\Gamma_g \to \Lambda_g$ is $\Theta_{2n+1}$. Writing $J_g \coloneqq \mr{Hom}(H_n,S\pi_n(SO(n)))$, we conclude that there is a short exact sequence of groups as in \eqref{eqn:setup}
\begin{equation}\label{eqn:lambda-setup} 1 \lra J_g \lra \Lambda_g \lra G'_g \lra 1.\end{equation}

\subsection{Recollection of embedding calculus} \label{sec:embedding-calculus} Embedding calculus is a method to study spaces of embeddings via a tower of approximations, whose layers can be described in homotopy-theoretic terms. Our exposition mostly follows \cite{weisspedrosheaves}, but also refers to the older paper \cite{weissembeddings,weissembeddingserratum}. Though some of the theorems in these papers are stated for manifolds without boundary, they also hold with boundary per \cite[Section 9]{weisspedrosheaves} and \cite[Section 10]{weissembeddings}. Other models for the embedding calculus Taylor tower can be found in \cite{goodwilliekleinweiss,turchincontext,boavidaweisscats}.

\subsubsection{The embedding calculus Taylor tower}\label{sec:EmbTower}

Fix two $d$-dimensional manifolds $M$ and $N$ with the same boundary $\partial M = K = \partial N$. Then the space $\mr{Emb}_K(M,N)$ is the value on $N$ of a continuous functor $\mr{Emb}_K(-,N) \colon \cat{Mfd}_{d,K}^\mr{op} \to \cat{Top}$ (with the weak $C^\infty$-topology, cf.~\cite[\S 1.2]{weisspedrosheaves}). Here $\cat{Mfd}_{d,K}$ is the category enriched in topological spaces with objects given by $d$-dimensional smooth manifolds having boundary identified with $K$, and morphisms given by spaces of embeddings rel boundary, and $\cat{Top}$ is the enriched category of spaces. The category $\cat{Mfd}_{d,K}$ admits a collection of Grothendieck topologies $\cJ_k$ for $k \geq 1$; in $\cJ_k$ a collection $\{U_i\}$ of open subsets of $M$ is a cover if every subset of the interior of $M$ of cardinality $\leq k$ is contained in some $U_i$. 

The $k$th Taylor approximation $T_k(\mr{Emb}_K(-,N))$ is the homotopy sheafification of the presheaf $\mr{Emb}_K(-,N)$ with respect to $\cJ_k$. This means it is up to homotopy the best approximation to $\mr{Emb}_K(M,N)$ built out of the restrictions of embeddings to $\leq k$ discs in $M$, and hence is explicitly given by a right homotopy Kan extension, cf.\ \cite[Definition 4.2]{weisspedrosheaves}: $T_k(\mr{Emb}_K(M,N))$ is the derived mapping space, with respect to the objectwise weak equivalences,
\begin{equation}\label{eqn:derived-map}\bR \mr{map}_{\cat{PSh}(\cat{Disc}_{\leq k,K})}(\mr{Emb}_K(-, M),\mr{Emb}_K(-,N))\end{equation}
between the objects $\mr{Emb}_K(-, M)$ and $\mr{Emb}_K(-,N)$ of the topological category of space-valued presheaves on the full subcategory $\cat{Disc}_{\leq k,K} \subset \cat{Mfd}_{d,K}$ on $d$-dimensional manifolds diffeomorphic rel boundary to a disjoint union of $\leq k$ discs and a collar on $K$. Derived mapping spaces are only well-defined up to homotopy; if we need a point-set model we can pick the Dwyer--Kan mapping spaces \cite[3.1]{DwyerKan}, or pick cofibrant-fibrant replacements in the projective model structure of \cite[Section 3.0.1]{weisspedrosheaves} and take the strict mapping space. By \cite[Corollary 4.7]{DwyerKan} these are equivalent.

Since every $\cJ_k$-cover is a $\cJ_{k-1}$-cover, there is a Taylor tower \cite[Section 3.1]{weisspedrosheaves}
\[\begin{tikzcd} & \vdots \dar\\[-2pt]
\mr{Emb}_K(M,N) \arrow{ru} \rar \arrow{rd} \arrow{rdd} & T_k(\mr{Emb}_K(M,N)) \dar \\[-2pt]
& T_{k-1}(\mr{Emb}_K(M,N)) \dar \\[-2pt]
& \vdots \end{tikzcd}\]
starting at $T_1(\mr{Emb}_K(M,N))$. 

Using Goodwillie's multiple disjunction results \cite{goodwilliethesis}, Goodwillie--Weiss \cite{goodwillieweiss} and Goodwillie--Klein \cite{goodwillieklein} proved that if the handle dimension $h$ of $M$ rel $K$ satisfies $h \leq d-3$, then the map 
\[\mr{Emb}_K(M,N) \lra \underset{k \to \infty}{\mr{holim}}\, T_{k}(\mr{Emb}_K(M,N))\]
is a weak equivalence. More precisely, $\mr{Emb}_K(M,N) \to T_{k} (\mr{Emb}_K(M,N))$ is $(-(d-1)+k(d-2-h))$-connected by \cite[Corollary 2.5]{goodwillieweiss}. Strictly speaking their results apply to an older model \cite[page 84]{weissembeddings} of the embedding calculus tower, but by \cite[Proposition 8.3]{weisspedrosheaves} that model is equivalent to the one described here.

In the case $M=N$, the space $\mr{Emb}_K(M) \coloneqq \mr{Emb}_K(M)$ has a composition law making it into a topological monoid and in particular an $H$-space with strict unit given by the identity map. The functoriality of the above construction makes the Taylor tower into a tower of $H$-spaces with units up to homotopy. More precisely, up to homotopy there is a well-defined composition of derived mapping spaces of objects in $\cat{PSh}(\cat{Disc}_{\leq k,K})$ as in \eqref{eqn:derived-map}; taking $M=N$ gives the multiplication of the $H$-space structure with unit up to homotopy. Furthermore, restriction gives a functor $\cat{PSh}(\cat{Disc}_{\leq k,K}) \to \cat{PSh}(\cat{Disc}_{\leq k-1,K})$; this induces a map $T_k(\mr{Emb}_K(M)) \to T_{k-1}(\mr{Emb}_K(M))$ of $H$-spaces with units up to homotopy. In fact, if one is willing to pick models one can use the Dwyer--Kan mapping spaces of \cite{DwyerKan} to make the tower one of unital topological monoids.

\subsubsection{The layers} \label{sec:layers-bundles}

Fixing an embedding $\iota \colon M \to N$, we obtain a basepoint in each approximation $T_{k}(\mr{Emb}_K(M,N))$, which we call $T_k(\iota)$. We now describe more explicitly the layers
\[L_k(\mr{Emb}_K(M,N)_\iota) \coloneqq \begin{cases} T_1(\mr{Emb}_K(M,N)) & \text{if $k=1$,} \\	\underset{{T_{k-1}(\iota)}}{\mr{hofib}}\,\left[T_{k}(\mr{Emb}_K(M,N)) \to T_{k-1}(\mr{Emb}_K(M,N))\right] & \text{if $k \geq 2$.} \end{cases}\]

The first layer $T_{1}(\mr{Emb}_K(M,N))$ is given by formal embeddings: it is the space $\mr{Bun}_K(TM,TN)$ of bundle maps $TM \to TN$ that are the identity near $K$. This follows from \cite[Proposition 7.6]{weisspedrosheaves}. By definition $\mr{Bun}_K(TM,TN)$ is independent of the basepoint $\iota$, so we shall write $L_1(\mr{Emb}_K(M,N))$ instead of $L_1(\mr{Emb}_K(M,N)_\iota)$.

For $k \geq 2$, the $k$th layer is weakly equivalent to the relative section space of a particular locally trivial fibre bundle built from configuration spaces, which we will now describe. Let $\ul{k} \coloneqq \{1,\ldots,k\}$ and consider the ordered configuration space $\mr{Emb}(\ul{k},N)$ of $k$ points in $N$. For each $I \subset J \subset \ul{k}$ there is a forgetful map $\mr{Emb}(J,N) \to \mr{Emb}(I,N)$. We can combine these into a cubical diagram
\[\{0 < 1\}^{k} \ni I \longmapsto \mr{Emb}(I,N) \in \cat{Top}.\]

We will consider a space of sections of a fibre bundle whose fibres are homeomorphic to total homotopy fibres of this cubical diagram taken over certain basepoints; this is the homotopy fibre of the map
\begin{equation}\label{eqn:tohofib} \mr{Emb}(\ul{k},N) \lra \mr{holim}_{I \subsetneq \ul{k}}\, \mr{Emb}(I,N)\end{equation}
over a certain basepoint. To see it is natural in $N$, we use an explicit model for the total homotopy fibre of this cubical diagram \cite[Proposition 5.5.8]{munsonvolic}. 

\begin{definition}\label{def:tohofib-model}  
The \emph{total homotopy fibre $\mr{tohofib}_{I \subset x_{\ul{k}}} \mr{Emb}(-,N)$} over an ordered configuration $x_{\ul{k}} = (x_1,\ldots,x_k) \in \mr{Emb}(\ul{k},N)$ consists of collections of continuous maps 
	\[\left\{[0,1]^{\ul{k} \setminus I} \overset{f_I}\lra \mr{Emb}(I,N)\right\}_{I \subset \ul{k}}\]
which satisfy
	\begin{enumerate}[\indent (i)]
		\item for each $I \subset J$, extension by zero gives an inclusion $[0,1]^{\ul{k} \setminus J} \hookrightarrow [0,1]^{\ul{k} \setminus I}$, and the following diagram should commute
		\[\begin{tikzcd} {[0,1]}^{\ul{k} \setminus J} \rar[hook] \dar{f_J} & {[0,1]}^{\ul{k} \setminus I} \dar{f_I} \\
		\mr{Emb}(J,N) \rar & \mr{Emb}(I,N), \end{tikzcd}\]
		\item for each $I$, if $d \in [0,1]^{\ul{k} \setminus I}$ has at least one entry equal to $1$ then $f_I(d)(i) = x_i$ for all $i \in \ul{k} \setminus I$.
	\end{enumerate}
	This is topologised as a subspace of the product $\prod_{I \subset \ul{k}} \mr{map}([0,1]^{\ul{k} \setminus I},\mr{Emb}(I,N))$ of mapping spaces with the compact-open topology. A basepoint is given by the collection of maps $\{f_I\}$ satisfying $f_I(d)(i) = x_i$ for all $d \in [0,1]^I$ and $i \in \ul{k} \setminus I$.
\end{definition}

These are the fibres of a space over $\mr{Emb}(\ul{k},N)$:

\begin{definition}
	Let $\tilde{Z}_k(N)$
	be the subspace of those collections of maps
	\[\left\{[0,1]^{\ul{k} \setminus I} \overset{g_I}{\lra} \mr{Emb}(I,N)\right\}_{I \subset \ul{k}}\]
	in $\prod_{I \subset \ul{k}} \mr{map}([0,1]^{\ul{k} \setminus I},\mr{Emb}(I,N))$ that satisfy conditions (i) and (ii) of Definition \ref{def:tohofib-model} for some configuration $x_{\ul{k}} \in \mr{Emb}(\ul{k},N)$. 
	
	There is a map 
	\[\tilde{t}^\mr{id}_k \colon \tilde{Z}_k(N) \lra \mr{Emb}(\ul{k},N)\]
	given by mapping $(g_I)_{I \subset \ul{k}}$ to the unique $x_{\ul{k}}$ of condition (ii).\end{definition}

By the isotopy extension theorem, the map $\tilde{t}^\mr{id}_k$ is a locally trivial fibre bundle, with the total homotopy fibres of Definition \ref{def:tohofib-model} as its fibres. It has a section $\tilde{s}^\mr{id}_k$ given by sending $x_{\ul{k}}$ to the basepoint in the corresponding fibre, so that we have constructed a bundle with section
\[\begin{tikzcd} \tilde{Z}_k(N) \rar[swap]{\tilde{t}^\mr{id}_k} &
	\mr{Emb}(\ul{k},N). \lar[bend right=25,swap]{\tilde{s}^\mr{id}_k} \end{tikzcd}\]

The permutation action of the symmetric group $\fS_k$ on $\mr{Emb}(\ul{k},N)$ naturally extends to an action on the total space $\tilde{Z}_k(N)$ as follows: a permutation of $\ul{k}$ sends the total homotopy fibre of $x_{\ul{k}}$ to that over $x_{\sigma(\ul{k})}$ by simultaneously acting on the indexing sets and the domains of the embeddings. For brevity we shall use the notation
\[C_k(N) \coloneqq \mr{Emb}(\ul{k},N)/\fS_k \qquad \text{and} \qquad {Z}_k(N) \coloneqq \tilde{Z}_k(N)/\fS_k\]
for the quotients. The maps $\tilde{t}_k$ and $\tilde{s}^\mr{id}_k$ are equivariant for these actions, so we may take the quotient by the $\fS_k$-action to get another locally trivial fibre bundle with the same fibres and a section:
\[\begin{tikzcd}{Z}_k(N) \rar[swap]{t^\mr{id}_k} &
	C_k(N). \lar[bend right=25,swap]{s^\mr{id}_k} \end{tikzcd}\]

Finally, pulling this bundle back along the map $\iota_* \colon C_k(M) \to C_k(N)$ induced by an embedding $\iota \colon M \to N$ we get a locally trivial fibre bundle with section
\[\begin{tikzcd}\iota^* {Z}_k(N) \rar[swap]{t^\iota_k} &
	C_k(M) \lar[bend right=25,swap]{s^\iota_k}.\end{tikzcd}\]
We can think of $C_k(M)$ as a subspace of $M^k/\fS_k$, and demand that sections satisfy properties on open neighbourhoods of subsets of $M^k/\fS_k$. Given a section $s \colon C_k(M) \to \iota^* Z_k(N)$, we let $\mr{supp}(s) \subset C_k(M)$ be the closure of the subset where $s \neq s_k^\iota$. If the inverse image of $C_k(M) \setminus \mr{supp}(s)$ in $M^k$ contains an open neighbourhood of 
\[\Delta_\partial \coloneqq \left\{(x_1,\ldots,x_k) \in M^k \middle| \parbox{3.7cm}{\centering $x_i = x_j$ for some $i \neq j$ \\
	or $x_i \in K$ for some $i$}\right\},\]
we say that ``$s$ is equal to $s_k^\iota$ near the fat diagonal or when at least one particle is near $K$". The following appears in \cite[Theorem 9.2]{weissembeddings} (with the necessary modifications for boundary conditions explained in \cite[Section 10]{weissembeddings}, in particular \cite[Example 10.3]{weissembeddings}). A homotopy equivalent way of phrasing the boundary condition is given in Section \ref{sec:SSLayers}.

\begin{proposition}\label{prop:emb-calc-fibers}
The homotopy fibre 
\[\mr{hofib}_{T_{k-1}(\iota)}\left[T_{k}(\mr{Emb}_K(M,N)) \lra T_{k-1}(\mr{Emb}_K(M,N))\right]\]
is weakly equivalent to the space of sections of the pullback bundle $(\iota)^* {Z}_k(N)  \to C_k(M)$, which equal $s_k^\iota$ near the fat diagonal or when at least one particle is near $K$.\end{proposition}

From \eqref{eqn:derived-map}, it is clear that embedding calculus tower is natural, with respect to embeddings relative to the boundary $K$, in the variables $M$ and $N$. As a consequence the layers $L_k(\mr{Emb}_K(M,N))$ are contravariantly functorial in $M$ and covariantly functorial in $N$:
\begin{align*}f \colon M' \hookrightarrow M \qquad &\rightsquigarrow \qquad L_k(\mr{Emb}_K(M',N)_{\iota \circ f}) \xrightarrow{L_k(f)^*} L_k(\mr{Emb}_K(M,N)_\iota), \\
g \colon N \hookrightarrow N' \qquad &\rightsquigarrow \qquad L_k(\mr{Emb}_K(M,N)_{\iota}) \xrightarrow{L_k(g)_*} L_k(\mr{Emb}_K(M,N')_{g \circ \iota}).
\end{align*} 
We claim  that these operations are induced by the naturality of the fibre bundles in $M$ and $N$. Firstly, given an embedding $f \colon M' \hookrightarrow M$ that is the identity near $K$ we can pull back along $f$ to get a map of fibre bundles with section
\begin{equation}\label{eq:Var1stVariable}
\begin{tikzcd}(\iota \circ f)^* {Z}_k(N) \dar[swap]{t_k^{\iota \circ f}} \rar & \iota^* {Z}_k(N) \dar[swap]{t^\iota_k} \\
C_k(M') \rar{f_*} \uar[bend right=25,swap]{s_k^{\iota \circ f}} & C_k(M) \uar[bend right=25,swap]{s^\iota_k}.\end{tikzcd}
\end{equation}
Secondly, given an embedding $g \colon N \hookrightarrow N'$ that is the identity near $K$ there is a natural transformation $g_* \colon \mr{Emb}(-,N) \Rightarrow \mr{Emb}(-,N')$, which induces a map of fibre bundles with section
\begin{equation}\label{eq:Var2ndVariable}
\begin{tikzcd}\iota^* {Z}_k(N) \dar[swap]{t^\iota_k} \rar{g_*} & (g \circ \iota)^*{Z}_k(N') \dar[swap]{t_k^{g \circ \iota}} \\
C_k(M) \rar{\mr{id}} \uar[bend right=25,swap]{s^\iota_k} & C_k(M) \uar[bend right=25,swap]{s_k^{g \circ \iota}}.\end{tikzcd}
\end{equation}

To see that these maps induce $L_k(f)^*$ and $L_k(g)_*$, we need to trace through the arguments in \cite{weissembeddings}. To identify $L_k(g)_*$, we observe that the classification of homogeneous functors as relative sections of a fibration in \cite[Theorem 8.5]{weissembeddings} as well as the identification of that fibration for the layers of a good functor in \cite[Proposition 9.1]{weissembeddings}, naturally depend on the input functor, in this case $\mr{Emb}_K(-,N)$. To identify $L_k(g)_*$, one observes that $M$ enters in these classification results through $|\cI^{(k)}|$. This space and its identification with $\mr{Emb}(\ul{k},M)$ in the proof of \cite[Theorem 9.2]{weissembeddings} is natural in embeddings.

\subsubsection{Applying embedding calculus to $W_{g,1}$}\label{sec:trick} Embedding calculus as explained above does not directly apply to the space of self-embeddings $\mr{Emb}^{\cong}_{\half \partial}(W_{g,1})$ because we are not working relative to the entire boundary. However, this is easily fixed by removing those points which do not lie in the interior of the subset $D^{2n-1} \subset \partial W_{g,1}$. That is, following Weiss \cite{weissdalian} we shall apply embedding calculus to the non-compact manifold 
\[W_{g,1}^\circ \coloneqq W_{g,1} \setminus (\partial W_{g,1} \setminus \mr{int}(D^{2n-1})).\] This manifold is isotopy equivalent to $W_{g,1}$ rel $D^{2n-1}$ and by \cite[Section 3.1]{kupersdisk}, there is a homotopy equivalence
\[\mr{Emb}^{\cong}_{\half \partial}(W_{g,1}) \simeq \mr{Emb}^{\cong}_{\partial}(W_{g,1}^\circ)\] 
as topological monoids (as above, we use the weak $C^\infty$-topology even when the manifolds involved are non-compact). Often it is the case that replacing $W_{g,1}$ by $W_{g,1}^\circ$ does not affect the homotopy type of various mapping and section spaces as long as one works relative to $D^{2n-1}$. (Of course it does when one works with respect to the full boundary $\partial W_{g,1}$ of $W_{g,1}$.) Unless the difference is relevant for the argument, we shall use the notation $W_{g,1}$ for simplicity.

\section{The first layer: bundle maps}\label{sec:first-layer} The goal of this section is to prove Proposition \ref{prop:filteredlayer1}, concerning the rational homotopy groups of the first layer $L_1(\mr{Emb}_{\half \partial}(W_{g,1}))$. This first layer is given by the space $\mr{Bun}_{\half \partial}(TW_{g,1})$ of bundle maps $TW_{g,1} \to TW_{g,1}$ which are the identity near $\half \partial W_{g,1} \subset W_{g,1}$, cf.~Section \ref{sec:layers-bundles}. (Following Section \ref{sec:trick}, we implicitly replace $W_{g,1}$ with $W_{g,1}^\circ$ to apply embedding calculus as described in Section \ref{sec:EmbTower}.)

\subsection{Trivialising the tangent bundle}\label{sec:TrivTang}
The tangent bundle $TW_{g,1}$ is trivialisable via an orientation-preserving isomorphism of vector bundles which we denote $\tau \colon TW_{g,1} \to W_{g,1} \times \bR^{2n}$; this is not unique, even up to homotopy, so we shall keep track of its effect. 

Let us denote an element of $\mr{Bun}_{\half \partial}(TW_{g,1})$ by a pair $(f,\ell)$ of a continuous map $f \colon W_{g,1} \to W_{g,1}$ and a bundle map $\ell \colon TW_{g,1} \to TW_{g,1}$ covering it. The framing $\tau$ allows us to identify such data with maps $W_{g,1} \to W_{g,1} \times \mr{GL}_{2n}(\bR)$, whose first component $W_{g,1} \to W_{g,1}$ is equal to the identity of $W_{g,1}$ near $\half \partial W_{g,1}$  and the second component $W_{g,1} \to \mr{GL}_{2n}(\bR)$ is constant equal to $\mr{id} \in \mr{GL}_{2n}(\bR)$ near $\half \partial W_{g,1}$ . More precisely there are homeomorphisms
\begin{equation}\label{eqn:bun-ident}\begin{aligned} \kappa_\tau \colon \mr{Bun}_{\half \partial}(TW_{g,1}) &\overset{\cong}\lra \mr{map}_{\half \partial}(W_{g,1},W_{g,1} \times \mr{GL}_{2n}(\bR))\\
(f,\ell) &\longmapsto \left( w \mapsto (f(w),\tau_{f(w)} \circ \ell_w \circ \tau^{-1}_w)\right),\\
\left((w,v) \mapsto (f(w),\tau_{f(w)}^{-1} \circ \lambda_w \circ \tau_w(v) \right)&\longmapsfrom (f,\lambda)\end{aligned}\end{equation}
where the subscript on the right hand side indicates that the maps satisfy the boundary conditions indicated above.

To prevent any confusion about the monoid structure used, we shall use $\circledast$ to denote the monoid structure on the right-hand side which corresponds to composition of bundle maps.

\begin{lemma}\label{lem:SDaction}
Under the homeomorphism of \eqref{eqn:bun-ident}, the monoid structure given by composition of bundle maps is described on $\mr{map}_{\half \partial}(W_{g,1},W_{g,1} \times \mr{GL}_{2n}(\bR))$ by
\[(f,\lambda) \circledast (g,\rho) = (f \circ g,(\lambda \circ g) \cdot \rho),\]
where $\cdot$ denotes pointwise multiplication of maps $W_{g,1} \to \mr{GL}_{2n}(\bR)$. 
\end{lemma}
\begin{proof}We have
\begin{align*}
\kappa_\tau^{-1}(f,\lambda) \circ \kappa_\tau^{-1}(g,\rho) &= \left((w,v) \mapsto (f(g(w)),\tau^{-1}_{f(g(w))} \circ \lambda_{g(w)} \circ \tau_{g(w)}(\tau^{-1}_{g(w)} \circ \rho_w \circ \tau_w(v)))\right)\\
&=  \left((w,v) \mapsto (f(g(w)),\tau^{-1}_{f(g(w))} \circ \lambda_{g(w)} \circ \rho_w \circ \tau_w(v))\right)\\
&= \kappa^{-1}_\tau(f \circ g, (\lambda \circ g) \cdot \rho)
\end{align*}\end{proof}

\subsection{The group of homotopy-invertible path components}\label{sec:lbun}
Using this identification we can describe the group 
\[\Upsilon_g \coloneqq \pi_0(\mr{Bun}_{\half \partial}(TW_{g,1}))^{\times}\] of homotopy invertible path components under composition. Recall that $H_n$ is shorthand for $H_n(W_{g,1};\bZ)$.

\begin{lemma}\label{lem:semi-direct-1}
The homeomorphism $\kappa_\tau$ induces an isomorphism of groups
\[(\kappa_\tau)_* \colon \Upsilon_g = \pi_0(\mr{Bun}_{\half \partial}(TW_{g,1}))^{\times} \overset{\cong}\lra \mr{GL}(H_n) \ltimes \mr{Hom}(H_n,\pi_n(SO(2n))),\]
where $\mr{GL}(H_n)$ acts on $\mr{Hom}(H_n,\pi_n(SO(2n)))$ by precomposition.
\end{lemma}

\begin{proof}
The homeomorphism $\kappa_\tau$ of \eqref{eqn:bun-ident} gives a bijection from $\pi_0(\mr{Bun}_{\half \partial}(TW_{g,1}))$
to the set
\[\pi_0(\mr{map}_{\half \partial}(W_{g,1},W_{g,1})) \times \pi_0(\mr{map}_{\half \partial}(W_{g,1},\mr{GL}_{2n}(\bR))).\]
The homotopy equivalence of pairs $(W_{g,1},\half \partial W_{g,1}) \simeq (\vee_{2g} S^{n},\ast)$ gives a weak equivalence of topological monoids 
\[\mr{map}_{\half \partial}(W_{g,1},W_{g,1}) \simeq \mr{map}_\ast(\vee_{2g}S^n,\vee_{2g}S^n).\]
Thus the first term may be identified with the monoid $\mr{End}_\bZ(H_n)$ of $\bZ$-module endomorphisms of $H_n$ under composition, by sending $[f]$ to the endomorphism $f_* \colon H_n \to H_n$.
	
For the second term, we recall from Section \ref{sec:emb-path-components} that the manifold $W_{g,1}$ has $2g$ standard cores $C_i$, which give elements $a_i$ of $H_n$ forming a basis. We may then identify $\pi_0(\mr{map}_{\half \partial}(W_{g,1},\mr{GL}_{2n}(\bR)))$ with the set $\mr{Hom}(H_n,\pi_n(SO(2n)))$, by sending $[\lambda]$ to the homomorphism uniquely determined by
\[a_i \longmapsto \left[C_i \ni x \mapsto  \lambda_x \in \mr{GL}_{2n}(\bR)\right].\]
By Lemma \ref{lem:SDaction}, under these identifications the operation $\circledast$ on the set $\mr{End}_\bZ(H_n) \times \mr{Hom}(H_n,\pi_n(SO(2n)))$ is given by
\begin{equation}\label{eqn:ast-formula} (A,\alpha) \circledast (B,\beta) = (A \circ B,\alpha \circ B+\beta),\end{equation}
where $\circ$ denotes the composition of endomorphisms of $H_n$ or of a homomorphism $H_n \to \pi_n(SO(2n))$ with an endomorphism of $H_n$. An element $(A,\alpha)$ is invertible with respect to $\circledast$ if and only if $A$ is (in which case its inverse is $(A^{-1},-\alpha \circ A^{-1})$). We also read off from \eqref{eqn:ast-formula} that the group of invertible elements is a semi-direct product.
\end{proof}

The expression of $\Upsilon_g$ as a semi-direct product as in Lemma \ref{lem:semi-direct-1} depends on the choice of trivialisation $\tau$, but the ensuing description as an extension
\[1 \lra \mr{Hom}(H_n,\pi_n(SO(2n))) \lra \Upsilon_g \lra \mr{GL}(H_n) \lra 1\]
is independent of this choice. To see this, note that another trivialisation $\tau'$ differs from $\tau$ by an element $\phi \in \mr{map}_{\half \partial}(W_{g,1}, SO(2n))$, giving an element $[\phi] \in \mr{Hom}(H_n,\pi_n(SO(2n)))$, and the isomorphism $(\kappa_{\tau'})_* \circ (\kappa_\tau)_*^{-1}$ is then given by conjugation by $[\phi]$. Such a conjugation is a non-trivial automorphism of the group $\mr{GL}(H_n) \ltimes \mr{Hom}(H_n,\pi_n(SO(2n)))$, but it is trivial on the normal subgroup $\mr{Hom}(H_n,\pi_n(SO(2n)))$ (because this is abelian) and on the quotient $\mr{GL}(H_n)$ (where it becomes conjugation by the identity).

In these terms, let us describe the homomorphism 
\[\Lambda_g= \pi_0(\mr{Emb}^{\cong}_{\half \partial}(W_{g,1})) \lra \Upsilon_g = \pi_0(\mr{Bun}_{\half \partial}(TW_{g,1}))^\times\] induced by taking the derivative of a self-embedding. 

\begin{lemma}\label{lem:semi-direct-2} 
There is a commutative diagram with exact rows
	\[\begin{tikzcd} 1 \rar & J_g = \mr{Hom}(H_n,S\pi_n(SO(n))) \rar \dar & \Lambda_g \rar \dar & G'_g \rar \dar[hook] & 1 \\
	1 \rar & \mr{Hom}(H_n,\pi_n(SO(2n))) \rar & \Upsilon_g \rar & \mr{GL}(H_n) \rar & 1,\end{tikzcd}\]
where the left vertical map is induced by the stabilisation $S\pi_n(SO(n)) \to \pi_n(SO(2n))$.
\end{lemma}
	
\begin{proof}Since both homomorphisms $\Lambda_g \to G'_g$ and $\Upsilon_g \to \mr{GL}(H_n)$ are given by the action on the middle-dimensional homology, we obtain a commutative square
	\[\begin{tikzcd} \Lambda_g \rar \dar & G'_g \dar[hook] \\
	\Upsilon_g \rar & \mr{GL}(H_n).\end{tikzcd}\]
Hence the kernel of the top horizontal map gets sent to the kernel of the bottom horizontal map, and we obtain a commutative diagram as in the statement.

To understand the map between these kernels, we need to recall the surjection $\chi \colon I_g \to \mr{Hom}(H_n,S \pi_n(SO(n)))$ following Theorem \ref{thm:kreck}. Every element of $I_g$ can be represented by a diffeomorphism $f$ which fixes the cores $C_i$ pointwise, and after picking orientation-preserving trivialisations $\tau_i \colon \nu_i \times \bR \smash{\overset{\sim}\to} \bR^{n+1}$ of the once-stabilised normal bundles of the cores, $\chi(f)$ is determined by sending the basis element $a_i \in H_n$ to 
\[[\tau_i \circ (Df|_{\nu_i} \times \bR) \circ \tau_i^{-1}] \in \pi_n(SO(n+1)).\]

The element in $\mr{Hom}(H_n,\pi_n(SO(2n)))$ to which diffeomorphism $f$ is sent is determined by sending the basis elements $a_i \in H_n$ to 
\[[\tau \circ Df|_{C_i} \circ \tau^{-1}] \in \pi_n(SO(2n)).\]
Adding a two-dimensional trivial bundle, we prove the following claim:

\vspace{1.5ex}
\noindent \textbf{Claim.} $[(\tau \circ Df|_{C_i} \circ \tau^{-1}) \times \bR^2] \in \pi_n(SO(2n+2))$ is the $(n+1)$-fold stabilisation of $[\tau_i \circ (Df|_{\nu_i} \times \bR) \circ \tau_i^{-1}] \in \pi_n(SO(n+1))$. 
\vspace{-1ex}

\begin{proof} We will use the following two facts. Firstly, if two maps $S^n \to \mr{GL}_{2n+2}(\bR)$ differ by pointwise conjugation by a map $S^n \to \mr{GL}_{2n+2}(\bR)$ they represent the same element of $\pi_n(SO(2n+2))$ by a Hilton--Eckmann argument. Secondly, if a map $G \colon S^n \to \mr{GL}_{2n+2}(\bR))$ is given by
\begin{equation}\label{eqn:pres-subspace} S^n \ni x \longmapsto \begin{bmatrix} \mr{id}_{n+1} & \gamma(x) \\
0 & g(x) \end{bmatrix} \in \mr{GL}_{2n+2}(\bR),\end{equation}
with $g \colon S^n \to \mr{GL}_{n+1}(\bR)$ and $\gamma \colon S^n \to \mr{Lin}(\bR^{n+1},\bR^{n+1})$, then the homotopy class $[G] \in \pi_n(SO(2n+2))$ is equal to the $(n+1)$-fold stabilisation of the homotopy class $[g] \in \pi_n(SO(n+1))$.

We now fix a trivialisation $\tilde{\tau}_i$ which fits in a commutative diagram of short exact sequences of vector bundles over $C_i$
\[\begin{tikzcd} 1 \rar &  TC_i \times \bR \rar \dar[swap]{\cong} & TW_{g,1} \times \bR^2 \dar{\tilde{\tau}_i}[swap]{\cong} \rar & \nu_i \times \bR \dar{\tau_i}[swap]{\cong} \rar & 1 \\
1 \rar & C_i \times \bR^{n+1} \rar & C_i \times \bR^{2n+2} \rar & C_i \times \bR^{n+1} \rar & 1. \end{tikzcd}\]
 
Using the fact that $f$ fixes $C_i$ pointwise and hence is the identity on $TC_i$, we see that $[(\tilde{\tau}_i \circ Df|_{C_i} \circ \tilde{\tau}_i^{-1}) \times \bR^2] \in \pi_n SO(2n+2)$ is of the form \eqref{eqn:pres-subspace}, with $g = \tau_i \circ (Df|_{\nu_i} \times \bR) \circ \tau_i^{-1}$. This differs by conjugation with $(\tau \times \bR^2) \circ \tilde{\tau}_i^{-1}$ from $[(\tau \circ Df|_{C_i} \circ \tau^{-1}) \times \bR^2]$. The claim then follows from the above two facts.
\end{proof}

Since the stabilisation homomorphism $\pi_n(SO(2n)) \to \pi_n(SO(2n+2))$ is an isomorphism, this implies that $J_g = \mr{Hom}(H_n,S\pi(SO(n))) \to \mr{Hom}(H_n,\pi_n(SO(2n)))$ on generators is induced by the map $S\pi_n(SO(n)) \lra \pi_n(SO(2n))$.\end{proof}

To understand better the homomorphism $\Lambda_g \to \Upsilon_g$, we combine Theorem 1.4 of \cite{levineexotic} with Table \ref{tab.spino} to get:

\begin{lemma}\label{lem:levine} For $n \geq 3$, the stabilisation $S\pi_n(SO(n)) \to \pi_n(SO(2n))$ is: 
	\begin{enumerate}[\indent (i)]
		\item surjective with kernel $\bZ/2$ when $n$ is even, 
		\item an isomorphism when $n$ is odd, $\neq 3,7$, 
		\item injective with cokernel $\bZ/2$ if $n = 3,7$.
	\end{enumerate}
\end{lemma}

\subsection{The higher rational homotopy groups}
We next study the action of the group $\Lambda_g$ on $\pi_i(\mr{Bun}_{\half \partial}(TW_{g,1}), \mr{id}) \otimes \bQ$ via the derivative map $\Lambda_g \to \Upsilon_g$ and conjugation, which is the action which arises from the embedding calculus tower. It will suffice to study the action of $\Upsilon_g$, which by Lemma \ref{lem:semi-direct-2}  fits into a short exact sequence
\[1 \lra \mr{Hom}(H_n,\pi_n(SO(2n))) \lra \Upsilon_g \lra \mr{GL}(H_n) \lra 1.\]

\begin{proposition}\label{prop:filteredlayer1}
For all $i>0$ the $\Upsilon_g$-representation $\pi_i(\mr{Bun}_{\half \partial}(TW_{g,1}),\mr{id}) \otimes \bQ$ is $gr$-algebraic. 
\end{proposition}

\begin{proof}
Let $\cL_*(-)$ denote the functor assigning the free graded Lie algebra to a graded $\bQ$-vector space. For a $\bQ$-vector space $A$, we write $A[n]$ for the graded vector space with $A$ put in degree $n$.  We will prove the more precise statement that this representation is an extension of $\mr{Hom}(H_n,\cL_*(H_n[n-1] \otimes \bQ))[1]$ by $\mr{Hom}(H_n,\pi_{*+n}(SO(2n))\otimes \bQ)$, and on each term the $\Upsilon_g$-action is given by the evident $\mr{GL}(H_n) \cong \mr{GL}_{2g}(\bZ)$-action, which is algebraic.
	
Using the homeomorphism $\kappa_\tau$ of \eqref{eqn:bun-ident}, there is a fibration sequence
\[\mr{map}_{\half \partial}(W_{g,1}, \mr{GL}_{2n}(\bR)) \lra \mr{Bun}_{\half \partial}(TW_{g,1})^\times \lra \mr{map}_{\half \partial}(W_{g,1},W_{g,1})^\times\]
of topological monoids, where the multiplication is given by composition on the base and by pointwise multiplication of maps on the fibre. This is split by a map of monoids, by sending a map $f \colon W_{g,1} \to W_{g,1}$ to the bundle map
\[TW_{g,1} \overset{\tau}\lra W_{g,1} \times \bR^{2n} \xrightarrow{f \times \bR^{2n}} W_{g,1} \times \bR^{2n} \overset{\tau^{-1}}\lra TW_{g,1}.\]
The long exact sequence of (rational) homotopy groups therefore splits into short exact sequences.

The multiplication on $\mr{map}_{\half \partial}(W_{g,1}, \mr{GL}_{2n}(\bR))$ given by pointwise multiplication of maps extends to an $(n+1)$-fold loop space structure, as $W_{g,1} \simeq \vee^{2g} S^n$. Thus the action of $\pi_0$ of this group on its higher homotopy groups is trivial, and so the action of $\Upsilon_g$ on $\pi_i(\mr{map}_{\half \partial}(W_{g,1}, \mr{GL}_{2n}(\bR)),\mr{const}_\mr{id}) = \mr{Hom}(H_n, \pi_{i+n}(\mr{GL}_{2n}(\bR)))$ descends to an action of $\mr{GL}(H_n)$. It follows from Lemma \ref{lem:SDaction} that this action is given by precomposition.

The group $\Upsilon_g$ acts on the rational homotopy groups $\pi_*(\mr{map}_{\half \partial}(W_{g,1},W_{g,1})^\times,\mr{id})$ via the projection map $\Upsilon_g \to \pi_0(\mr{map}_{\half \partial}(W_{g,1},W_{g,1})^\times) = \mr{GL}(H_n)$. Using the homotopy equivalence of pairs $(W_{g,1},\half \partial W_{g,1}) \simeq (\vee_{2g} S^n,\ast)$ we get
\[\pi_i(\mr{map}_{\half \partial}(W_{g,1},W_{g,1}),\mr{id}) \otimes \bQ \cong \mr{Hom}(H_n,\pi_{i+n}(W_{g,1}) \otimes \bQ).\]
The Hilton--Milnor theorem gives an identification of $\Upsilon_g$-representations 
$$\pi_{*+n}(W_{g,1}) \otimes \bQ \cong \cL_*(H_n[n-1] \otimes \bQ)[1],$$
a shift of free graded Lie algebra on $H_n \otimes \bQ$ concentrated in degree $n-1$, and the resulting identification
\[\pi_*(\mr{map}_{\half \partial}(W_{g,1},W_{g,1}),\mr{id}) \otimes \bQ \cong \mr{Hom}(H_n,\cL_*(H_n[n-1] \otimes \bQ))[1]\]
in positive degrees is one of graded $\mr{GL}(H_n)$-representations.
\end{proof}

\section{The higher layers: section spaces}

In this section, our goal is to prove Proposition \ref{prop:filteredhigherlayers}, concerning the rational homotopy groups of the higher layers $L_k(\mr{Emb}_{\half \partial}(W_{g,1})_\mr{id})$ of the embedding calculus tower.

\subsection{Bousfield--Kan homotopy spectral sequences} For every tower of fibrations of based spaces
\[\cdots \lra X_2 \lra X_1 \lra X_0\]
there is an ``extended'' spectral sequence of homotopy groups, as in \cite[IX.\S 4]{bousfieldkan}. Letting $F_n$ denote the homotopy fibre of $X_n \to X_{n-1}$ (with $X_{-1} = \ast$ by convention), we get sequences
\[\cdots \to \pi_2(X_{n-1}) \to \pi_1(F_n) \to \pi_1(X_n) \to \pi_1(X_{n-1}) \overset{\circlearrowright}\to \pi_0(F_n) \to \pi_0(X_n) \to \pi_0(X_{n-1}),\] 
with rightmost three terms pointed sets, next three terms groups, and the remainder abelian groups. The maps into $\pi_0$-terms are maps of pointed sets and the maps into $\pi_i$-terms for $i \geq 2$ are group homomorphisms, with $\pi_2(X_{n-1})$ mapping into the centre of $\pi_1(F_n)$. The sequence is exact in the sense that the kernel of a map is the image of the previous one, with ``kernel" taken to mean the inverse images of the basepoint/identity element. Finally, the decoration on the map $\pi_1(X_{n-1}) \smash{\overset{\circlearrowright}\to} \pi_0(F_n)$ is to indicate that it extends to an action of $\pi_1(X_{n-1})$ on $\pi_0(F_n)$; exactness here is the property that two elements of $\pi_0(F_n)$ are in the same orbit if and only if they map to the same element of $\pi_0(X_n)$. Let us call such a sequence an \emph{extended long exact sequence}.

These extended long exact sequences assemble to an \emph{extended exact couple} (in the sense of \cite[\S.IX.4.1]{bousfieldkan})
\[\begin{tikzcd} D^1 \arrow{rr}{i} & & D^1 \ar{ld}{j} \\
& E^1, \arrow{lu}{k} & \end{tikzcd} \qquad \parbox{5cm}{$D^1_{p,q} = \pi_{q-p}(X_p),$ \\
	$E^1_{p,q} = \pi_{q-p}(F_p)$.}\]
with $i$ of bidegree $(-1,1)$, $j$ of bidegree $(0,-1)$ and $k$ of bidegree $(0,0)$. We can iteratively form a derived couple by taking
\begin{align*}D^{r}_{p,q} &= \mr{im}(\pi_{q-p}(X_{p+r}) \to \pi_{q-p}(X_p)),\\
E^{r}_{p,q} &= \frac{\ker(\pi_{q-p}(F_p) \to \pi_{q-p}(X_p)/D^r_{p,q})}{\text{action of $\ker(\pi_{q-p+1}(X_{p-1}) \to \pi_{q-p+1}(X_{p-r-1}))$}},\end{align*}
the latter reducing to the cokernel of the boundary homomorphism as long as $q-p \geq 1$. For this to make sense, one needs Bousfield and Kan's crucial observation \cite[p.\ 259]{bousfieldkan} that the derived couple of an extended exact couple is again an extended exact couple.

The result is an extended spectral sequence, the \emph{Bousfield--Kan homotopy spectral sequence} 
\[E^1_{p,q} = \pi_{q-p}(F_p) \Longrightarrow \pi_{q-p}(\mr{holim}_p\, X_p).\]
Its properties are explained on \cite[p.\ 260]{bousfieldkan}. The differentials $d^r \colon E^r_{p,q} \to E^r_{p+r,q+r-1}$ are homomorphisms when $q-p \geq 2$, whose images are central when $q-p=2$. In these cases
\[E^{r+1}_{p,q} = \frac{E^r_{p,q} \cap \ker(d_r)}{E^r_{p,q} \cap \mr{im}(d_r)}.\]
When $q-p=1$, the differential $d^r$ extends to an action and
\[E^{r+1}_{p,q} = \frac{E^r_{p,p}}{\text{action of $E^r_{p-r,q-r+1}$}}.\]

Convergence conditions for this spectral sequence are described in \cite[IX.\S 5]{bousfieldkan} and \cite[Section 4]{bousfieldhomotopy}. \emph{Complete convergence} for $q-p \geq 1$ as in \cite[IX.\S 5.3]{bousfieldkan} means that $\pi_{q-p}(\mr{holim}_p\, X_p)$ is the limit of a tower of epimorphisms with kernels given by entries on the $E^\infty$-page. By \cite[Lemma IX.\S 5.4]{bousfieldkan} this holds for $q-p \geq 1$ if $\lim^1_r E^r_{p,q}$ vanishes (see \cite[IX.\S 2]{bousfieldkan} for a discussion of $\smash{\lim^1}$ in the non-abelian setting), a condition similar to that for strong convergence of a half-plane spectral sequence with entering differentials \cite[Theorem 7.1]{Boardman}. Complete convergence holds for example if there are only finitely many non-zero differentials into or out of each entry by \cite[Proposition IX.\S 5.7]{bousfieldkan}, as will be the case in the examples we consider.

One way for a tower of fibrations of based spaces to arise is by filtering the totalisation of a based cosimplicial space $Y_\bullet$: the $p$th space in the tower $\cdots \to \mr{Tot}(Y_\bullet)_1 \to \mr{Tot}(Y_\bullet)_0$ is given by
\[\mr{Tot}(Y_\bullet)_p = \{f_i \colon \Delta^i \to Y_i,\, \text{satisfying simplicial relations}\} \subset \prod_{i=0}^p \mr{map}(\Delta^i,Y_i).\]
As explained in \cite[X.\S 6,7]{bousfieldkan}, in this case
\[E^1_{p,q} = \pi_q(Y_p) \cap \ker(s^0) \cap \cdots \cap \ker(s^{p-1})\]
as long as $q \geq p \geq 0$, and is $0$ otherwise, where the $s^i$ denote the codegeneracy maps. For $q \geq 2$, the differential is induced by the alternating sum of the coface maps $d^i \colon \pi_q(Y_p) \to \pi_q(Y_{p+1})$, and thus the $E^2$-page is given by the cohomology of the cosimplicial abelian groups $E^1_{\ast,q}$. A similar description exists for $q=0,1$. Above we discussed when it completely converges to $\pi_{q-p}(\mr{Tot}(Y_\bullet))$ for $q-p \geq 1$; in particular this happens if there are only finitely many non-zero differentials into or out of each entry.

\subsection{The Federer spectral sequence} 

In \cite{federer}, Federer constructed an extended spectral sequence for the homotopy groups of the space $\mr{map}(X,Y)$ of maps $X \to Y$ based at $f$, in the case that $X$ is a finite CW-complex and $Y$ is a simple path-connected space:
\[E^2_{p,q} = H^p(X;\pi_{q}(Y)) \Longrightarrow \pi_{q-p}(\mr{map}(X,Y),f).\]

We shall need a variation of this spectral sequence for relative section spaces (the Federer spectral sequence is recovered by taking $E = X \times Y$ and $A = \varnothing$).
Such a spectral sequence has appeared before in \cite{schultz}. For a fibration $\pi \colon E \to B$, a subspace $A \subset B$, and a section $\sigma\vert_A \colon A \to E$, let us write $\mr{Sect}(\pi ; \sigma\vert_A)$ for the space of sections of $\pi$ extending $\sigma\vert_A$ in compact-open topology.

We will occasionally want to change the base of the fibration. Suppose $g \colon (B,A) \to (B',A')$ is a map of pairs, that is, a continuous map $g \colon B \to B'$ such that $g(A) \subset A'$. Then we can pull back a fibration $\pi' \colon E' \to B'$ with section $\sigma'\vert_{A'}$ along $g$ to obtain another fibration $g^* \pi' \colon g^* E' \coloneqq B \times_{B'} E' \to B$, with a section $g^* \sigma'|_A \colon A \to E$ induced by universal property of pullbacks:
\[\begin{tikzcd} A \arrow{rd}[description]{g^* \sigma'\vert_{A}} \arrow[bend right=15]{rdd}[swap]{\mr{inc}} \arrow[bend left=15]{rrd}{\sigma' \vert_{A'}\circ g\vert_A } &[10pt] & \\
& g^* E' \rar{\tilde{g}} \dar{g^* \pi'} & E' \dar{\pi'} \\
& B \rar{g} & B'.\end{tikzcd}\]
The following is proven by an elementary argument:

\begin{lemma}If both $A \hookrightarrow B$ and $A' \hookrightarrow B'$ are Hurewicz cofibrations and $g \colon (B,A) \to (B',A')$ is a homotopy equivalence of pairs, then the induced map
\[\mr{Sect}(\pi' ; \sigma'\vert_{A'}) \lra \mr{Sect}(g^* \pi' ; g^* \sigma'\vert_{A})\]
is a weak equivalence.
\end{lemma}

\begin{theorem}\label{thm:Federer}
Suppose $i \colon A \to B$ is a Hurewicz cofibration such that $(B,A)$ is homotopy equivalent to a relative CW pair, and let $\pi \colon E \to B$ be a fibration with $1$-connected fibres and $\sigma \colon B \to E$ be a section. Then there is an extended spectral sequence
\[E^2_{p,q} = \begin{cases}
H^{p}\left(B,A;\underline{\pi_{q}}(\pi)\right) & \text{if $p \geq 0$ and $q-p \geq 0$,}\\
0 & \text{otherwise,}
\end{cases} \quad \Longrightarrow \pi_{q-p}(\mr{Sect}(\pi; \sigma\vert_A),\sigma),\]
with differentials $d^r \colon E^r_{p,q} \to E^r_{p+r,q+r-1}$. Here $\underline{\pi_{q}}(\pi)$ denotes the local system 
\begin{align*} \underline{\pi_{q}}(\pi) \colon \Pi(B) &\lra \text{$\cat{Set}_\ast$, $\cat{Gr}$, or $\cat{Ab}$} \\
b &\longmapsto \pi_{q}(\pi^{-1}(b), \sigma(b)). \end{align*}
If $(B,A)$ is homotopy equivalent to a finite-dimensional relative CW pair, then this spectral sequence converges completely for $q-p \geq 1$.\end{theorem}

The observant reader may have noticed we did not define $H^{p}(B,A;\underline{\pi_{q}}(\pi))$ for $q=0,1$. Since the functors $\underline{\pi_q}(\pi)$ for $q=0,1$ take trivial values by the assumption that the fibres are $1$-connected, we will take these groups to be $0$.

\begin{proof}
Out of the singular simplicial sets $\mr{Sing}(A) \subset \mr{Sing}(B)$, we can form a cosimplicial space with $p$-cosimplices given by the subspace of 
\[\mr{map}(\Delta^p \times \mr{Sing}(B)_p,E) \cong \prod_{\tau \in \mr{Sing}(B)_p} \mr{map}(\Delta^p,E)\]
consisting of collections of maps $\{f_\tau \colon \Delta^p \to E\}_{\tau \in \mr{Sing}(B)_p}$ that satisfy $\pi \circ f_\tau = \tau$, and $f_\tau = \sigma \circ \tau$ when $\tau \in \mr{Sing}(A)$. Its totalisation is the relative space of sections $\mr{Sect}(\epsilon^* \pi;\epsilon^* \sigma|_{|\mr{Sing}(A)|})$ of the pullback of $\pi$ along $\epsilon \colon |\mr{Sing}(B)| \to B$. The assumption that $(B,A)$ is homotopy equivalent to a relative CW pair implies that $(|\mr{Sing}(B)|,|\mr{Sing}(A)|) \to (B,A)$ is a homotopy equivalence of pairs. By the previous lemma, the totalisation of the cosimplicial space is weakly equivalent to $\mr{Sect}(\pi; \sigma\vert_A)$.
	
Form the Bousfield--Kan spectral sequence for its totalisation, whose $E^1_{p,q}$-entry is given for $0 \leq p \leq q$ by the subset of $\prod_{\tau \in \mr{Sing}(B)_p} \pi_q(\mr{Sect}(\tau^*\pi),\tau^*\sigma)$ of those elements that are trivial when $\tau$ is an element of $\mr{Sing}(A)_p$. With the differential given by the alternating sum of the coface maps, we see that the $E^2_{p,q}$-entry will be $H^p(B,A;\smash{\underline{\pi_{q}}}(\pi))$ (as we assumed that the fibre of $\pi$ is $1$-connected, this is just cohomology with a local coefficient system of abelian groups).

If $(B,A)$ is homotopy equivalent to a finite-dimensional relative CW complex of dimension $d$ then $E^2_{p,q}=0$ for $p > d$, so there are finitely-many non-zero differentials out of each entry and hence the spectral sequence converges completely by \cite[Lemma IX.5.7]{bousfieldkan}.
\end{proof}

This spectral sequence is natural in maps of fibrations with section
\[\begin{tikzcd} 
E \arrow{r}{F} \arrow{d}{\pi}& E' \arrow{d}{\pi'} \\
B \arrow[bend left=25]{u}{\sigma} \rar{f} & \arrow[bend left=25]{u}{\sigma'} B', 
\end{tikzcd}\]
in the following sense. As explained above, given such a commutative diagram we can produce a pullback fibration $f^* \pi'$ with section $f^* \sigma'$. Then there is a zigzag of maps of spectral sequences as above, given on the $E^2$-page by 
\[\begin{tikzcd} 
H^p\left(B',A';\underline{\pi_{p+q}}(\pi')\right) \rar{f^*} & H^p\left(B,A;f^*\underline{\pi_{p+q}}(\pi')\right) \ar[equal]{d}\\[-10pt]
H^p\left(B,A;\underline{\pi_{p+q}}(\pi)\right) \rar{F_*}& H^p\left(B,A;\underline{\pi_{p+q}}(f^*\pi')\right) 
\end{tikzcd}\]
and converging to 
\[\pi_{p+q}(\mr{Sect}(\pi'; \sigma'\vert_{A'}),{\sigma'}) \to \pi_{p+q}(\mr{Sect}(f^* \pi'; f^*\sigma'\vert_A), {f^* \sigma'}) \leftarrow \pi_{p+q}(\mr{Sect}(\pi; \sigma\vert_A),{\sigma}).\]

\subsection{Application to layers}\label{sec:SSLayers} 

We shall now apply this to study ${L}_k(\mr{Emb}_{K}(M,N)_{\iota})$ as in Section \ref{sec:layers-bundles}. In terms of the fibration
\[t^\iota_k \colon \iota^* {Z}_k(N) \lra C_k(M),\]
with section $s^\iota_k$, it is given by space of sections which are equal to $s^\iota_k$ near the fat diagonal or when at least one particle is near $K$. The fibre of the map $t_k^\iota$ over a configuration $x_{\ul{k}} = (x_1,\ldots,x_k)$ is $\mr{tohofib}_{I \subset x_{\ul{k}}} \mr{Emb}(-,N)$, which by Theorem B of \cite{goodwillieklein} is $(-(d-3)+k(d-2))$-connected, where $d = \dim(N)$, so at least 1-connected for all $k \geq 2$ as long as $d \geq 2$. 

We shall rephrase the condition on the support of sections in Proposition \ref{prop:emb-calc-fibers} as being relative to a certain subspace $\nabla_\partial$. This subspace will not be unique, but any two choices will admit a common homotopy equivalent refinement. To see the resulting definition coincides with the definition in Proposition \ref{prop:emb-calc-fibers}, we will observe that subspaces of the form $\nabla_\partial$ are cofinal in the poset of open neighbourhoods considered in that proposition.

To define the pair $(C_k(M),\nabla_\partial)$ and understand its homotopy type, we use the Fulton--MacPherson compactifications of configuration spaces, discussed in detail for smooth manifolds without boundary in \cite{sinha} and adapted without much difficulty to smooth manifolds with boundary. Fixing a proper neat embedding $M \hookrightarrow [0,\infty) \times \bR^{N-1}$, recording the location, relative angles between pairs of particles and relative distances between triples of particles gives an inclusion
\[\mr{Emb}(\ul{k},M) \lra ([0,\infty) \times \bR^{N-1})^k \times (S^{N-1})^{{k \choose 2}} \times [0,1]^{{k \choose 3}}.\]
The Fulton--MacPherson compactification $\mr{Emb}[\ul{k},M]$ of $\mr{Emb}(\ul{k},M)$ is the closure of its image. This is a smooth manifold with corners and free $\fS_k$-action. The quotient $C_k[M]$ is the Fulton--MacPherson compactification of $C_k(M)$ and likewise a smooth manifold with corners. 

In particular, $C_k[M]$ is a PL-manifold with boundary and thus there exists a closed collar $C \subset C_k[M]$ of the boundary $\partial C_k[M]$, unique up to isotopy. This is PL-homeomorphic to $\partial C_k[M] \times [0,1]$, with $\partial C_k[M] \times \{0\}$ corresponding to $\partial C_k[M] \subset C_k[M]$, and we take $C' \subset C$ to be the inverse image of $\partial C_k[M] \times [0,\frac{1}{2}]$ under this homeomorphism. We then define
\[\nabla_\partial \coloneqq C' \cap C_k(M),\]
which is PL-homeomorphic to $\partial C_k[M] \times (0,\frac{1}{2}]$. We use both $C$ and $C'$ in Lemma \ref{lem:emb-to-prod}; in the lemma below one may use $C$ instead of $C'$.

\begin{lemma}There is an equivalence of pairs $(C_k[M],\partial C_k[M]) \simeq (C_k(M),\nabla_\partial)$, and $(C_k[M],\partial C_k[M])$ admits the structure of a finite-dimensional CW pair.\end{lemma}

\begin{proof}Picking a PL-triangulation proves the second claim. The first claim follows from the observation that if $N$ is a PL-manifold with boundary $\partial N$ and $C'$ is a collar on $\partial N$, there is a homotopy equivalence of pairs $(N \setminus \partial N,C' \setminus \partial N) \simeq (N,\partial N)$.
\end{proof}

The inclusion $\nabla_\partial \hookrightarrow C_k(M)$ is a Hurewicz cofibration, so Theorem \ref{thm:Federer} applies to the fibration $t^\iota_k \colon \iota^* Z_k(N) \to C_k(M)$, giving a completely convergent extended spectral sequence
\begin{equation}\label{eqn:federer-applied} 
E^2_{p,q} = \begin{cases}
H^{p}\left(C_k(M),\nabla_\partial;\underline{\pi_{q}}(t_k^\iota)\right) & \text{$p \geq 0$, $q{-}p \geq 0$}\\
0 & \text{otherwise,}
\end{cases}\Rightarrow \pi_{q-p}({L}_k(\mr{Emb}_K(M,N)_{\iota})).
\end{equation}

If $M$ is 1-connected of dimension $\geq 3$, then $\mr{Emb}(\underline{k},M)$ is also $1$-connected and the local system $\underline{\pi_{q}}(t_k^\iota)$ may be trivialised when pulled back along the principal $\fS_k$-bundle
\[\pi \colon \mr{Emb}(\underline{k},M) \lra C_k(M).\]
Let $\tilde{\nabla}_\partial$ be the inverse image of $\nabla_\partial$ under the map $\pi$. By transfer we obtain an isomorphism 
\begin{equation}\label{eqn:e2-explicit} E^2_{p,q} \otimes \bQ \cong \left[H^p(\mr{Emb}(\underline{k},M), \tilde{\nabla}_\partial;\bQ) \otimes_\bQ \pi_q(\mr{tohofib}_{I \subset x_{\ul{k}}} \mr{Emb}(I,N))\otimes \bQ\right]^{\fS_k}.\end{equation}
This is an identification only of the rationalised $E^2$-page. We shall not attempt to ``rationalise'' the entire spectral sequence or any subsequent pages, which might not make sense when $p-q=0,1$.

Let us consider the functoriality of the above under embeddings, using the notation of Section \ref{sec:SSLayers}. An embedding $f \colon M' \to M$ induces a map of fibrations \eqref{eq:Var1stVariable}, giving an morphism of spectral sequences which on $E^2$ is given by $H^*(f_*; \mr{id})$ and converges to the map induced by ${L}_k(f)^*$. On the other hand, an embedding $g \colon N \to N'$ induces a map of fibrations \eqref{eq:Var2ndVariable}, giving an morphism of spectral sequences which on $E^2$ is given by $H^*(\mr{id}; (g_*)_*)$ and converges to the map induced by ${L}_k(g)_*$.

\subsection{Homotopy and cohomology groups of configuration spaces} \label{sec:homotopy-cohomology-conf} In this section we obtain qualitative results on configuration spaces, with the goal of eventually applying these to the description \eqref{eqn:e2-explicit} of the $E^2$-page of \eqref{eqn:federer-applied}. In particular, we shall identify the groups
\[H^*(\mr{Emb}(\ul{k},M),\tilde{\nabla}_\partial;\bQ) \quad \text{ and } \quad \pi_*(\mr{tohofib}_{I \subset x_{\ul{k}}}\, \mr{Emb}(I,N)) \otimes \bQ,\]
in a sufficiently natural form that we can understand the $\Gamma_g$-actions for $M = N = W_{g,1}$. For the sake of Section \ref{sec:nilpotent} we will use coefficients in an arbitrary commutative ring $\bk$.

\subsubsection{Relative cohomology} Let $\bar{\Delta}_\partial$ be the closure in $M^k$ of $\tilde{\nabla}_\partial \subset \mr{Emb}(\ul{k},M)$. This contains the closed subset
\[\Delta_\partial = \left\{(x_1,\ldots,x_k) \in M^k \middle| \parbox{3.7cm}{\centering $x_i = x_j$ for some $i \neq j$ \\
	or $x_i \in K$ for some $i$}\right\}.\]

\begin{lemma}\label{lem:emb-to-prod} 
There is an isomorphism
	\[H^*(M^k,\Delta_\partial;\bk) \cong H^*(\mr{Emb}(\ul{k},M),\tilde{\nabla}_\partial;\bk),\]
of $\bk$-modules with commuting $\fS_k$- and $\pi_0(\mr{Diff}_\partial(M))$-actions.\end{lemma}

\begin{proof}
Recall the collar $C \subset C_k[M]$ is homeomorphic to $\partial C_k[M] \times [0,1]$ and under this identification $C'$ is given by $\partial C_k[M] \times [0,\frac{1}{2}]$. We shall use these identifications. The collars $C,C'$ lift to $\fS_k$-equivariant collars $\tilde{C},\tilde{C}'$ of the boundary in $\mr{Emb}[\ul{k},M]$. Consider now the homotopy
	\begin{align*} H_t \colon \mr{Emb}[\ul{k},M] &\lra \mr{Emb}[\ul{k},M] \\
	\vec{x} &\longmapsto \begin{cases} \left(\vec{y},\frac{\min(0,s-t/2)}{(1-t/2)}\right) & \text{if $\vec{x} = (\vec{y},s) \in \tilde{C} \cong \partial \mr{Emb}[\ul{k},M] \times [0,1]$} \\
	\vec{x} & \text{otherwise},\end{cases}\end{align*}
which pushes points near the boundary of $\mr{Emb}[\ul{k},M]$ into the boundary using the collar. By definition of $\mr{Emb}[\ul{k},M]$ there is a ``macroscopic position'' map $\mu \colon \mr{Emb}[\ul{k},M] \to M^k$, which records the underlying location of the configuration points. Observe that for $\vec{x} \in M^k$, the element $\mu \circ H_t \colon \mr{Emb}[\ul{k},M] \to M^k$ takes the same values on all elements of $\mu^{-1}(\vec{x})$. Thus $H_1$ extends to a map $\overline{H_1} \colon M^k \to M^k$, and $H_t$ extends to a homotopy $\overline{H_t}$ of such maps. These exhibit $\overline{H_1}$ as a homotopy inverse to the inclusion of pairs $i : (M^k,\Delta_\partial) \to (M^k,\bar{\Delta}_\partial)$.

Now $(\mr{Emb}(\ul{k},M) \setminus \mr{int}(\tilde{C}'),(\mr{Emb}(\ul{k},M) \setminus \mr{int}(\tilde{C}')) \cap \tilde{C}) \to (\mr{Emb}(\ul{k},M),\tilde{\nabla}_\partial)$ is a homotopy equivalence of pairs, and $\mu$ induces a map of pairs
\[\left(\mr{Emb}(\ul{k},M) \setminus \mr{int}(\tilde{C}'),(\mr{Emb}(\ul{k},M) \setminus \mr{int}(\tilde{C}')) \cap \tilde{C}\right)  \lra \left(M^k,\bar{\Delta}_\partial\right).\]
The induced maps on cohomology are isomorphisms:
\begin{align*}H^*(\mr{Emb}(\ul{k},M),\tilde{\nabla}_\partial;\bk) &\overset{\cong}{\longleftarrow} H^*(\mr{Emb}(\ul{k},M) \setminus \mr{int}(\tilde{C}'),(\mr{Emb}(\ul{k},M) \setminus \mr{int}(\tilde{C}')) \cap \tilde{C};\bk) \\
&\overset{\cong}{\lra} H^*(M^k,\bar{\Delta}_\partial;\bk) \\
&\overset{\cong}{\longleftarrow} H^*(M^k,\Delta_\partial;\bk),\end{align*}
with second an isomorphism by excision. Since all maps are equivariant for the actions of $\fS_k$ and diffeomorphisms supported away from the boundary, these are isomorphisms of $\bk$-modules with commuting $\fS_k$- and $\pi_0(\mr{Diff}_\partial(M))$-actions\end{proof}

Let us from now on suppose that the cohomology of $M$ consists of free $\bk$-modules, as it is for $W_{g,1} = \# S^n \times S^n \setminus \mr{int}(D^{2n})$. Under this assumption, the K\"unneth theorem gives an isomorphism $H^*(M^k;\bk) \cong H^*(M;\bk)^{\otimes k}$. Since the action of the mapping class group $\pi_0(\mr{Diff}_\partial(M))$ on $H^*(M^k;\bk) \cong H^*(M;\bk)^{\otimes k}$ is evident, it will suffice to understand the groups $H^*(\Delta_\partial;\bk)$ with their action of the mapping class group. To do so, we will express $\Delta_\partial$ as a homotopy colimit. 

Let $\Pi^\ast(\ul{k})$ be the poset of non-discrete partitions $\omega$ of $\{1,\ldots,k,\ast\}$, ordered by refinement, and consider the functor $\Delta_\delta : \Pi^\ast(\ul{k}) \to \mathsf{Top}$ given by
\[\Delta_\delta(\omega) = \Delta_\delta^\omega \coloneqq \left\{(x_1,\ldots,x_k) \in M^k \middle| \parbox{6.3cm}{\centering $x_i = x_j$ if $i,j$ in the same element of $\omega$ \\
	$x_i \in K$ if $i$ in the same element of $\omega$ as $\ast$}\right\}.\]

\begin{example}$\Pi^\ast(\ul{2})$ contains four partitions: $\{1,2\}\{\ast\}$, $\{1,\ast\}\{2\}$, $\{2,\ast\}\{1\}$, and $\{1,2,\ast\}$, the first three partitions being incomparable and all larger than the last partition. If $K$ is contractible, then the value of $\Delta_\partial$ on the first three partitions is homotopy equivalent to $M$ and its value on the last partition is contractible.
\end{example}

\begin{lemma}
The inclusions $\Delta_\delta^\omega \hookrightarrow \Delta_\partial$ assemble to a homeomorphism
\[\mr{colim}_{\Pi^\ast(\ul{k})}\,\Delta_\delta \overset{\cong}{\lra} \Delta_\partial\]
and the canonical map $\mr{hocolim}_{\Pi^\ast(\ul{k})}\,\Delta_\delta \to \mr{colim}_{\Pi^\ast(\ul{k})}\,\Delta_\delta$ is a weak equivalence.
\end{lemma}

\begin{proof}
The first identification is obvious. 

Before discussing the weak equivalence, let us discuss the Reedy model structure, which may be used to efficiently present the homotopy colimit. A \emph{Reedy category} is a category $\cat{C}$ with two subcategories $\cat{C}_+$ and $\cat{C}_-$ containing all objects, and a function $\deg \colon \mr{ob}(\cat{C}) \to \bN$ such that (i) non-identity morphisms in $\cat{C}_+$ strictly increase $\deg$, (i') non-identity morphisms in $\cat{C}_-$ strictly decrease $\deg$, and (ii) every morphism in $\cat{C}$ factors uniquely as a composition as a morphism in $\cat{C}_-$ followed by a morphism in $\cat{C}_+$. As every non-identity morphism in $\Pi^\ast(\ul{k})$ increases the number of parts of a partition, we can make $\Pi^\ast(\ul{k})$ into a Reedy category by taking $\deg$ to be the number of parts, $\cat{C}^+$ all morphisms and $\cat{C}_-$ to be the identity morphisms, cf.~\cite[Example 2.3]{RiehlVerity}.

Let us use the Str\o m model structure on the category $\cat{Top}$ of topological spaces, then there is a Reedy model structure on the category $\cat{Top}^\cat{C}$ of functors $\cat{C} \to \cat{Top}$ \cite[Theorem 4.18]{RiehlVerity}. We will need two facts about this model structure: (i) $\mr{colim} \colon \cat{Top}^\cat{C} \to \cat{Top}$ is a left Quillen functor \cite[\S 8]{RiehlVerity}, (ii) a diagram $F \in \cat{Top}^\cat{C}$ is cofibrant if for each object $c \in \cat{C}$ the \emph{latching maps} \cite[\S 3]{RiehlVerity}
\[\underset{c \overset{+}{\to} c'}{\mr{colim}}\, F(c') \lra F(c)\]
are closed Hurewicz cofibrations. Here the colimit is taken over the full subcategory of the comma-category in $\cat{C}_+/c$ not containing the identity.

Now to prove the weak equivalence it suffices to prove that $\omega \mapsto \Delta_\delta^\omega$ is cofibrant in the Reedy model structure, as then we may use its colimit to compute the homotopy colimit \cite[Definition 8.1]{RiehlVerity}. This amounts to proving that the latching maps are cofibrations. These latching maps identify the inclusion into $\Delta_\delta^\omega \subset M^k$ of all $\Delta_\delta^{\omega'} \subset M^{k'}$ such that $\omega' \prec \omega$. The result follows from the union theorem for closed Hurewicz cofibrations \cite{Lillig}.
\end{proof}

Thus there is a Bousfield--Kan spectral sequence (a special case of \cite[XII.\S 5]{bousfieldkan})
\begin{equation}\label{eqn:bk-cohomology} E^1_{p,q} = \bigoplus_{\omega_0 \prec \ldots \prec \omega_p \in N_p(\Pi^\ast(\ul{k}))} H^q(\Delta^{\omega_0}_\delta;\bk) \Longrightarrow H^{p+q}(\Delta_\partial;\bk).\end{equation}
As the above constructions are equivariant for diffeomorphisms of $M$ fixing $K$ pointwise, we can read those features of the action on the abutment which are relevant for this paper from the action on the $E^1$-page. 

Let us assume that $K$ is contractible. Then the action of the mapping class group on $H^q(\Delta_\delta^\omega;\bk)$ again factors over $\mr{Aut}(H^*(M;\bk))$ as $\Delta^\omega_\delta$ is equivariantly homotopy equivalent to $M^{k'}$ for some $k' < k$. Let us now specialise to $M = W_{g,1}$, $K = \half \partial W_{g,1} = D^{2n-1}$, and take $\bk = \bQ$. Recall that the mapping class group $\pi_0(\mr{Diff}_\partial(W_{g,1}))$ was denoted $\Gamma_g$.

\begin{proposition}\label{prop:conf-rational-cohomology-alg}For each $k \geq 1$ the $\Gamma_g$-representation $H^*(W_{g,1}^k,\Delta_\partial;\bQ)$ is $gr$-algebraic.\end{proposition}

\begin{proof}Using Lemma \ref{lem:filteredalgebraic} (a), (b), (c), and the long exact sequences of $\Gamma_g$-representations
	\[\cdots \lra H^q(W_{g,1}^k,\Delta_\partial;\bQ) \lra H^q(W_{g,1}^k;\bQ) \lra H^q(\Delta_\partial;\bQ) \lra \cdots,\]
we see that $H^*(W_{g,1}^k,\Delta_\partial;\bQ)$ is $gr$-algebraic if both $H^*(W_{g,1}^k;\bQ)$ and  $H^*(\Delta_\partial;\bQ)$ are. As pointed out above, the K\"unneth isomorphism provides an isomorphism (of $\Gamma_g$-representations) $H^*(W_{g,1}^k;\bQ) \cong H^*(W_{g,1};\bQ)^{\otimes k}$. This shows that the action of $\Gamma_g$ factors over $G'_g$ and by Lemma \ref{lem:filteredalgebraic} (e) it is $gr$-algebraic as such.

The Bousfield--Kan spectral sequence \eqref{eqn:bk-cohomology} computing $H^*(\Delta_\partial;\bQ)$ as a $\Gamma_g$-representation, has $E^1$-page given by a direct sum of tensor products of $H^*(W_{g,1};\bQ)$ as $\Gamma_g$-representations. By the same reasoning as above its $E^1$-page is $gr$-algebraic and hence so is the abutment by Lemma \ref{lem:serre-class-ss}.
\end{proof}

\subsubsection{Rational homotopy} Next we study the rational homotopy groups of the ordered configuration spaces $\mr{Emb}(\ul{k},N)$, which we base at a configuration near the boundary $K$ so that we may assume that the $\mr{Diff}_\partial(N)$-action fixes the basepoint.

We shall use the Totaro spectral sequence \cite{totaro}. This is derived from the Leray spectral sequence for the inclusion $\mr{Emb}(\ul{k},N) \hookrightarrow N^k$, and when $\bk$ is a field it has $E_2$-page given by the bigraded $\bk$-algebra $H^*(N^k;\bk)[G_{ab}]$ with $H^q(N^k;\bk)$ in bidegree $(q,0)$ and generators $G_{ab}$ of bidegree $(0,d-1)$ for $1 \leq a, b\leq k$, $a \neq b$, subject to the following relations: 
\begin{enumerate}[\indent (i)]
\item $G_{ab} = (-1)^d G_{ba}$,
\item $G_{ab}^2 = 0$, 
\item $G_{ab}G_{bc}+G_{bc}G_{ca}+G_{ca}G_{ab} = 0$ for $a,b,c$ distinct,
\item $p^*_a(x)G_{ab} = p^*_b(x)G_{ab}$ where $p_a \colon N^k \to N$ denotes the $a$th projection map and $x \in H^*(N;\bk)$.
\end{enumerate}
Following Totaro's construction one finds that this description of the $E_2$-page also holds with $\bk$-coefficients if $\bk$ is a localisation of the integers and $H^*(N;\bk)$ consists of free $\bk$-modules. It converges to $H^{p+q}(\mr{Emb}(\ul{k},N);\bk)$ and is natural in embeddings. Let us now specialise to $N = W_{g,1}$ and $\bk = \bQ$.

\begin{lemma}
Each $\Gamma_g$-representation $H^i(\mr{Emb}(\ul{k},W_{g,1});\bQ)$ is $gr$-algebraic.
\end{lemma}

\begin{proof}In the Totaro spectral sequence the $\Gamma_g$-action on the $E_2$-page factors over $G'_g$ and as such is an algebraic representation, so in particular $gr$-algebraic. By Lemma \ref{lem:serre-class-ss} the abutment is also $gr$-algebraic.
\end{proof}

As $\mr{Emb}(\ul{k},W_{g,1})$ is $1$-connected for $n \geq 2$, we can convert this to a statement about rational homotopy groups using Lemmas \ref{lem:filtered-algebraic-homology-to-homotopy}, \ref{lem:homology-to-cohomology}, and \ref{lem:gr-algebraic-serre}.

\begin{corollary}\label{cor:HtyEmbIsAlg}
Each $\Gamma_g$-representation $\pi_i(\mr{Emb}(\ul{k},W_{g,1})) \otimes \bQ$ is $gr$-algebraic.
\end{corollary}

The total homotopy fibre $\mr{tohofib}_{I \subset x_{\ul{k}}} \mr{Emb}(I,W_{g,1})$ taken at a configuration $x_{\ul{k}}$ near the boundary admits a basepoint-preserving $\mr{Diff}_\partial(W_{g,1})$-action, because the entire cubical diagram $I \mapsto \mr{Emb}(I,W_{g,1})$ of based spaces does.

\begin{proposition}\label{prop:conf-rational-homotopy-alg} 
Each $\Gamma_g$-representation $\pi_i(\mr{tohofib}_{I \subset x_{\ul{k}}} \mr{Emb}(I,W_{g,1})) \otimes \bQ$ is $gr$-algebraic.
\end{proposition}

\begin{proof}As the total homotopy fibre $\mr{tohofib}_{I \subset x_{\ul{k}}} \mr{Emb}(I,W_{g,1})$ is the homotopy fibre of the map \eqref{eqn:tohofib} over a certain basepoint, there is a natural based map
	\[\mr{tohofib}_{I \subset x_{\ul{k}}} \mr{Emb}(I,W_{g,1}) \lra \mr{Emb}(\ul{k},W_{g,1}),\]
which is equivariant for the basepoint-preserving action of $\mr{Diff}_\partial(W_{g,1})$.
	
Because all maps in the diagram admit sections up to homotopy by adding particles near the boundary, the induced map on homotopy groups
	\[\pi_*(\mr{tohofib}_{I \subset x_{\ul{k}}} \mr{Emb}(I,W_{g,1})) \lra \pi_*(\mr{Emb}(\ul{k},W_{g,1}))\]
	is the inclusion of a summand. Since this inclusion is $\Gamma_g$-equivariant and the right hand side after tensoring with $\bQ$ is $gr$-algebraic by Corollary \ref{cor:HtyEmbIsAlg}, so is the left hand side after tensoring with $\bQ$ by Lemma \ref{lem:filteredalgebraic} (a).
\end{proof}

\subsection{Algebraicity of rational homotopy groups} 
We now put together the results of this section to prove the remaining proposition concerning the layers of the embedding calculus tower. It concerns the $\Gamma_g$-action on $L_k(\mr{Emb}_{\half \partial}(W_{g,1})_{\mr{id}})$, which is induced by conjugation with a diffeomorphism of $W_{g,1}$. We know that this action factors through $\Lambda_g$, but it is convenient to remember the geometric origin of this action. 

As mentioned in Section \ref{sec:EmbTower} the layer ${L}_k(\mr{Emb}_{\half \partial}(W_{g,1})_\mr{id})$ is an $H$-space, so its positive-degree homotopy groups are abelian and so can be rationalised.

\begin{proposition}\label{prop:filteredhigherlayers} 
For $i \geq 1$ and $k \geq 2$ the $\Gamma_g$-representations 
$$\pi_i(L_k(\mr{Emb}_{\half \partial}(W_{g,1})_{\mr{id}})) \otimes \bQ$$
are $gr$-algebraic.
\end{proposition}

\begin{proof}
We apply the spectral sequence \eqref{eqn:federer-applied} with $M=N=W_{g,1}$ and $K = \half \partial W_{g,1}$. This spectral sequence converges completely as it has $E^2_{p,q}=0$ for all large enough $p$, so for $i \geq 1$ the abelian group $\pi_{i}({L}_k(\mr{Emb}_{\half \partial}(W_{g,1})_\mr{id}))$ has a finite filtration with $p$th filtration quotient a subquotient of $E^2_{p,i+p}$. This is natural for the $\Gamma_g$-action, so $\pi_{i}({L}_k(\mr{Emb}_{\half \partial}(W_{g,1})_\mr{id})) \otimes \bQ$ has a finite filtration by $\Gamma_g$-subrepresentations with $p$th filtration quotient a subquotient of $E^2_{p,i+p} \otimes \bQ$. By Lemma \ref{lem:filteredalgebraic} it suffices to show that each $E^2_{p,i+p} \otimes \bQ$ is $gr$-algebraic

As described in Section \ref{sec:SSLayers} we have
\[E^2_{p,i+p} \otimes \bQ = \left[H^{p}(\mr{Emb}(\ul{k},W_{g,1}),\tilde{\nabla}_\partial;\bQ) \otimes_\bQ (\pi_{i+p}(\mr{tohofib}_{I \subset x_{\ul{k}}} \mr{Emb}(I,W_{g,1})) \otimes \bQ)\right]^{\fS_k}.\]
A diffeomorphism $\phi$ of $W_{g,1}$ acts on this by $[(\phi^{-1}_*)^* \otimes \phi_*]^{\fS_k}$. By parts (a) and (e) of Lemma \ref{lem:filteredalgebraic} it suffices to show that the $\Gamma_g$-representations $H^{p}(\mr{Emb}(\ul{k},W_{g,1}),\nabla_\partial;\bQ)$ and $\pi_{i+p}(\mr{tohofib}_{I \subset x_{\ul{k}}} \mr{Emb}(I,W_{g,1})) \otimes \bQ$ are both $gr$-algebraic, which they are by Propositions \ref{prop:conf-rational-cohomology-alg} and \ref{prop:conf-rational-homotopy-alg} respectively.\end{proof}

\section{Proof of Theorem \ref{thm:main} and Corollary \ref{cor:main}} \label{sec:proofs}

To obtain a structural understanding of the cohomology of the Torelli space of $W_{g,1}$, we will consider the Torelli analogue of the space of self-embeddings. We proved in Section \ref{sec:emb-path-components} that the action on homology gives a surjective homomorphism 
\[\beta_g \colon \Lambda_g  = \pi_0(\mr{Emb}^{\cong}_{\half \partial}(W_{g,1})) \lra G'_g = \begin{cases} \mr{Sp}_{2g}(\bZ) & \text{if $n$ is 1, 3, or 7,} \\
\mr{Sp}_{2g}^q(\bZ) & \text{if $n$ is odd but not 1, 3, or 7,} \\
\mr{O}_{g,g}(\bZ) & \text{if $n$ is even.}\end{cases}\]
As before, we shall assume that $2n \geq 6$.

\begin{definition}
The \emph{embedding Torelli group} $\mr{TorEmb}^{\cong}_{\half \partial}(W_{g,1})$ is the group-like submonoid of $\mr{Emb}^{\cong}_{\half \partial}(W_{g,1})$ consisting of those path-components in the kernel of $\beta_g$.
\end{definition}

The classifying space of this monoid fits into a fibration sequence
\[B\mr{TorEmb}^{\cong}_{\half \partial}(W_{g,1}) \lra B\mr{Emb}^{\cong}_{\half \partial}(W_{g,1}) \lra BG'_g,\]
and so $B\mr{TorEmb}^{\cong}_{\half \partial}(W_{g,1})$ has an action of $G'_g$ up to homotopy. To prove that the cohomology of this space consists of algebraic $G'_g$-representations, we will use that we already understand the group $\Lambda_g = \pi_0(\mr{Emb}_{\half \partial}^{\cong}(W_{g,1}))$ and focus our attention on the identity component $\smash{\mr{Emb}^\mr{id}_{\half \partial}}(W_{g,1})$. By definition there is a fibration sequence
\[B\mr{Emb}^\mr{id}_{\half \partial}(W_{g,1}) \lra B\mr{Emb}_{\half \partial}^{\cong}(W_{g,1}) \lra B\Lambda_g,\]
and hence the simply-connected space $B\mr{Emb}^\mr{id}_{\half \partial}(W_{g,1})$ has a basepoint-preserving $\Lambda_g$-action up to homotopy, as in Section \ref{sec:algebr-cohomology}.

\begin{theorem}\label{thm:emb-alg} 
Suppose that $2n \geq 6$. Then for each $i \geq 1$, the $\Lambda_g$-representation $\pi_{i+1}(B\mr{Emb}^\mr{id}_{\half \partial}(W_{g,1})) \otimes \bQ$ is $gr$-algebraic.
\end{theorem}

\begin{proof} 
Let us first recall how the $\Lambda_g$-action may be constructed geometrically. The identity component of the self-embeddings of $W_{g,1}$ admits an action of $\mr{Diff}_\partial(W_{g,1})$ by letting a diffeomorphism $\phi$ act on an embedding $e$ by $\phi \cdot e = \phi^{-1} \circ e \circ \phi$. This action is via morphisms of monoids, so induces an action on $B\mr{Emb}^\mr{id}_{\half \partial}(W_{g,1})$, too. The induced $\Gamma_g$-action on $\pi_{i+1}(B\mr{Emb}^\mr{id}_{\half \partial}(W_{g,1}))$ is then identified with the $\Gamma_g$-action on $\pi_{i}(\mr{Emb}_{\half \partial}(W_{g,1}), \mr{id})$ for $i \geq 1$. This action factors through the surjection $\Gamma_g \to \Lambda_g$, so to prove the theorem it is therefore enough to prove that the $\Gamma_g$-representations $\pi_{i}(\mr{Emb}_{\half \partial}(W_{g,1}), \mr{id}) \otimes \bQ$ are $gr$-algebraic for all $i \geq 1$. 

To do so we consider the embedding calculus tower for $\mr{Emb}_{\half \partial}(W_{g,1})$, which is a tower of $H$-spaces, and apply the Bousfield--Kan spectral sequence to it. As $\mr{id} \in \mr{Emb}_{\half \partial}(W_{g,1})$ is a $\mr{Diff}_\partial(W_{g,1})$-invariant base point, the Taylor approximations and layers inherit an action of $\mr{Diff}_\partial(W_{g,1})$, so their homotopy groups inherit a $\Gamma_g$-action. This spectral sequence is given by
	\[E^1_{p,q} = \pi_{q-p}(L_p(\mr{Emb}_{\half \partial}(W_{g,1})_\mr{id})) \Longrightarrow \pi_{q-p}(\mr{Emb}_{\half \partial}(W_{g,1}), \mr{id}),\]
	and by the $H$-space structure this is a spectral sequence of abelian groups for $q-p \geq 1$. By naturality of the Taylor tower and the Bousfield--Kan homotopy spectral sequence, this spectral sequence has an action of $\Gamma_g$. By Theorem B of \cite{goodwillieklein} the $p$th layer of this tower is $(-(2n-1)+p(n-2))$-connected, so the spectral sequence converges completely by \cite[Proposition IX.\S 5.7]{bousfieldkan}. Thus for each $i \geq 1$ the abelian group $\pi_{i}(\mr{Emb}_{\half \partial}(W_{g,1}), \mr{id})$ has a finite filtration with $p$th filtration quotient a subquotient of $E^1_{p, i+p}$. Hence $\pi_{i}(\mr{Emb}_{\half \partial}(W_{g,1}), \mr{id}) \otimes \bQ$ has a finite filtration with $p$th filtration quotient a subquotient of $E^1_{p, i+p} \otimes \bQ$.

By Lemma \ref{lem:filteredalgebraic}, the theorem follows as soon as we establish that each $\Gamma_g$-representation $E^1_{p,i+p} \otimes \bQ$ is $gr$-algebraic for $i \geq 1$. This was the content of previous two sections: Proposition \ref{prop:filteredlayer1} for $p=1$, Proposition \ref{prop:filteredhigherlayers} for $p \geq 2$.
\end{proof}

\begin{corollary}\label{cor:id-comp-alg} Suppose that $2n \geq 6$. Then the $\Lambda_g$-representations 
$$H^i(B\mr{Emb}^\mr{id}_{\half \partial}(W_{g,1});\bQ)$$
are $gr$-algebraic.\end{corollary}
\begin{proof}
Combine Theorem \ref{thm:emb-alg} with Lemma \ref{lem:filtered-algebraic-homotopy-to-homology}.
\end{proof}

\begin{corollary}\label{cor:algebraic-torelli-emb} 
Suppose that $2n \geq 6$ and that $g \geq 2$. Then the $G'_g$-representations $H^i(B\mr{TorEmb}^{\cong}_{\half \partial}(W_{g,1});\bQ)$ are algebraic.\end{corollary}

\begin{proof}
There is a fibration sequence
\[B\mr{Emb}^\mr{id}_{\half \partial}(W_{g,1}) \lra B\mr{TorEmb}^{\cong}_{\half \partial}(W_{g,1})\lra BJ_g,\]
with $J_g = \pi_0(\mr{TorEmb}^{\cong}_\partial(W_{g,1})) \overset{\cong}{\to} \mr{Hom}(H_n,S\pi_n(SO(n)))$ as in \eqref{eqn:lambda-setup}, to which we will apply Lemma \ref{lem:i-connective}. Given Corollary \ref{cor:id-comp-alg}, to apply that lemma we need to prove that the $G'_g$-representations $H^i(J_g;\bQ)$ are algebraic, but we have
\[H^i(J_g;\bQ) = \Lambda^i [H_n \otimes (S\pi_n(SO(2n)) \otimes \bQ)^\vee]\]
which is indeed algebraic.\end{proof}

We may now deduce Theorem \ref{thm:main}, which said that for $2n \geq 6$ and $g \geq 2$, the rational cohomology groups of $B\mr{Tor}_{\partial}(W_{g,1})$ are algebraic $G'_g$-representations.

\begin{proof}[Proof of Theorem \ref{thm:main}]
The Weiss fibration sequence \eqref{eqn:weiss} provides a commutative diagram with rows and columns fibration sequences
\[\begin{tikzcd} 
B\mr{Diff}_\partial(D^{2n}) \rar[equals] \dar  & B\mr{Diff}_\partial(D^{2n}) \rar \dar & \ast \dar \\
	B\mr{Tor}_\partial(W_{g,1}) \rar \dar & B\mr{Diff}_\partial(W_{g,1}) \rar \dar & BG'_g \dar[equals] \\ B\mr{TorEmb}^{\cong}_{\half \partial}(W_{g,1}) \rar &  B\mr{Emb}^{\cong}_{\half \partial}(W_{g,1}) \rar & BG'_g,
\end{tikzcd}\]
where the left and middle columns deloop compatibly by \cite[Theorem 4.17]{kupersdisk} and hence the action of the fundamental group of the base on the cohomology of the fibre is trivial, and on the left column the $G'_g$- and $B\mr{Diff}_\partial(D^{2n})$-actions commute. Thus we get a Serre spectral sequence of $G'_g$-representations 
	\[E^2_{p,q} = H^p(B\mr{TorEmb}^{\cong}_{\half \partial}(W_{g,1}); \bQ) \otimes H^q(B\mr{Diff}_\partial(D^{2n});\bQ) \Longrightarrow H^{p+q}(B\mr{Tor}_\partial(W_{g,1});\bQ).\]
 Using Corollary \ref{cor:algebraic-torelli-emb} and the fact that $H^*(B\mr{Diff}_\partial(D^{2n});\bQ)$ is degree-wise finite-dimensional by \cite[Theorem A]{kupersdisk}, the $E^2$-page consists of algebraic $G'_g$-representations, so by Theorem \ref{thm:representations} so does the abutment.\end{proof}

\begin{proof}[Proof of Corollary \ref{cor:main}]
Consider the map of fibrations
\[\begin{tikzcd} 
B\mr{Tor}_\partial(W_{g,1}) \rar \dar[equals] & B\mr{Diff}^{G_g''}_\partial(W_{g,1}) \rar \dar & BG''_g \dar \\ 
B\mr{Tor}_\partial(W_{g,1}) \rar &  B\mr{Diff}_\partial(W_{g,1}) \rar & BG'_g,
\end{tikzcd}\]
which on $E_2$-pages of the associated Serre spectral sequences induces
\[H^p(G'_g ; H^q(B\mr{Tor}_\partial(W_{g,1});\bQ) \otimes V) \lra H^p(G''_g ; H^q(B\mr{Tor}_\partial(W_{g,1});\bQ) \otimes V).\]
To compare the left and right hand side, we shall use work of Borel. By Corollary 5.5 of \cite{kupersdisk} each $H^q(B\mr{Tor}_\partial(W_{g,1});\bQ) \otimes V$ is a finite-dimensional $G'_g$-representation, so by theorems of Borel \cite{borelstable,borelstable2} (see Theorem 2.3 of \cite{KR-WTorelli} for a description of Borel's results adapted to this situation, using the bounds from \cite{tshishikuBorel}) the map
\[H^p(G_\infty;\bQ) \otimes [H^q(B\mr{Tor}_\partial(W_{g,1});\bQ) \otimes V]^{G''_g} \to H^p(G''_g ; H^q(B\mr{Tor}_\partial(W_{g,1});\bQ) \otimes V)\]
is an isomorphism for $p < g-e$, with $e=0$ if $n$ is odd and $e=1$ if $n$ is even. This only uses that $G''_g$ is an arithmetic subgroup of $\mathbf{G}(\bQ)$, so also holds with $G''_g$ replaced with $G'_g$.

Now by Theorem \ref{thm:main} the representation $H^q(B\mr{Tor}_\partial(W_{g,1});\bQ) \otimes V$ is algebraic, and by our assumptions both $G'_g$ and $G''_g$ are Zariski-dense in $\mathbf{G}(\bQ)$ (see Section 2.1.1 of \cite{KR-WTorelli}). Thus the $G'_g$- and $G''_g$-invariants coincide, so the map of total spaces induces an isomorphism on homology in total degrees $* < g-e$.
\end{proof}

\section{Proof of Theorem \ref{thm:nilp}}\label{sec:nilpotent}

In this section we prove that for $2n \geq 6$ the spaces $B\mr{Tor}_\partial(W_{g,1})$ are nilpotent. A suitable reference for the theory of nilpotent spaces is \cite[Chapter 3]{MayPonto}, but we recall the definition here. 

\begin{definition}
A path-connected based space $(X,x_0)$ is \emph{nilpotent} if $\pi_1(X,x_0)$ is a nilpotent group and for each $i>1$ the $\pi_1(X,x_0)$-module $\pi_i(X,x_0)$ is nilpotent. More generally, a space is \emph{nilpotent} if each of its path-components is nilpotent for each basepoint.\end{definition}

Examples of nilpotent spaces include simply-connected spaces and $n$-fold loop spaces. Nilpotent spaces are preserved by various constructions: e.g.~products and homotopy fibres, a special case of \cite[Proposition 4.4.1]{MayPonto}:

\begin{lemma}\label{lem:nilp-fib} Suppose $p \colon E \to B$ is a surjective fibration with $B$ path-connected. For $e \in E$, let $E_e$ denote the path-component containing $e$, $F$ denote the fibre over $p(e)$, and $F_e$ denote the path-component of $F$ containing $e$. If $E_e$ is nilpotent, then $F_e$ is nilpotent.
\end{lemma}

We showed in Lemma \ref{lem:NilpIsSerre} that the class $\cC$ of nilpotent $\Gamma$-modules which are finitely-generated as abelian groups is an equivariant Serre class, and showed in Lemma \ref{lem:NilpIsDualClosed} that this class is closed under duals or $\mr{Ext}^1_\bZ(-,\bZ)$. This gives us Lemmas \ref{lem:serre-class-ss}, \ref{lem:filtered-algebraic-homotopy-to-homology}, \ref{lem:filtered-algebraic-homology-to-homotopy}, and \ref{lem:homology-to-cohomology} as tools. We shall apply some of these with $\Gamma = I_g$.

We will now commence the proof of Theorem \ref{thm:nilp}, which says that $B\mr{Tor}_\partial(W_{g,1})$ is nilpotent as long as $2n \geq 6$. We will first prove the corresponding statement for $B\mr{TorEmb}^{\cong}_{\half \partial}(W_{g,1})$. This requires two pieces of input, analogous to Propositions \ref{prop:conf-rational-cohomology-alg} and \ref{prop:conf-rational-homotopy-alg}.

\begin{lemma}\label{lem:conf-cohomology-nilp} The $I_g$-modules $H^i(W_{g,1}^k,\Delta_\partial;\bZ)$ are nilpotent.\end{lemma}

\begin{proof}There is a long exact sequence of $I_g$-modules
	\[\cdots \lra H^i(W_{g,1}^k,\Delta_\partial;\bZ) \lra H^i(W_{g,1}^k;\bZ) \lra H^i(\Delta_\partial;\bZ) \lra \cdots,\]
	and since the $I_g$-action on $H^i(W_{g,1}^k;\bZ)$ is trivial, it suffices to prove that each $I_g$-module $H^i(\Delta_\partial;\bZ)$ is nilpotent. This follows by applying Lemma \ref{lem:serre-class-ss} to the spectral sequence \eqref{eqn:bk-cohomology} with $\bk = \bZ$, using the observation that the $I_g$-action on the $E^1$-page is trivial.
\end{proof}

\begin{lemma}\label{lem:conf-homotopy-nilp} The $I_g$-modules $\pi_i(\mr{tohofib}_{I \subset x_{\ul{k}}}\mr{Emb}(I,W_{g,1}))$ are nilpotent.\end{lemma}

\begin{proof}By the Totaro spectral sequence the $I_g$-modules $H^i(\mr{Emb}(\ul{k};W_{g,1});\bZ)$ are finitely generated as abelian groups and have nilpotent $I_g$-action, so by Lemmas \ref{lem:filtered-algebraic-homology-to-homotopy} and \ref{lem:homology-to-cohomology}  the homotopy groups of $\mr{Emb}(\ul{k};W_{g,1})$ are as well. By the argument of Proposition \ref{prop:conf-rational-homotopy-alg}, $\pi_i(\mr{tohofib}_{I \subset x_{\ul{k}}}\mr{Emb}(I,W_{g,1}))$ is a summand of $\pi_i(\mr{Emb}(\ul{k},W_{g,1}))$.
\end{proof}

\begin{proposition}\label{prop:TorEmbNilpotent}
For $2n \geq 6$, $B\mr{TorEmb}^{\cong}_{\half \partial}(W_{g,1})$ is nilpotent.
\end{proposition}

\begin{proof}
	
	It is a path-connected space with fundamental group $J_g \cong \mr{Hom}(H_n,S\pi_n(SO(n)))$, which is abelian so in particular nilpotent. It remains to show that $J_g$ acts nilpotently on the higher homotopy groups. That is, we need to show that the $J_g$-module $\pi_{i+1}(B\mr{TorEmb}^{\cong}_{\half \partial}(W_{g,1}))$ for $i>0$ admits a finite filtration by sub-$J_g$-modules such that the action on the associated graded is trivial. That is, it should lie in the class $\cC$ of $\bZ[J_g]$-modules which are finitely-generated as abelian groups and have nilpotent $J_g$-actions.
	
There is an $I_g$-module structure on the $J_g$-module $\pi_{i+1}(B\mr{TorEmb}_\partial(W_{g,1}))$ given by the surjection $I_g \to J_g$, which is induced by the geometric action of $I_g \subset \Gamma_g$ by conjugation. Hence it suffices to prove that $I_g$ acts nilpotently. Furthermore, by Lemma \ref{lem:serre-class-ss} the property of being a nilpotent $I_g$-module passes through the Bousfield--Kan homotopy spectral sequence used in the proof of Theorem \ref{thm:emb-alg}, so it suffices to prove that $I_g$ acts nilpotently on the higher homotopy groups of the layers $L_p(\mr{Emb}_{\half \partial}(W_{g,1})_\mr{id})$.
	
	There are two cases to consider. The first is $p=1$, in which case we have \[L_1(\mr{Emb}_{\half \partial}(W_{g,1})_\mr{id}) \simeq \mr{Bun}_{\half \partial}(TW_{g,1}).\] The first part of the proof of Proposition \ref{prop:filteredlayer1} did not require rational coefficients, and that argument shows that $I_g$ acts trivially on the higher homotopy groups of $\mr{Bun}_{\half \partial}(TW_{g,1})$.
	
	The second case is $p>1$, in which case we have 
	\[L_p(\mr{Emb}_{\half \partial}(W_{g,1})_\mr{id}) \simeq \mr{Sect}(t^\mr{id}_k,s^\mr{id}_k|_{\nabla_\partial}),\]
with the right hand side as described in Proposition \ref{prop:emb-calc-fibers}.	The Federer spectral sequence \eqref{eqn:federer-applied} is a spectral sequence of $I_g$-modules of the form
	\[E^2_{p,q} = \begin{cases}
	H^{p}\left(C_k(W_{g,1}),\nabla_\partial;\underline{\pi_{q}}(t_k^\mr{id})\right) & \text{$p \geq 0$, $q{-}p \geq 0$},\\
	0 & \text{else,}
	\end{cases}\Longrightarrow \pi_{q-p}(\mr{Sect}(t^\mr{id}_k,s^\mr{id}_k|_{\nabla_\partial})).\]
	Since $C_k(W_{g,1})$ is finite-dimensional, this is concentrated in finitely-many columns and by \cite[Proposition IX.\S 5.7]{bousfieldkan} converges completely.
	
	Working integrally, the simplification of the $E^2$-page which we gave in in \eqref{eqn:e2-explicit} does not apply. Instead, there is a trigraded spectral sequence of $I_g$-modules converging to the $E^2$-page of the Federer spectral sequence
	\[H^{p'}\left(B\fS_k,H^{q'}\big(\mr{Emb}(\ul{k},W_{g,1}),\tilde{\nabla}_\partial;\pi_q(\mr{tohofib}_{I \subset x_{\ul{k}}}\,\mr{Emb}(I,W_{g,1}))\big)\right) \Longrightarrow E^2_{p'+q',q}.\]
	In light of Lemma \ref{lem:emb-to-prod} for $\bk = \bZ$ and the fact that the local systems of coefficients $\pi_q(\mr{tohofib}_{I \subset x_{\ul{k}}}\,\mr{Emb}(I,W_{g,1}))$ are trivial by simply-connectivity, we may replace this $\smash{E^2}$-page by
	\[H^{p'}\left(B\fS_k,H^{q'}\big(W_{g,1}^k,\Delta_\partial;\pi_q(\mr{tohofib}_{I \subset x_{\ul{k}}}\,\mr{Emb}(I,W_{g,1}))\big)\right).\]
	By property \eqref{enum:equiv-serre-1} of equivariant Serre classes, it suffices to prove that these entries lie in $\cC$.
	
	The entry $E^2_{p',q'}$ may be computed using the bar complex
	\[C^{p'} = \mr{Hom}_{\bZ[\fS_k]}\left(\bZ[\fS_k]^{p'+1},H^{q'}\big(W_{g,1}^k,\Delta_\partial;\pi_q(\mr{tohofib}_{I \subset x_{\ul{k}}}\,\mr{Emb}(I,W_{g,1}))\big)\right),\]
	so by property \eqref{enum:equiv-serre-1} again, it suffices to prove that each of groups \[H^{q'}\left(W_{g,1}^k,\Delta_\partial;\pi_q(\mr{tohofib}_{I \subset x_{\ul{k}}}\,\mr{Emb}(I,W_{g,1}))\right)\] lies in $\cC$. By property \eqref{enum:equiv-serre-2} and the universal coefficients theorem, it suffices to prove that $H^{q'}(W_{g,1}^k,\Delta_\partial;\bZ)$ and $\pi_q(\mr{tohofib}_{I \subset x_{\ul{k}}}\mr{Emb}(I,W_{g,1}))$ are nilpotent $I_g$-modules. We did so in Lemmas \ref{lem:conf-cohomology-nilp} and \ref{lem:conf-homotopy-nilp}.
\end{proof}

We now prove Theorem \ref{thm:nilp}:

\begin{proof}[Proof of Theorem \ref{thm:nilp}] In the proof of Theorem \ref{thm:main} we constructed a fibre sequence
	\[B\mr{Diff}_\partial(D^{2n}) \lra B\mr{Tor}_\partial(W_{g,1}) \lra B\mr{TorEmb}^{\cong}_{\half \partial}(W_{g,1})\]
	which deloops. In particular, $B\mr{Tor}_\partial(W_{g,1})$ is the homotopy fibre of a map 
	\[B\mr{TorEmb}^{\cong}_{\half \partial}(W_{g,1}) \lra B^2\mr{Diff}_\partial(D^{2n}).\]
	The domain is a nilpotent space by the previous proposition, so by Lemma \ref{lem:nilp-fib} the space $B\mr{Tor}_\partial(W_{g,1})$ is also nilpotent.	
\end{proof}

\section{Generalisation to tangential structures} \label{sec:tangential} 

In this section we extend Theorem \ref{thm:main} and Corollary \ref{cor:main} to moduli spaces of manifolds equipped with a $\theta$-structure, encoded by a fibration $\theta \colon B \to BO(2n)$ which classifies a $2n$-dimensional vector bundle $\theta^*\gamma$ over $B$, where $\gamma$ denotes the universal $2n$-dimensional vector bundle over $BO(2n)$.

A \emph{$\theta$-structure} on $W_{g,1}$ is a map of vector bundles $\ell \colon TW_{g,1} \to \theta^* \gamma$. We shall fix a boundary condition $\ell_\partial \colon TW_{g,1}|_{\partial W_{g,1}} \to \theta^* \gamma$ and only consider $\theta$-structures extending this boundary condition: let $\mr{Bun}_\partial(TW_{g,1},\theta^* \gamma; \ell_\partial)$ denote the space of bundle maps extending $\ell_\partial$. A bundle map is a continuous map which is a fibrewise linear isomorphism, and this space is given the compact-open topology. The group $\mr{Diff}_\partial(W_{g,1})$ acts through the derivative map $\mr{Diff}_\partial(W_{g,1}) \to \mr{Bun}_\partial(TW_{g,1})$ on $\mr{Bun}_\partial(TW_{g,1},\theta^* \gamma; \ell_\partial)$ by precomposition. The object of interest in this section is the homotopy quotient
	\[B\mr{Diff}^\theta_\partial(W_{g,1};\ell_\partial) \coloneqq \mr{Bun}_\partial(TW_{g,1},\theta^* \gamma; \ell_\partial) \sslash \mr{Diff}_\partial(W_{g,1}).\]
	
Though the notation may suggest otherwise, this is \emph{not} the classifying space of a topological monoid and in general has many path components, which are in bijection with the orbits of the action of the mapping class group $\Gamma_g = \pi_0(\mr{Diff}_\partial(W_{g,1}))$ on the set of path components $\pi_0(\mr{Bun}_\partial(TW_{g,1},\theta^* \gamma; \ell_\partial))$. 

We shall denote by $B\mr{Diff}^\theta_\partial(W_{g,1};\ell_\partial)_\ell$ the path component containing a particular $\theta$-structure $\ell$, and by $G^{\theta,[\ell]}_g$ the image of the composition
\[\check{\Gamma}^{\theta,\ell}_g \coloneqq \pi_1(B\mr{Diff}^\theta_\partial(W_{g,1};\ell_\partial), \ell) \lra \Gamma_g = \pi_1(B\mr{Diff}_\partial(W_{g,1}), *) \lra G'_g.\]
We shall first show that this is an arithmetic group:

\begin{proposition}\label{prop:theta-fin-index} Let $2n \geq 6$ and $B$ be $n$-connected. Then $G^{\theta,[\ell]}_g \leq G'_g$ has finite index.
\end{proposition}

The proof will be given in Section \ref{sec:proof:theta-fin-index}, after some preparation.

\begin{remark}
	The assumption that $B$ be $n$-connected is minor. Since $W_{g,1}$ is $(n-1)$-connected, for any tangential structure $\theta' \colon B' \to BO(2n)$ the relevant spaces of $\theta'$-structures are weakly equivalent to the corresponding spaces of $\theta''$-structures for $\theta'' \colon B'\langle n \rangle \to B' \to BO(2n)$, with $B' \langle n \rangle \to B'$ the $n$-connected cover.
\end{remark}

We may define a version of the Torelli space with $\theta$-structures by
\[B\mr{Tor}^\theta_\partial(W_{g,1};\ell_\partial)_\ell \coloneqq \mr{hofib} \left[B\mr{Diff}^\theta_\partial(W_{g,1};\ell_\partial)_\ell \to BG^{\theta,[\ell]}_g\right];\]
this is a connected space. It has a $G^{\theta,[\ell]}_g$-action up to homotopy, so its rational cohomology groups are $G^{\theta,[\ell]}_g$-representations, and as $G^{\theta,[\ell]}_g$ is an arithmetic group by Proposition \ref{prop:theta-fin-index} one may ask whether they are algebraic $G^{\theta,[\ell]}_g$-representations. Analogously to Theorem \ref{thm:main}, we will show that they are.

\begin{theorem}\label{thm:maintangent} Let $2n \geq 6$, $B$ be $n$-connected, and $H^*(B;\bQ)$ be finite-dimensional in each degree. For $g \geq 2$ the $G^{\theta,[\ell]}_g$-representations $H^i(B\mr{Tor}_\partial^\theta(W_{g,1};\ell_\partial)_\ell;\bQ)$ are algebraic.\end{theorem}

The proof will be given in Section \ref{sec:proof:maintangent}. The analogue of Corollary \ref{cor:main} for $\theta$-structures follows by a nearly identical argument. After that we prove the analogue of Theorem \ref{thm:nilp}:

\begin{theorem}\label{thm:nilp-tangential} Let $2n \geq 6$ and $B$ be $n$-connected. Then $B\mr{Tor}^\theta_\partial(W_{g,1};\ell_\partial)_\ell$ is nilpotent.\end{theorem}

\subsection{Notation}\label{sec:Notation}
 We collect here the notation for various groups used in this section. Further details and results regarding these objects will be given in the following sections. 

First, we recall that there are maps of group extensions
\[\begin{tikzcd}1 \rar &[-10pt] I_g \dar[two heads] \rar & \Gamma_g = \pi_0(\mr{Diff}_\partial(W_{g,1})) \dar[two heads] \rar & G'_g \rar \dar[equals] &[-10pt] 1 \\
1 \rar & J_g \dar \rar & \Lambda_g = \pi_0(\mr{Emb}^{\cong}_{\half \partial}(W_{g,1})) \dar \rar & G'_g \rar \dar & 1 \\
1 \rar & \mr{Hom}(H_n,\pi_n(SO(2n))) \rar & \Upsilon_g = \pi_0(\mr{Bun}_\ast(TW_{g,1})) \rar & \mr{GL}(H_n) \rar & 1, \end{tikzcd}\]
with $J_g = I_g/\Theta_{2n+1} = \mr{Hom}(H_n,S\pi_n(SO(n)))$. 

We fix a boundary condition $\ell_\partial$ of $\theta$-structures near $\partial W_{g,1}$, and let $\mr{Str}^{\theta}_\partial(W_{g,1})$ denote the set of homotopy classes of $\theta$-structure extending $\ell_\partial$. An element in this set is denoted $[\ell]$, and we introduce notation for its stabiliser:
\[\Gamma^{\theta,[\ell]}_g \coloneqq \mr{Stab}_{\Gamma_g}([\ell]).\]
This receives a surjection from
\[\check{\Gamma}^{\theta,\ell}_g \coloneqq \pi_1(B\mr{Diff}^{\theta}_\partial(W_{g,1};\ell_\partial), \ell),\]
and we set 
\begin{equation*}
G^{\theta,[\ell]}_g \coloneqq  \mr{im}\left[\Gamma^{\theta,[\ell]}_g \to \Gamma_g \to G'_g\right] \quad\quad\quad I^{\theta,[\ell]}_g \coloneqq  \mr{Stab}_{I_g}([\ell]) = \Gamma^{\theta,[\ell]}_g \cap I_g,
\end{equation*}
giving an extension
\[1 \lra I^{\theta,[\ell]}_g \lra \Gamma^{\theta,[\ell]}_g \lra G^{\theta,[\ell]}_g \lra 1.\]

The boundary condition $\ell_\partial$ of $\theta$-structures near $\partial W_{g,1}$ by restriction determines one, $\ell_{\half \partial}$ near a fixed point $\ast \in \partial W_{g,1}$, and we let $\mr{Str}^{\theta}_\ast(W_{g,1})$ denote the set of homotopy classes of $\theta$-structure extending $\ell_{\half \partial}$. An element in this set is denoted $[[\ell]]$ and, though we shall not need it, its stabiliser is denoted $\Gamma_g^{\theta, [[\ell]]}$. The $\Gamma_g$-action on $\mr{Str}^{\theta}_\ast(W_{g,1})$ descends to an action of $\Lambda_g$, and we write
\[\Lambda^{\theta,[[\ell]]}_g \coloneqq \mr{Stab}_{\Lambda_g}([[\ell]]).\]
This receives a surjection from
\[\check{\Lambda}^{\theta,\ell}_g \coloneqq \pi_1(B\mr{Emb}^{\theta}_{\half \partial}(W_{g,1};\ell_{\half \partial}), \ell),\]
and we set 
\begin{equation*}
G^{\theta,[[\ell]]}_g \coloneqq \mr{im}\left[\Lambda^{\theta,[[\ell]]}_g \to \Lambda_g \to G'_g\right] \quad\quad\quad J^{\theta,[[\ell]]}_g \coloneqq \mr{Stab}_{J_g}([[\ell]]) = \Lambda^{\theta,[[\ell]]}_g \cap J_g,
\end{equation*} 
giving an extension
\[1 \lra J^{\theta,[[\ell]]}_g \lra \Lambda^{\theta,[[\ell]]}_g \lra G^{\theta,[[\ell]]}_g \lra 1.\]
Finally, we define $L^{\theta,\ell}_g$ as the kernel of the composition $\check{\Lambda}^{\theta,\ell}_g \to {\Lambda}^{\theta,[[\ell]]}_g \to  G^{\theta,[[\ell]]}_g$, giving an extension
\[1 \lra L^{\theta,\ell}_g \lra \check{\Lambda}^{\theta,\ell}_g \lra G^{\theta,[[\ell]]}_g \lra 1.\] 

\subsection{Spaces of $\theta$-structures on $W_{g,1}$} 

The proof of Proposition \ref{prop:theta-fin-index} and Theorem \ref{thm:maintangent} requires a careful study of the action of $\Gamma_g$ on the set of homotopy classes of $\theta$-structures on $W_{g,1}$. We will write
\[F \coloneqq \mr{Fr}(\theta^*\gamma)\]
for the frame bundle of the vector bundle $\theta^*\gamma \to B$; this is a principal $\mr{GL}_{2n}(\bR)$-bundle, and $F$ is homotopy equivalent to the homotopy fibre of $\theta: B \to BO(2n)$. A bundle map $TW_{g,1} \to \theta^* \gamma$ is precisely the same as a $\mr{GL}_{2n}(\bR)$-equivariant map $\mr{Fr}(TW_{g,1}) \to F$.

Under the assumption that $B$ is $n$-connected, so in particular simply-connected, we may choose once and for all an orientation of the bundle $\theta^*\gamma$, and let
\[F^\mr{or} \coloneqq \mr{Fr}^\mr{or}(\theta^*\gamma)\]
denote the oriented frame bundle, which is path-connected (it is a path component of $F$). An orientation-preserving bundle map $TW_{g,1} \to \theta^* \gamma$ is precisely the same as a $\mr{GL}_{2n}^+(\bR)$-equivariant map $\mr{Fr}^\mr{or}(TW_{g,1}) \to F^\mr{or}$.

If $\tau \colon TW_{g,1} \to W_{g,1} \times \bR^{2n}$ is a choice of (orientation-preserving) framing then choosing a basepoint $f_0 \in F^\mr{or}$ defines a $\mr{GL}_{2n}^+(\bR)$-equivariant map
\[\ell^\tau \colon \mr{Fr}^\mr{or}(TW_{g,1}) \overset{\tau}\lra W_{g,1} \times \mr{GL}_{2n}^+(\bR) \xrightarrow{(x, g) \mapsto g \cdot f_0} F^\mr{or}\]
which up to homotopy does not depend on the choice of $f_0$, as $F^\mr{or}$ is path-connected. It does however depend on $\tau$.

\begin{lemma}\label{lem:SpaceThetaStr}
Let $B$ be $n$-connected. 
\begin{enumerate}[\indent (i)]
\item\label{it:UniqueBdy} Up to homotopy there is a unique orientation preserving boundary condition $\ell_\partial$ which extends to a $\theta$-structure $\ell$ on all of $W_{g,1}$, and $\ell^\tau_\partial \coloneqq \ell^\tau\vert_{\partial W_{g,1}}$ represents this homotopy class of boundary condition.

\item\label{it:ThetaSpace} For such a boundary condition there is a homotopy equivalence
\[\mr{Bun}_\partial(TW_{g,1},\theta^* \gamma; \ell_\partial) \simeq \mr{map}_\partial(W_{g,1}, F^\mr{or}),\]
depending on a framing $\tau$ and a homotopy from $\ell_\partial$ to $\ell^\tau_\partial$.
\end{enumerate}
\end{lemma}

Recall that we write
\[\mr{Str}^\theta_\partial(W_{g,1}) = \pi_0(\mr{Bun}_\partial(TW_{g,1},\theta^* \gamma; \ell_\partial)),\]
for the set of homotopy classes of $\theta$-structures on $W_{g,1}$ rel boundary, the omission of the boundary condition $\ell_\partial$ from the notation justified by Lemma \ref{lem:SpaceThetaStr} (\ref{it:UniqueBdy}). It is important to be aware that under the bijection
\[\pi_0(\mr{Bun}_\partial(TW_{g,1},\theta^* \gamma; \ell_\partial)) \cong \pi_0(\mr{map}_\partial(W_{g,1}, F^\mr{or}))\]
given by Lemma \ref{lem:SpaceThetaStr} (\ref{it:ThetaSpace}) the action of the mapping class group $\Gamma_g = \pi_0(\mr{Diff}_\partial(W_{g,1}))$ on the set of homotopy classes of $\theta$-structures does \emph{not} in general correspond to the action by precomposition on the mapping space. Instead an analysis analogous to that of Section \ref{sec:TrivTang} must be made, which we will do below.

\begin{proof}[Proof of Lemma \ref{lem:SpaceThetaStr}]
If $X \times \bR^{2n}$ is a trivial bundle then there is a homeomorphism
\[\mr{map}(X, F^\mr{or}) \overset{\sim}\lra \mr{Bun}^{\mr{or}}(X \times \bR^{2n},\theta^*\gamma)\]
to the space of $\mr{GL}_{2n}^+(\bR)$-equivariant maps $X \times \mr{GL}_{2n}^+(\bR) \to F^\mr{or}$, given by sending the map $f \colon X \to F^\mr{or}$ to the $\mr{GL}_{2n}^+(\bR)$-equivariant map
\[\ell^f \colon X \times \mr{GL}_{2n}^+(\bR) \xrightarrow{(x, g) \mapsto g \cdot f(x)} F^\mr{or}.\]

Fixing a choice of (orientation-preserving) framing $\tau \colon TW_{g,1} \to W_{g,1} \times \bR^{2n}$, any $\theta$-structure $\ell \in \mr{Bun}^{\mr{or}}(TW_{g,1},\theta^* \gamma)$ therefore corresponds to a map $f_\ell \colon W_{g,1} \to F^\mr{or}$, and the associated boundary condition $\ell_\partial \in \mr{Bun}^{ \mr{or}}(TW_{g,1}\vert_{\partial W_{g,1}},\theta^* \gamma)$ corresponds to the composition $f_\ell \circ \mr{inc} \colon S^{2n-1} \to W_{g,1} \to F^\mr{or}$. We hence need to show that the map $[W_{g,1},F^\mr{or}] \to [S^{2n-1},F^\mr{or}]$ of homotopy classes, induced by restriction to the boundary, is constant. It suffices to prove this for based homotopy classes, as $[W_{g,1},F^\mr{or}]_* \to [W_{g,1},F^\mr{or}]$ is surjective because $F^\mr{or}$ is path-connected.

From the fibration $SO(2n) \simeq \mr{GL}_{2n}^+(\bR) \to \mr{Fr}^\mr{or}(\theta^*\gamma) \to B$ we obtain an exact sequence
\[\pi_{n+1}(B, x_1) \overset{\partial}\lra \pi_{n}(SO(2n),\mr{id}) \overset{a_*}\lra \pi_n(F^\mr{or}, x_1) \lra \pi_n(B, x_1)=0,\]
using our assumption that $B$ is $n$-connected. In particular the map $a_*$ is surjective, and since all Whitehead products vanish in an $H$-space such as $SO(2n)$, the Whitehead bracket
\[[-,-] \colon \pi_n(F^\mr{or}, x_1) \times \pi_n(F^\mr{or}, x_1) \lra \pi_{2n-1}(F^\mr{or}, x_1)\]
must also be zero. The inclusion of the boundary $\mr{inc} \colon S^{2n-1} \to W_{g,1}$ is represented by the homotopy class
\[\sum_{i=1}^g [a_i, b_i] \in \pi_{2n-1}(W_{g,1}, x_0),\]
where $a_i, b_i \colon S^n \to W_{g,1}$ form a hyperbolic basis of the intersection form on $\pi_n(W_{g,1}, x_0) \cong H_n(W_{g,1};\bZ)$. In particular it is a sum of Whitehead products, and hence the map $f_\ell \circ \mr{inc} \colon S^{2n-1} \to F^\mr{or}$ must be nullhomotopic. This proves 
part (\ref{it:UniqueBdy}).

For part (\ref{it:ThetaSpace}) we observe that the above identification using $\tau$ gives a map of homotopy cartesian squares from
\begin{equation*}
\begin{tikzcd} 
\mr{map}_\partial(W_{g,1}, F^\mr{or}) \rar \dar &\mr{map}(W_{g,1}, F^\mr{or}) \dar\\
\{\mr{const}\} \rar  & \mr{map}(\partial W_{g,1}, F^\mr{or})
\end{tikzcd}
\end{equation*}
to 
\begin{equation*}
\begin{tikzcd} 
\mr{Bun}_\partial(TW_{g,1},\theta^* \gamma; \ell_\partial^\tau) \rar \dar &\mr{Bun}^{\mr{or}}(TW_{g,1},\theta^* \gamma) \dar\\
\{\ell^\tau_\partial\} \rar  & \mr{Bun}^{\mr{or}}(TW_{g,1}\vert_{\partial W_{g,1}},\theta^* \gamma)
\end{tikzcd}
\end{equation*}
which is an equivalence at each corner apart from the top left corner, so is also an equivalence at the top left corner. Finally, a choice of homotopy from $\ell_\partial^\tau$ to $\ell_\partial$ gives an equivalence $\mr{Bun}_\partial(TW_{g,1},\theta^* \gamma; \ell_\partial^\tau) \simeq \mr{Bun}_\partial(TW_{g,1},\theta^* \gamma; \ell_\partial)$, using the homotopy lifting property for the right-hand map of the second square.
\end{proof}

\subsubsection{Mapping class group action}\label{sec:MCGact}

Given a choice of framing $\tau$ we have produced a bijection
\[\mr{Str}^\theta_\partial(W_{g,1}) = \pi_0(\mr{Bun}_\partial(TW_{g,1},\theta^* \gamma; \ell_\partial)) \cong \pi_0(\mr{map}_\partial(W_{g,1},F^\mr{or})),\]
and we wish to understand the orbits and stabilisers of the natural action of $\Gamma_g$ on the left-hand side. To do so we must describe the corresponding $\Gamma_g$-action on the right-hand side. In Section \ref{sec:first-layer} we have described how a choice of framing identifies the topological monoid $\mr{Bun}_{\half \partial}(TW_{g,1})$ with $\mr{map}_{\half \partial}(W_{g,1},W_{g,1} \times \mr{GL}_{2n}(\bR))$, and we determined the induced composition law on this space. The same discussion goes through when we impose a boundary condition on the entire boundary instead: there is a homeomorphism
\[\mr{Bun}_{\partial}(TW_{g,1}) \overset{\cong}\lra \mr{map}_{\partial}(W_{g,1},W_{g,1} \times \mr{GL}_{2n}^+(\bR))\]
under which composition of bundle maps corresponds to the operation 
\[(f,\lambda) \circledast (g,\rho) = (f \circ g,(\lambda \circ g) \cdot \rho),\]
with $\circ$ denoting composition of maps and $\cdot$ denoting pointwise multiplication. 

In the proof of Lemma \ref{lem:SpaceThetaStr} we have similarly used $\tau$ to identify the space of $\theta$-structures $\mr{Bun}_{\partial}(TW_{g,1},\theta^* \gamma; \ell_\partial)$ with $\mr{map}_{\partial}(W_{g,1}, F^\mr{or})$. Similarly to Lemma \ref{lem:SDaction} one sees that under this identification the right action of $\mr{Bun}_{\partial}(TW_{g,1})$ by precomposition corresponds to
\begin{align*}
\mr{map}_\partial(W_{g,1}, F^\mr{or}) \times \mr{map}_{\partial}(W_{g,1},W_{g,1} \times \mr{GL}^+_{2n}(\bR))  &\lra \mr{map}_\partial(W_{g,1}, F^\mr{or})\\
(h, (f,\lambda)) &\longmapsto  (h \circ f) \cdot \lambda,
\end{align*}
where here $\cdot$ denotes the left $\mr{GL}^+_{2n}(\bR)$-action on $F^\mr{or}$. We write $h \circledast (f,\lambda) $ for this operation. Note that it is \emph{not} equal to precomposition on the mapping space.

\subsubsection{Relaxing the boundary condition}\label{sec:Relaxing}

In the long exact sequence of homotopy groups for the fibration sequence 
\[\mr{map}_\partial(W_{g,1}, F^\mr{or}) \lra \mr{map}_*(W_{g,1}, F^\mr{or}) \lra \mr{map}_*(\partial W_{g,1}, F^\mr{or}),\]
based at the constant maps to the basepoint $f_0 \in F^\mr{or}$, the maps $\pi_i(\mr{map}_*(W_{g,1}, F^\mr{or})) \to \pi_i(\mr{map}_*(\partial W_{g,1}, F^\mr{or}))$ are given by a sum of Whitehead products of elements in $\pi_n(\Omega^i F^\mr{or})$. For $i>0$ such Whitehead products vanish as $\Omega^i F^\mr{or}$ is a loop space; for $i=0$ they also vanish as discussed in the proof of Lemma \ref{lem:SpaceThetaStr}. Thus this yields a short exact sequence in the sense of groups and sets
\begin{equation}\label{eqn:theta-struc-pi0} 
0 \lra \pi_{2n}(F^\mr{or}) \overset{\circlearrowright}\lra \pi_0(\mr{map}_\partial(W_{g,1}, F^\mr{or})) \lra \mr{Hom}(H_n,\pi_n(F^\mr{or})) \lra 0;
\end{equation}
(By $G \overset{\circlearrowright}\to X$ we indicate an action of a group $G$ on a set $X$.) Recalling that $\mr{Str}^\theta_\ast(W_{g,1})$ denotes the homotopy classes of $\theta$-structures on $W_{g,1}$ equal to $\ell_\partial$ near the point $\ast \in \partial W_{g,1}$, this may be rewritten as
\[\begin{tikzcd}0 \rar & \mr{Str}^\theta_\partial(D^{2n}) \rar{\circlearrowright} \dar{\cong}& \mr{Str}^\theta_\partial(W_{g,1}) \rar \dar{\cong} & \mr{Str}^\theta_\ast(W_{g,1}) \rar \dar{\cong} & 0 \\
0 \rar & \pi_{2n}(F^\mr{or}) \rar{\circlearrowright} & \pi_0(\mr{map}_\partial(W_{g,1}, F^\mr{or})) \rar & \mr{Hom}(H_n,\pi_n(F^\mr{or})) \rar & 0, \end{tikzcd}
\]
where the vertical bijections depend on the choice of framing $\tau$.

This short exact sequence is equivariant for the right action of the mapping class group $\Gamma_g$ in  the following sense. Let $\Gamma_g$ act on the middle term via the derivative $\Gamma_g \to \pi_0(\mr{Bun}_{\partial}(TW_{g,1}))$ and $\circledast$ above. Similarly, $\Gamma_g$ acts on the right-hand term via the derivative
\[\Gamma_g \lra \Upsilon_g = \pi_0(\mr{Bun}_{\half \partial}(TW_{g,1})),\]
and factors over $\Lambda_g = \pi_0(\mr{Emb}^{\cong}_{\half \partial}(W_{g,1}))$. The right hand side was identified with $\mr{GL}(H_n) \ltimes \mr{Hom}(H_n,\pi_n(SO(2n)))$ in Section \ref{sec:lbun}, and in terms of this identification the action $(B,\beta) \in \Upsilon_g$ is given by 
\[\alpha \circledast (B, \beta) = \alpha \circ B + \iota_*\beta,\]
where $\iota \colon SO(2n) \subset \mr{GL}^+_{2n}(\bR) \to F^\mr{or}$ is given by acting on the basepoint of $F^\mr{or}$. With these actions the map
\[\pi_0(\mr{map}_\partial(W_{g,1}, F^\mr{or})) \lra \mr{Hom}(H_n,\pi_n(F^\mr{or}))\]
is $\Gamma_g$-equivariant, and the $\Gamma_g$- and $\pi_{2n}(F^\mr{or})$-actions on $\pi_0(\mr{map}_\partial(W_{g,1}, F^\mr{or}))$ commute (this is because the $\pi_{2n}(F^\mr{or})$-action is by changing the $\theta$-structure in a small disc near the boundary, and diffeomorphisms in $\Gamma_g$ can be assumed to fix such a disc).

\subsection{Framings} 

An important example of a tangential structure satisfying the conditions given in the beginning of this section is a framing: we take the tangential structure to be $\mr{fr} \colon EO(2n) \to BO(2n)$.

In this case $F^\mr{or} \simeq SO(2n)$, so when we specialise \eqref{eqn:theta-struc-pi0} to framings, we see that the set $\mr{Str}^\mr{fr}_\partial(W_{g,1})$ of homotopy classes of framings of $W_{g,1}$ extending $\ell_\partial$ is in bijection with the middle term of the short exact sequence in the sense of groups and sets
\[\begin{tikzcd}0 \rar &[-10pt] \mr{Str}^\mr{fr}_\partial(D^{2n}) \rar{\circlearrowright} \dar{\cong}&[-5pt] \mr{Str}^\mr{fr}_\partial(W_{g,1}) \rar \dar{\cong} &[-5pt] \mr{Str}^\mr{fr}_\ast(W_{g,1}) \rar \dar{\cong} &[-10pt] 0 \\
0 \rar & \pi_{2n}(SO(2n)) \rar{\circlearrowright} & \pi_0(\mr{map}_\partial(W_{g,1}, SO(2n))) \rar & \mr{Hom}(H_n,\pi_n(SO(2n))) \rar & 0. \end{tikzcd}
\]
For $n \geq 3$, the groups $\pi_{n}(SO(2n))$ were determined by Bott and the groups $\pi_{2n}(SO(2n))$ by Kervaire \cite{KervaireNonstable}. We will only use that $\pi_{2n}(SO(2n))$ is always finite.

Recall from Section \ref{sec:emb-path-components} that $I_g$ denotes the Torelli subgroup of the mapping class group $\Gamma_g$, that is the kernel of $\alpha_g \colon \Gamma_g \to G'_g$.

\begin{lemma}\label{lem:ig-action-framing} 
For $n \geq 3$, the action of the subgroup $I_g \leq \Gamma_g$ on $\mr{Str}^\mr{fr}_\partial(W_{g,1,})$ via $\circledast$ has finitely many orbits.
\end{lemma}

\begin{proof} The framing $\tau$ gives a bijection $\mr{Str}^\mr{fr}_\partial(W_{g,1,}) \cong \pi_0(\mr{map}_\partial(W_{g,1},SO(2n)))$. Since $\pi_{2n}(SO(2n))$ is finite, it suffices to prove that there are finitely many orbits for the $I_g$-action on the set $\mr{Str}^\mr{fr}_\ast(W_{g,1})$ of homotopy classes of framings relative to a point. As before, we identify this set through the framing with $\pi_{2n}(SO(2n)) \backslash \pi_0(\mr{map}_\partial(W_{g,1},SO(2n))) = \pi_0(\mr{map}_*(W_{g,1},SO(2n)))$.

To study the $I_g$-action, recall from Theorem \ref{thm:kreck} that $I_g$ is an extension
	\[1 \lra \Theta_{2n+1} \lra I_g \lra \mr{Hom}(H_n,S\pi_n(SO(n))) \lra 1.\]
The action of $\Gamma_g$ on the set $\mr{Hom}(H_n,\pi_n(SO(2n)))$ is through the derivative map $\Gamma_g \to \Upsilon_g = \pi_0(\mr{Bun}_{\half \partial}(TW_{g,1}))$, whose structure was determined in Lemma \ref{lem:semi-direct-2}. Thus we need to understand the image of $I_g$ in $\Upsilon_g$.

The derivatives of elements of $\Theta_{2n+1}$ are bundle maps supported in a small disc which can be taken to be near the boundary: when only half of the boundary is required to be fixed these may be homotoped to the identity (or one may use Lemma \ref{lem:DiscActTriv}). Thus the homomorphism $I_g \to \Upsilon_g$ factors over the quotient group 
\[J_g = I_{g}/\Theta_{2n+1} = \mr{Hom}(H_n,S\pi_n(SO(n))).\]
By Lemma \ref{lem:semi-direct-2}, the map from $J_g = \mr{Hom}(H_n,S\pi_n(SO(n)))$ lands in the subgroup $\mr{Hom}(H_n,\pi_n(SO(2n))) \subset \Upsilon_g$, and this homomorphism is induced by applying the homomorphism $S\pi_n(SO(n)) \to \pi_n(SO(2n))$ to the target. Since this is surjective when $n \neq 1,3,7$ by Lemma \ref{lem:levine}, in this case the image of $I_g$ is exactly $\mr{Hom}(H_n,\pi_n(SO(2n)))$. The action of this group on $\mr{Hom}(H_n,\pi_n(SO(2n)))$ is through addition, which is transitive. This proves that there is a unique $I_g$-orbit for $n \neq 3,7$. On the other hand if $n=3,7$ then $S\pi_n(SO(n)) \to \pi_n(SO(2n))$ has cokernel $\bZ/2$, again by Lemma \ref{lem:levine}, so the set of orbits is in bijection with $\mr{Hom}(H_n, \bZ/2)$ and is still finite.
\end{proof}

Recall from Section \ref{sec:Notation} that we write $\Gamma_g^{\mr{fr}, [\ell]}$ for the stabiliser of $[\ell] \in \mr{Str}^\mr{fr}_\partial(W_{g,1,})$ under the action of $\Gamma$, and write $G^{\mr{fr}, [\ell]}_g = \mr{im}\left[\Gamma_g^{\mr{fr}, [\ell]} \to \Gamma_g \to G'_g \right]$.

\begin{corollary}\label{cor:framing-fin-ind} 
The group $G^{\mr{fr}, [\ell]}_g$ has finite index in $G'_g$.
\end{corollary}
\begin{proof}
This image is the stabiliser of $[\ell] \in \mr{Str}^\mr{fr}_\partial(W_{g,1,})/I_g$ with respect to the residual $G'_g = \Gamma_g / I_g$-action. As this is a finite set by Lemma \ref{lem:ig-action-framing}, this stabiliser has finite index.
\end{proof}

\subsection{Proof of Proposition \ref{prop:theta-fin-index}}\label{sec:proof:theta-fin-index}

Since $EO(2n)$ is contractible, there is a unique map $EO(2n) \to B$ over $BO(2n)$ up to homotopy, using which any framing determines a $\theta$-structure: we say such $\theta$-structures come from framings. This induces a map $SO(2n) \to F^\mr{or}$ as well as a map of short exact sequences \eqref{eqn:theta-struc-pi0}:
\[\begin{tikzcd}0 \rar &[-12pt] \pi_{2n}(SO(2n))) \dar \rar &[-5pt] \pi_0(\mr{map}_\partial(W_{g,1}, SO(2n))) \rar \dar &[-5pt] \mr{Hom}(H_n,\pi_n(SO(2n))) \rar \dar[two heads] &[-12pt] 0 \\
0 \rar & \pi_{2n}(F^\mr{or}) \rar & \pi_0(\mr{map}_\partial(W_{g,1}, F^\mr{or})) \rar & \mr{Hom}(H_n,\pi_n(F^\mr{or})) \rar & 0 \end{tikzcd}\]
with right-hand map surjective because $\pi_n(SO(2n)) \to \pi_n(F^\mr{or})$ is surjective, by our assumption that $B$ is $n$-connected. This is identified with the map $\mr{Str}^\mr{fr}_\ast(W_{g,1}) \to \mr{Str}^\theta_\ast(W_{g,1})$, which is therefore also surjective.

Thus, given a $[\ell] \in \mr{Str}^\theta_\partial(W_{g,1}) \cong \pi_0(\mr{map}_\partial(W_{g,1}, F^\mr{or}))$, there is another  $[\ell_0]$ coming from a framing which has the same image in $ \mr{Str}^\theta_\ast(W_{g,1})$. Since the bottom sequence is exact, these differ by the action of an element $\pi_{2n}(F^\mr{or})$. Changing the $\theta$-structure by this element in a small disc near the boundary, we obtain a homotopy equivalence $B\mr{Diff}^\theta_\partial(W_{g,1};\ell_\partial)_\ell \simeq B\mr{Diff}^\theta_\partial(W_{g,1};\ell_\partial)_{\ell_0}$. In conclusion, each path component of $B\mr{Diff}^\theta_\partial(W_{g,1};\ell_\partial)$ is homotopy equivalent (over $B\mr{Diff}_\partial(W_{g,1})$) to one which comes from a framing.

In particular, the group $G^{\theta,[\ell]}_g$ is conjugate to $G^{\theta,[\ell_0]}_g$ where $\ell_0$ is a $\theta$-structure that comes from a framing. By Corollary \ref{cor:framing-fin-ind} the inclusion 
\[G^{\mr{fr},[\ell_0]}_g \subset G^{\theta,[\ell_0]}_g \subset G'_g\]
has finite index, and hence so does $G^{\theta,[\ell_0]}_g \subset G'_g$.

\subsection{Proof of Theorem \ref{thm:maintangent}} \label{sec:proof:maintangent} We shall repeat the proof the argument for Theorem \ref{thm:main} while carrying along the tangential structure $\theta \colon B \to BO(2n)$. Recall that we assume that $B$ is $n$-connected, and $H^*(B;\bQ)$ is finite-dimensional in each degree.

\subsubsection{The Weiss fibration sequence with tangential structures} The first step of Theorem \ref{thm:main} was to reduce from diffeomorphisms to self-embeddings. We shall do the same for tangential structures.

From the boundary condition of $\theta$-structures $\ell_\partial$  near $\partial W_{g,1}$ we can extract by restriction a new boundary condition $\ell_{\half \partial}$ near $\half \partial W_{g,1}$. The topological monoid $\mr{Emb}^{\cong}_{\half \partial}(W_{g,1})$ acts on the space $\mr{Bun}_{\half \partial}(TW_{g,1},\theta^*\gamma;\ell_{\half \partial})$ through the derivative map \[\mr{Emb}^{\cong}_{\half \partial}(W_{g,1}) \lra \mr{Bun}_{\half \partial}(TW_{g,1}).\]
In analogy with $B\mr{Diff}^\theta_\partial(W_{g,1};\ell_\partial)$, we take the homotopy quotient
\[B\mr{Emb}^{\cong,\theta}_{\half \partial}(W_{g,1};\ell_{\half \partial}) \coloneqq \mr{Bun}_{\half \partial}(TW_{g,1},\theta^* \gamma;\ell_{\half \partial}) \sslash \mr{Emb}^{\cong}_{\half \partial}(W_{g,1}).\]
In the same way as the set of path components of $\smash{B\mr{Diff}^\theta_\partial(W_{g,1};\ell_\partial)}$ is given by the set of orbits $\mr{Str}^\theta_\partial(W_{g,1})/\Gamma_g$, the set of path components of $\smash{B\mr{Emb}^{\cong,\theta}_{\half \partial}(W_{g,1};\ell_{\half \partial})}$ is given by $\mr{Str}^\theta_\ast(W_{g,1})/\Lambda_g$.

The difference between diffeomorphisms and self-embeddings with tangential structures is described by an analogue of the Weiss fibre sequence \eqref{eqn:weiss}. Let $\ell_{\partial_0}$ be obtained by restricting the standard framing on $\bR^{2n}$ to a neighbourhood of $\partial D^{2n}$ and considering it as a boundary condition for $\theta$-structures on $D^{2n}$; this is a representative of the unique homotopy class of boundary conditions which extends, as in Lemma \ref{lem:SpaceThetaStr}. Then 
\[B \mr{Diff}^\theta_{\partial}(D^{2n};\ell_{\partial_0}) \coloneqq \mr{Bun}_\partial(TD^{2n},\theta^* \gamma;\ell_{\partial_0}) \sslash \mr{Diff}_\partial(D^{2n}).\]

Boundary connected sum makes the fibration sequence
\begin{equation}\label{eqn:tang-disc}\mr{Bun}_\partial(TD^{2n},\theta^* \gamma;\ell_{\partial_0}) \lra B \mr{Diff}^\theta_{\partial}(D^{2n};\ell_{\partial_0}) \lra B\mr{Diff}_\partial(D^{2n})\end{equation}
one of group-like $E_{2n}$-algebras, using models as in Remark \ref{rem:e2n-algebra}. By choosing suitable models reminiscent of Moore loops, we can extract from this a fibration sequence of group-like topological monoids, cf.~\cite[Section 4.2]{kupersdisk}.

\begin{proposition}\label{prop:weiss-tang} There is a fibration sequence
	\[B \mr{Diff}^\theta_{\partial}(D^{2n};\ell_{\partial_0}) \lra B\mr{Diff}_\partial^\theta(W_{g,1};\ell_\partial) \lra B\mr{Emb}^{\cong,\theta}_{\half \partial}(W_{g,1};\ell_{\half \partial}),\]
	which deloops once.
\end{proposition}

\begin{proof}If suffices to establish the delooped version. For brevity, we write (consistent with \cite{kupersdisk})
	\begin{align*}\cat{BD} \coloneqq B\mr{Diff}_\partial(D^{2n}), \qquad &\cat{BD}^\theta \coloneqq B \mr{Diff}^\theta_{\partial}(D^{2n};\ell_{\partial_0}),\\
	\cat{BW}_{g,1} \coloneqq B\mr{Diff}_\partial(W_{g,1}), \qquad &\cat{BW}_{g,1}^\theta \coloneqq B \mr{Diff}^\theta_{\partial}(W_{g,1};\ell_{\partial}),\\
	\cat{TD}^\theta \coloneqq \mr{Bun}_\partial(TD^{2n},\theta^* \gamma;\ell_{\partial_0}), \qquad &\cat{TW}^\theta_{g,1} \coloneqq \mr{Bun}_\partial(TW_{g,1},\theta^* \gamma;\ell_{\partial}).\end{align*}
	 Without loss of generality $\ell_\partial$ agrees with $\ell_{\partial_0}$ on the complement of $\half \partial W_{g,1}$. Then boundary connected sum gives compatible actions of the topological monoids \eqref{eqn:tang-disc}, now written $\cat{TD}^\theta \to \cat{BD}^\theta \to \cat{BD}$, on each of the terms in $\cat{TW}^\theta_{g,1} \to \cat{BW}_{g,1}^\theta \to \cat{BW}_{g,1}$.	Taking homotopy quotients, we get a commutative diagram
	\[\begin{tikzcd} \cat{TW}^\theta_{g,1} \rar \dar & \cat{TW}^\theta_{g,1}\sslash \cat{TD}^\theta \dar \rar & \ast \sslash \cat{TD}^\theta \dar \\
	\cat{BW}_{g,1}^\theta \rar \dar &\cat{BW}_{g,1}^\theta \sslash \cat{BD}^\theta \dar \rar & \ast \sslash \cat{BD}^\theta \dar \\
	\cat{BW}_{g,1} \rar & \cat{BW}_{g,1} \sslash \cat{BD} \rar & \ast \sslash \cat{BD} \end{tikzcd}\]
with columns and rows fibration sequences. The middle row is the desired fibration sequence, so it remains to identify the center term with $B\mr{Emb}^{\cong,\theta}_{\half \partial}(W_{g,1};\ell_{\half \partial})$. As in \cite[Section 4]{kupersdisk}, in a suitable model, restriction to a small copy of $W_{g,1}$ inside $W_{g,1}$ gives a map of fibration sequences
\[\begin{tikzcd}\cat{TW}^\theta_{g,1}\sslash \cat{TD}^\theta \rar \dar & \mr{Bun}_{\half \partial}(TW_{g,1},\theta^* \gamma;\ell_{\half \partial}) \dar \\
\cat{BW}_{g,1}^\theta \sslash \cat{BD}^\theta \rar \dar & B\mr{Emb}^{\cong,\theta}_{\half \partial}(W_{g,1};\ell_{\half \partial}) \dar \\
\cat{BW}_{g,1} \sslash \cat{BD} \rar & B\mr{Emb}^{\cong}_{\half \partial}(W_{g,1}).\end{tikzcd}\]
The bottom map is a weak equivalence by \cite[Theorem 4.17]{kupersdisk}, so it suffices to prove that the top map is a weak equivalence. Upon picking a reference framing $\tau$, we can identify it with the map
\[\mr{map}_\partial(W_{g,1},F^\mr{or}) \sslash \Omega^{2n}(F^\mr{or}) \lra \mr{map}_{\half \partial}(W_{g,1},F^\mr{or}),\]
induced by restriction to a small copy of $W_{g,1}$ inside $W_{g,1}$, which is indeed a weak equivalence.
\end{proof}

\subsubsection{Embeddings with tangential structures} Now that Proposition \ref{prop:weiss-tang} has established the relationship between diffeomorphisms and self-embeddings with $\theta$-structures, we study the latter.

Let us recall some of the notation introduced in Section \ref{sec:Notation}. A $\theta$-structure $\ell$ on $W_{g,1}$ which extends $\ell_{\partial}$, and hence also extends $\ell_{\half \partial}$, gives a basepoint in $\smash{B\mr{Emb}^{\cong,\theta}_{\half \partial}(W_{g,1};\ell_{\half \partial})}$ and we write 
\[\check{\Lambda}^{\theta,\ell}_g \coloneqq \pi_1(B\mr{Emb}^{\cong,\theta}_{\half \partial}(W_{g,1};\ell_{\half \partial}),\ell)\]
for the fundamental group at this basepoint. This surjects onto the subgroup $\Lambda^{\theta,[[\ell]]}_g = \mr{Stab}_{\Lambda_g}([[\ell]])$ of $\Lambda_g = \pi_0(\mr{Emb}^{\cong}_{\half \partial}(W_{g,1}))$, and $G^{\theta,[[\ell]]}_g \subset G'_g$ denotes the image of $\Lambda^{\theta,[[\ell]]}_g$. We defined $L^{\theta,\ell}_g$ by the extension
\[1 \lra L^{\theta,\ell}_g \lra \check{\Lambda}^{\theta,\ell}_g \lra G^{\theta,[[\ell]]}_g \lra 1.\] 

\begin{lemma}\label{lem:g-emb-tang-arithmetic} 
If $2n \geq 6$ and $B$ is $n$-connected, then $G^{\theta,[[\ell]]}_g \leq G'_g$ has finite index.
\end{lemma} 
\begin{proof}
By definition we have inclusions $G^{\theta,[\ell]}_g \leq G^{\theta,[[\ell]]}_g \leq G'_g$. The result then follows as $G^{\theta,[\ell]}_g$ is a finite index subgroup of $G'_g$ by Proposition \ref{prop:theta-fin-index}.
\end{proof}

We can therefore apply the setup of \eqref{eqn:setup} and ask whether a $\check{\Lambda}^{\theta,\ell}_g$-representation is $gr$-algebraic.

The group $\Lambda^{\theta,[[\ell]]}_g$ acts (in the homotopy category, i.e.\ by homotopy classes of homotopy equivalences) on the path-component $\mr{Bun}_{\half \partial}(TW_{g,1},\theta^*\gamma;\ell_{\half \partial})_\ell$ of $\ell$, and as preparation we study the action of the subgroup $J^{\theta,[[\ell]]}_g$ on the rational homotopy groups of $\mr{Bun}_{\half \partial}(TW_{g,1},\theta^*\gamma;\ell_{\half \partial})_\ell$. This action does not preserve the basepoint $\ell$, but the space in question is simple (as it is homotopy equivalent to a path component of $\prod_{2g}\Omega^n F^\mr{or}$) so there is still a well-defined action on homotopy groups.

\begin{lemma}\label{lem:jtheta-triv} 
	$J^{\theta,[[\ell]]}_g$ acts trivially on $\pi_i(\mr{Bun}_{\half \partial}(TW_{g,1},\theta^*\gamma;\ell_{\half \partial})_\ell)$.\end{lemma}

\begin{proof}
	The action of $\smash{\mr{Emb}^{\cong}_{\half \partial}(W_{g,1})}$ on $\mr{Bun}_{\half \partial}(TW_{g,1},\theta^*\gamma;\ell_{\half \partial})$ is through the derivative map to $\mr{Bun}_{\half \partial}(TW_{g,1})^\times$, so the action of the group $\Lambda^{\theta,[[\ell]]}_g$ on the path-component $\mr{Bun}_{\half \partial}(TW_{g,1},\theta^*\gamma;\ell_{\half \partial})_\ell$ factors through an action of the stabiliser of $[[\ell]]$ in $\Upsilon_g = \pi_0(\mr{Bun}_{\half \partial}(TW_{g,1})^\times)$. 
	
	As in Proposition \ref{prop:filteredlayer1} and Lemma \ref{lem:SpaceThetaStr} (ii), a reference framing $\tau$ gives rise to a split fibration sequence
	\begin{equation}\label{eq:FibSeq}
	\mr{map}_{\half \partial}(W_{g,1}, \mr{GL}_{2n}(\bR)) \lra \mr{Bun}_{\half \partial}(TW_{g,1})^\times \lra \mr{map}_{\half \partial}(W_{g,1},W_{g,1})^\times,
	\end{equation}
	as well as an identification
	\begin{equation}\label{eq:Untwist}
	\mr{Bun}_{\half \partial}(TW_{g,1},\theta^*\gamma;\ell_{\half \partial}) \simeq \mr{map}_{\half \partial}(W_{g,1},F^\mr{or}).
	\end{equation}
	In this description the group $J^{\theta,[[\ell]]}_g$ acts through $\pi_0(\mr{map}_{\half \partial}(W_{g,1}, \mr{GL}_{2n}(\bR)))$. The topological monoid $\mr{map}_{\half \partial}(W_{g,1}, \mr{GL}_{2n}(\bR))$ acts on $\mr{map}_{\half \partial}(W_{g,1},F^\mr{or})$ by the pointwise action of $\mr{GL}_{2n}(\bR)$ on $F^\mr{or}$. By the Eckmann--Hilton argument, the action of an element of $\mr{map}_{\half \partial}(W_{g,1},\mr{GL}_{2n}(\bR))$ is homotopic to sending it to $\mr{map}_{\half \partial}(W_{g,1},F^\mr{or})$ and acting through the $n$-fold loop structure. As elements of $J^{\theta,[[\ell]]}_g$ stabilise $[[\ell]]$ they lie in the kernel of $\pi_0(\mr{map}_\ast(W_{g,1},\mr{GL}_{2n}(\bR))) \to \pi_0(\mr{map}_\ast(W_{g,1},F^\mr{or}))$, and hence they act on $\mr{map}_{\half \partial}(W_{g,1},F^\mr{or})$ by maps homotopic to the identity.
\end{proof}

\begin{proposition}\label{prop:emb-tang-gr} 
	Supposing that $2n \geq 6$, then for each $i \geq 1$ the $\check{\Lambda}^{\theta,\ell}_g$-representation $\pi_{i+1}(B\mr{Emb}^{\cong,\theta}_{\half \partial}(W_{g,1};\ell_{\half \partial}),\ell) \otimes \bQ$ is $gr$-algebraic.
\end{proposition}

\begin{proof}
	Let $\mr{Bun}_{\half \partial}(TW_{g,1},\theta^*\gamma;\ell_{\half \partial})_{\Lambda_g \cdot \ell}$ denote the set of path components in the $\Lambda_g$-orbit of $[[\ell]] \in \mr{Str}^\theta_*(W_{g,1}) = \pi_0(\mr{Bun}_{\half \partial}(TW_{g,1},\theta^*\gamma;\ell_{\half \partial}))$, so that there is a fibration sequence
	\begin{equation}\label{eq:ForgetThetaEmb}
	\mr{Bun}_{\half \partial}(TW_{g,1},\theta^*\gamma;\ell_{\half \partial})_{\Lambda_g \cdot \ell} \overset{i}\lra B\mr{Emb}^{\cong,\theta}_{\half \partial}(W_{g,1};\ell_{\half \partial})_\ell \lra B\mr{Emb}^{\cong}_{\partial}(W_{g,1}).
	\end{equation}
	This gives a long exact sequence of abelian groups (or groups, resp.~sets) with action of $\check{\Lambda}^{\theta,\ell}_g$, the fundamental group of the total space \cite[p.~385]{spanier}. By property \eqref{enum:equiv-serre-1} of equivariant Serre classes, it suffices to prove that the higher rational homotopy groups of fibre and base are $gr$-algebraic.
	
	For the base, the action is through the homomorphism $\check{\Lambda}^{\ell,\theta}_g \to \Lambda_g$. That the higher rational homotopy groups are $gr$-algebraic $\smash{\check{\Lambda}^{\ell,\theta}_g}$-representations thus follows from the Theorem \ref{thm:emb-alg}, which says that they are $gr$-algebraic $\Lambda_g$-representations.
	
	For the fibre, we recall that $\mr{Bun}_{\half \partial}(TW_{g,1},\theta^*\gamma;\ell_{\half \partial})_\ell$ is simple and thus its fundamental group acts trivially on its higher homotopy groups. Hence the action of $\smash{\check{\Lambda}^{\theta,\ell}_g}$ factors through the surjection $\smash{\check{\Lambda}^{\theta,\ell}_g} \to \smash{\Lambda^{\theta,[[\ell]]}_g}$, and the $\Lambda^{\theta,[[\ell]]}_g$-action is given by self-embeddings acting on $\theta$-structures through the derivative map.
	
	By Lemma \ref{lem:jtheta-triv}, the action of the subgroup $J^{\theta,[[\ell]]}_g \subset \Lambda^{\theta,[[\ell]]}_g$ on the higher rational homotopy groups of $\mr{Bun}_{\half \partial}(TW_{g,1},\theta^*\gamma;\ell_{\half \partial})_\ell$ is trivial, and hence the $\smash{\Lambda^{\theta,\ell}_g}$-action factors over $\smash{G^{\theta,[[\ell]]}_g}$. As in the proof of Lemma \ref{lem:jtheta-triv}, a reference framing $\tau$ gives an identification \eqref{eq:Untwist}, from which we read off that 
	\[\pi_{i+1}(\mr{Bun}_{\half \partial}(TW_{g,1},\theta^*\gamma;\ell_{\half \partial})_\ell ) \cong H_n^\vee \otimes \pi_{n+i+1}(F^\mr{or}).\] The split fibration sequence \eqref{eq:FibSeq} provides a section $\mr{GL}(H_n) \to \Upsilon_g$, and $G_g^{\theta, [[\ell]]} \leq GL(H_n)$ acts via this in the evident way. Since the rational homotopy groups of $F^\mr{or}$ are finite-dimensional because those of $B$ and $BO(2n)$ are, this action is algebraic.
\end{proof}

In fact, it is easy to describe the group $J^{\theta,[[\ell]]}_g$. 

\begin{lemma}\label{lem:j-theta-ell} 
The subgroup 
	\[J^{\theta,[[\ell]]}_g = \mr{Stab}_{J_g}([[\ell]]) \leq J_g = \mr{Hom}(H_n,S\pi_n(SO(n)))\]
	is given by $\mr{Hom}(H_n, K_n) \leq \mr{Hom}(H_n, S\pi_n(SO(n)))$, where $K_n \coloneqq \ker[S\pi_n(SO(n)) \to \pi_n(SO(2n)) \to \pi_n(F^\mr{or})]$.\end{lemma}
\begin{proof}
The action of $J_g = \mr{Hom}(H_n, S\pi_n(SO(n)))$ on $\mr{Str}^{\theta}_\ast(W_{g,1}) \cong \mr{Hom}(H_n,\pi_n(F^\mr{or}))$ is via the map $S\pi_n(SO(n)) \to \pi_n(SO(2n)) \to \pi_n(F^\mr{or})$ and the action by addition on $\mr{Hom}(H_n, \pi_n(F^\mr{or}))$, so its stabiliser is precisely as described.
\end{proof}

Recall that $L^{\theta,\ell}_g$ denotes the kernel of the homomorphism $\check{\Lambda}^{\theta,\ell}_g \to G^{\theta,[[\ell]]}_g$. Using the fibration sequence \eqref{eq:ForgetThetaEmb} we introduce the notation
	\[\mr{im}(i) \coloneqq \mr{im}\left[i \colon \pi_1(\mr{Bun}_{\half \partial}(TW_{g,1},\theta^*\gamma;\ell_{\half \partial}),\ell) \to L^{\theta,\ell}_g\right],\]
	and it follows from that  fibration sequence that there is a commutative diagram
		\begin{equation}\label{eq:imidiagram}
		\begin{tikzcd} 
\mr{im}(i) \rar \dar[equals]& L^{\theta,\ell}_g \rar \dar& J^{\theta,[[\ell]]}_g \dar\\
\mr{im}(i) \rar  \dar & \check{\Lambda}^{\theta,\ell}_g \rar \dar& {\Lambda}^{\theta,[[\ell]]}_g \dar\\
\{e\} \rar & G^{\theta,[[\ell]]}_g \rar[equals] & G^{\theta,[[\ell]]}_g
	\end{tikzcd}
	\end{equation}
	where each row and column is a group extension.

\begin{lemma}\label{lem:l-alg} 
	If $2n \geq 6$, then $L^{\theta,\ell}_g$ is abelian (and hence nilpotent). If in addition $g \geq 2$, then the $G^{\theta,[[\ell]]}_g$-representations $H^i(L^{\theta,\ell}_g;\bQ)$ are algebraic.
\end{lemma}

\begin{proof}
By Lemma \ref{lem:jtheta-triv} the top extension in \eqref{eq:imidiagram} is central, and as $J_g^{\theta,[[\ell]]}$ is abelian (because $J_g$ is) it follows that $L^{\theta,\ell}_g$ is abelian. 
	
	The group $\check{\Lambda}^{\theta,\ell}_g$ acts on its normal subgroup $L^{\theta,\ell}_g$ by conjugation, and by \eqref{eq:imidiagram} there are compatible actions on $\mr{im}(i)$ and $J^{\theta,[[\ell]]}_g$ which factor over $\Lambda^{\theta,[[\ell]]}_g$. Since the top extension in \eqref{eq:imidiagram} is central there is a Serre spectral sequence of $\Lambda^{\theta,[[\ell]]}_g$-representations
	\[E^2_{p,q} = H^p(J^{\theta,[[\ell]]}_g;\bQ) \otimes H^q(\mr{im}(i);\bQ) \Longrightarrow H^{p+q}(L^{\theta,\ell}_g;\bQ),\]
	so by Theorem \ref{thm:representations} it suffices to show that the $\Lambda^{\theta,[[\ell]]}_g$-representations $H^p(J^{\theta,[[\ell]]}_g;\bQ)$ and $H^q(\mr{im}(i);\bQ)$ are $gr$-algebraic.
	
	Lemma \ref{lem:j-theta-ell} shows that $J^{\theta,[[\ell]]}_g = \mr{Hom}(H_n,K_n)$ where $K_n \subset S\pi_n(SO(n))$ is a certain subgroup, necessarily finitely-generated, with the $\Lambda^{\theta,[[\ell]]}_g$-action induced by the $G^{\theta,[[\ell]]}_g$-action by precomposition. Thus $H^p(J^{\theta,[[\ell]]}_g;\bQ) \cong \Lambda^p[H_n \otimes (K_n \otimes \bQ)^\vee]$, which is algebraic.
	
	For the group $\mr{im}(i)$, we use that it is abelian and so the $\Lambda^{\theta,[[\ell]]}_g$-representations $H^q(\mr{im}(i);\bQ)$ are all $gr$-algebraic if and only if $\mr{im}(i) \otimes \bQ$ is. But $\mr{im}(i)$ is a quotient of $\pi_1(\mr{Bun}_{\half \partial}(TW_{g,1},\theta^*\gamma;\ell_{\half \partial}),\ell)$, which in the proof of Proposition \ref{prop:emb-tang-gr} we showed is identified with $H_n^\vee \otimes \pi_{n+1}(F^\mr{or})$; the rationalisation of this is algebraic.
\end{proof}

\subsubsection{Starting the proof of Theorem \ref{thm:maintangent}} 
In analogy with $B\mr{Tor}^{\theta}_\partial(W_{g,1};\ell_\partial)_\ell$, let us define
\[B\mr{TorEmb}^{\cong,\theta}_{\half \partial}(W_{g,1};\ell_{\half \partial})_\ell \coloneqq \mr{hofib}\left[B\mr{Emb}^{\cong,\theta}_{\half \partial}(W_{g,1};\ell_{\half \partial})_\ell \lra BG^{\theta,[[\ell]]}_g\right].\]
Its fundamental group is $L^{\theta,\ell}_g$, the kernel of $\check{\Gamma}^{\theta,\ell}_g \to G^{\theta,[[\ell]]}_g$. 

The proof of Theorem \ref{thm:maintangent} proceeds along the lines of Theorem \ref{thm:main}, with some additional work keeping track of fundamental groups and path components. We restrict the fibre sequence of Proposition \ref{prop:weiss-tang} to the path-component of total space given by $[\ell] \in \mr{Str}^\theta_\partial(W_{g,1})$, and produce the analogue of the commutative diagram in the proof of Theorem \ref{thm:main}, which now takes the form
\begin{equation}\label{eq:BigThetaDiagram}
\begin{tikzcd} 
B\mr{Diff}^\theta_\partial(D^{2n};\ell_{\partial_0})_{A} \rar \dar  & B\mr{Diff}^\theta_\partial(D^{2n};\ell_{\partial_0})_{B} \rar \dar & G^{\theta,[[\ell]]}_g/G^{\theta,[\ell]}_g \dar \\
B\mr{Tor}^\theta_\partial(W_{g,1} ;\ell_\partial)_\ell \rar \dar & B\mr{Diff}^\theta_\partial(W_{g,1};\ell_\partial)_{\ell} \rar \dar & BG^{\theta,[\ell]}_g \dar \\ 
B\mr{TorEmb}^{\cong,\theta}_{\half \partial}(W_{g,1};\ell_{\half \partial})_\ell \rar &  B\mr{Emb}^{\cong, \theta}_{\half \partial}(W_{g,1};\ell_{\half \partial})_\ell \rar & BG^{\theta,[[\ell]]}_g,
\end{tikzcd}
\end{equation}
in which the rows and columns are fibration sequences and the subscripts $A, B \subset \pi_0(B\mr{Diff}^\theta_\partial(D^{2n};\ell_{\partial_0}))$ denote certain collections of path components. As the leftmost two columns deloop, we see that $A$ and $B$ are in fact subgroups of $\pi_0(B\mr{Diff}^\theta_\partial(D^{2n};\ell_{\partial_0}))$ and may be described in terms of the exact sequences of groups $\pi_1(B\mr{Tor}^\theta_\partial(W_{g,1} ;\ell_\partial), \ell) \to L^{\theta,\ell}_g \to A \to 0$ and $\check{\Gamma}^{\theta, \ell}_g \to \check{\Lambda}^{\theta, \ell}_g \to B \to 0$.

\begin{remark}
 It follows from the delooping of the middle column that $G^{\theta,[\ell]}_g \leq G^{\theta,[[\ell]]}_g$ is in fact a normal subgroup.
\end{remark}

\subsubsection{Interlude: The groups $A$ and $B$ are finite} 
First observe that the map
\[\pi_1(\mr{Bun}_{\partial}(TW_{g,1},\theta^*\gamma;\ell_{\partial}), \ell) \lra \pi_1(\mr{Bun}_{\half \partial}(TW_{g,1},\theta^*\gamma;\ell_{\half \partial}), \ell)\]
is surjective, by the discussion in Section \ref{sec:Relaxing} and the long exact sequence on homotopy groups. From the exact sequence
\[\pi_1(\mr{Bun}_{\half \partial}(TW_{g,1},\theta^*\gamma;\ell_{\half \partial}), \ell) \lra L_g^{\theta, \ell} \lra J_g^{\theta, [[\ell]]} \lra 0\]
coming from $B\mr{Tor}^\theta_\partial(W_{g,1} ;\ell_\partial) \simeq \mr{Bun}_{\half \partial}(TW_{g,1},\theta^* \gamma;\ell_{\half \partial}) \sslash \mr{TorEmb}^{\cong}_{\half \partial}(W_{g,1})$ it then follows that there is a surjection $\smash{J_g^{\theta, [[\ell]]}} \to A$ with kernel given by the image of the natural homomorphism $I_g^{\theta, [\ell]} \to J_g^{\theta, [[\ell]]}$.

To finish our analysis of the group $A$, we need to understand how much the group $J^{\theta,[[\ell]]}_g = \mr{Stab}_{J_g}([[\ell]])$ differs from $I^{\theta,[\ell]}_g = \mr{Stab}_{I_g}([\ell]$. 

\begin{lemma}\label{lem:DiscActTriv}
	The action up to homotopy of the subgroup $\Theta_{2n+1} \leq I_g \leq \Gamma_g$ on the space $\mr{Bun}_{\partial}^{\theta}(TW_{g,1}; \ell_\partial)$ is trivial, so in particular $\Theta_{2n+1} \leq I_g^{\theta, [\ell]}$.
\end{lemma}
\begin{proof}
	This subgroup acts on $\mr{Bun}_{\partial}^{\theta}(TW_{g,1}; \ell_\partial)$ via the derivative map (with respect to the standard framing of $D^{2n} \subset \bR^{2n})$
	\[\Theta_{2n+1} = \pi_0(\mr{Diff}_\partial(D^{2n})) \lra \pi_{2n}(SO(2n)),\]
	which we claim is trivial. Smoothing theory identifies the space of smooth structures on a $2n$-dimensional topological manifold with the space of lifts to $BO(2n)$ of the map to $B\mr{Top}(2n)$ classifying its tangent microbundle \cite[Essays IV, V]{kirbysiebenmann}. In particular, for a disc this provides an equivalence \cite[Theorem 4.4(b)]{burglashof} \cite[Theorem V.3.4]{kirbysiebenmann}
	\[\mr{Diff}_\partial(D^{2n}) \simeq \Omega^{2n+1} \mr{Top}(2n)/O(2n),\] 
	under which the derivative map is the connecting map $\pi_{2n+1}(\mr{Top}(2n)/O(2n)) \to \pi_{2n}(O(2n))$ in the fibration sequence. Then the result follows directly from \cite[Proposition 5.4 (iv)]{burglashof}, which says that $\pi_i(O(k)) \to \pi_i(\mr{Top}(k))$ is injective for $k \geq 5$ and $i \leq k$.
\end{proof}

As a consequence of the previous lemma, the action of $I_g$ on $\mr{Str}^\theta_\partial(W_{g,1})$ factors over $J_g$, and $I_g^{\theta, [\ell]}/ \Theta_{2n+1}$ is identified with the stabiliser in $J_g$ of $[\ell] \in \mr{Str}^\theta_\partial(W_{g,1})$.

\begin{lemma}\label{lem:jgtheta-finite}
	The subgroup
	\[\mr{Stab}_{J_g}([\ell]) \leq J_g = \mr{Hom}(H_n, S\pi_n(SO(n)))\]
	has finite index in the subgroup $J_g^{\theta,[[\ell]]} = \mr{Hom}(H_n, K_n) \leq \mr{Hom}(H_n, S\pi_n(SO(n)))$, where $K_n \coloneqq \ker[S\pi_n(SO(n)) \to \pi_n(SO(2n)) \to \pi_n(F^\mr{or})]$.
\end{lemma}

\begin{proof}
	Recall from Section \ref{sec:Relaxing} that there is a surjection $\mr{Str}^\theta_\partial(W_{g,1}) \to \mr{Str}^\theta_\ast(W_{g,1})$, which by a choice of framing is identified with the surjection
	\[\pi_0(\mr{map}_\partial(W_{g,1}, F^\mr{or})) \lra \mr{Hom}(H_n,\pi_n(F^\mr{or})).\]
	The group $\pi_{2n}(F^\mr{or})$ acts freely and transitively on the fibres of this map. Let us write $[[\ell]] \in \mr{Hom}(H_n,\pi_n(F^\mr{or}))$ for the image of $[\ell] \in \pi_0(\mr{map}_\partial(W_{g,1}, F^\mr{or}))$ under this map. 
	
	If $[\varphi] \in J_g$ fixes $[\ell]$ then it certainly fixes $[[\ell]]$. Conversely, there is defined a function $f_\ell \colon \mr{Stab}_{J_g}([[\ell]]) \to \pi_{2n}(F^\mr{or})$ by the property
	\[[\ell]\circledast [\varphi] = f_\ell([\varphi]) \cdot [\ell].\]
	It follows from this formula and the fact that the action of $J_g$ commutes with that of $\pi_{2n}(F^\mr{or})$ that $f_\ell$ is a homomorphism. Therefore we have an exact sequence
	\[1 \lra \mr{Stab}_{J_g}([\ell]) \lra J_g^{\theta,[[\ell]]} = \mr{Stab}_{J_g}([[\ell]]) \overset{f_\ell}\lra \pi_{2n}(F^\mr{or}),\]
	with $\mr{Stab}_{J_g}([\ell])$ the group of interest. To finish the proof of this lemma we must therefore show that the homomorphism $f_\ell$ has finite image.
	
	We can apply the same construction to the full $\Lambda_g = \Gamma_g/\Theta_{2n+1}$-action, giving an exact sequence
	\[1 \lra \mr{Stab}_{\Lambda_g}([\ell]) \lra \Lambda^{\theta,[[\ell]]}_g = \mr{Stab}_{\Lambda_g}([[\ell]]) \overset{\hat{f}_\ell}\lra \pi_{2n}(F^\mr{or}),\]
	where $\hat{f}_\ell$ extends $f_\ell$.
.
	The Serre spectral sequence for the extension
	\[1 \lra J^{\theta,[[\ell]]}_g = \mr{Stab}_{J_g}([[\ell]]) \lra \Lambda^{\theta,[[\ell]]}_g = \mr{Stab}_{\Lambda_g}([[\ell]]) \lra G^{\theta,[[\ell]]}_g \lra 1,\]
	gives an exact sequence
	\[\cdots \lra \mr{Hom}(H_n, K_n)_{G^{\theta,[[\ell]]}_g} \lra (\Lambda^{\theta,[[\ell]]}_g)^\mr{ab} \lra (G^{\theta,[[\ell]]}_g)^\mr{ab} \lra 0.\]
	
	\vspace{0.5ex}
	\noindent \textbf{Claim.} As long as $g \geq 2$ the two outer terms of this exact sequence are finite.
	\vspace{-0.5ex}
	\begin{proof}[Proof of claim]
		For the left-hand term, first note that the group $\mr{Hom}(H_n, K_n)$ is finitely-generated, hence so are the coinvariants. Then the $G^{\theta,[[\ell]]}_g$-coinvariants of $\mr{Hom}(H_n, K_n) \otimes \bQ$ are isomorphic to the $G^{\theta,[[\ell]]}_g$-invariants. The subgroup $G^{\theta,[[\ell]]}_g \leq G'_g$ has finite index by Lemma \ref{lem:g-emb-tang-arithmetic}. As $g \geq 2$ (this only needs $g \geq 1$ when $n$ is odd), it follows that $G^{\theta,[[\ell]]}_g$ is Zariski dense in either $\Sp_{2g}(\bQ)$, $\SO_{g,g}(\bQ)$, or $\OO_{g,g}(\bQ)$, and we may as well take invariants with respect to this larger group: it is evident this is $0$. 
		
		For the right-hand term, the abelian group $(G^{\theta,[[\ell]]}_g)^\mr{ab}$ is finitely generated since the arithmetic group $G^{\theta,[[\ell]]}_g$ is, and $H^1(G^{\theta,[[\ell]]}_g;\bQ)$ vanishes since $g \geq 2$ by \cite[Corollary 7.6.17]{margulis}.
	\end{proof}
	
	It follows that $(\Lambda^{\theta,[[\ell]]}_g)^\mr{ab}$ is also a finite group. We then proceed as follows. The homomorphism $f_\ell \colon J^{\theta,[[\ell]]}_g \to \pi_{2n}(F^\mr{or})$ extends through the homomorphism $\smash{\hat{f}_\ell}$ as discussed above, and this in turn factors through a homomorphism $(\Lambda^{\theta,[[\ell]]}_g)^\mr{ab} \to \pi_{2n}(F^\mr{or})$, as the target is an abelian group:
	\[\begin{tikzcd}J^{\theta,[[\ell]]}_g \rar{f_\ell} \dar[hook] &  \pi_{2n}(F^\mr{or}) \\
	 \Lambda^{\theta,[[\ell]]}_g \arrow{ru}[description]{\hat{f}_\ell} \rar & (\Lambda^{\theta,[[\ell]]}_g)^\mr{ab} \uar.\end{tikzcd}\]
	 But $(\Lambda^{\theta,[[\ell]]}_g)^\mr{ab}$ is a finite group, so $f_\ell$ has finite image.
\end{proof}

As we have $A \cong J^{\theta, [[\ell]]}_g / \mr{Stab}_{J_g}([\ell])$, and $G^{\theta,[[\ell]]}_g/G^{\theta,[\ell]}_g$ is finite by Lemma \ref{lem:g-emb-tang-arithmetic}, we deduce:

\begin{corollary}\label{cor:Afinite}
The groups $A$ and $B$ are finite.
\end{corollary}

\subsubsection{Finishing the proof of Theorem  \ref{thm:maintangent}} 

We wish to analyse the Serre spectral sequence for the left-hand column of \eqref{eq:BigThetaDiagram}, and it is awkward that the fibre of this fibration is not connected. To deal with this we consider the subgroup
\[\overline{L^{\theta,\ell}_g} \coloneqq \mr{ker}(L^{\theta,\ell}_g \to A) = \mr{im}(\pi_1(B\mr{Tor}^\theta_\partial(W_{g,1} ;\ell_\partial), \ell) \to L^{\theta,\ell}_g)\]
of $L^{\theta,\ell}_g$ and let $\overline{B\mr{TorEmb}^{\cong,\theta}_{\half \partial}(W_{g,1};\ell_{\half \partial})_\ell}$ denote the corresponding covering space. There is then a fibration sequence
\begin{equation}\label{eq:connFibre}
B\mr{Diff}^\theta_\partial(D^{2n};\ell_{\partial_0})_{\ell_0} \lra B\mr{Tor}^\theta_\partial(W_{g,1} ;\ell_\partial)_\ell \lra \overline{B\mr{TorEmb}^{\cong,\theta}_{\half \partial}(W_{g,1};\ell_{\half \partial})_\ell},
\end{equation}
where the fibre is now path connected.

 Similarly, consider the subgroup $\overline{\check{\Lambda}^{\theta,\ell}_g} \coloneqq \mr{ker}(\check{\Lambda}^{\theta,\ell}_g \to B)$ of $\check{\Lambda}^{\theta,\ell}_g$ and the corresponding covering space $\overline{B\mr{Emb}^{\cong,\theta}_{\half \partial}(W_{g,1};\ell_{\half \partial})_\ell}$. It is easy to check that there is an induced fibration sequence
\[\overline{B\mr{TorEmb}^{\cong,\theta}_{\half \partial}(W_{g,1};\ell_{\half \partial})_\ell} \lra  \overline{B\mr{Emb}^{\cong, \theta}_{\half \partial}(W_{g,1};\ell_{\half \partial})_\ell} \lra BG^{\theta,[\ell]}_g,\]
so $G^{\theta,[\ell]}_g$ acts on the cohomology of $\overline{B\mr{TorEmb}^{\cong,\theta}_{\half \partial}(W_{g,1};\ell_{\half \partial})_\ell}$.

\begin{lemma}\label{lem:TorEmbThetaAlgebraic}
Suppose that $2 n \geq 6$ and that $g \geq 2$. Then the $G_g^{\theta,[\ell]}$-representations $H^i(\overline{B\mr{TorEmb}^{\cong,\theta}_{\half \partial}(W_{g,1};\ell_{\half \partial})_\ell};\bQ)$ are algebraic.
\end{lemma}
\begin{proof}
We consider the fibration
\[\tau_{>1} B\mr{Emb}^{\cong,\theta}_{\half \partial}(W_{g,1};\ell_{\half \partial})_\ell \lra \overline{B\mr{TorEmb}^{\cong,\theta}_{\half \partial}(W_{g,1};\ell_{\half \partial})_\ell} \lra B\overline{L^{\theta,\ell}_g},\]
where $\tau_{>1}$ denotes the 1-connected cover. Firstly, combining Proposition \ref{prop:emb-tang-gr} and Lemma \ref{lem:filtered-algebraic-homotopy-to-homology} shows that the $\check{\Lambda}_g^{\theta,\ell}$-representations \[H^i(\tau_{>1} B\mr{Emb}^{\cong,\theta}_{\half \partial}(W_{g,1};\ell_{\half \partial})_\ell;\bQ)\] are $gr$-algebraic. Secondly, we will show below that the $G_g^{\theta,\ell}$-representations $H^i(\overline{L_g^{\theta,\ell}};\bQ)$ are algebraic. Together these provide the input for Lemma \ref{lem:i-connective}, which gives the required result.

This second ingredient is proved using Lemma \ref{lem:l-alg}. As $L^{\theta, \ell}_g$ is abelian the action of the group $A$ on the kernel $\overline{L^{\theta, \ell}_g} = \mr{ker}(L^{\theta, \ell}_g \to A)$ is trivial so, as $A$ is finite by Corollary \ref{cor:Afinite}, the Serre spectral sequence shows that 
\[H^*({L^{\theta, \ell}_g};\bQ) \overset{\cong}\lra H^*(\overline{L^{\theta, \ell}_g};\bQ).\]
The second part of Lemma \ref{lem:l-alg} then gives the desired conclusion.
\end{proof}

We can now complete the proof of Theorem \ref{thm:maintangent}. We make two observations about the fibration \eqref{eq:connFibre}. Firstly, $B\mr{Diff}^\theta_\partial(D^{2n};\ell_{\partial_0})_{\ell_0}$ has degree-wise finite-dimensional rational cohomology. This follows from the fibre sequence \eqref{eqn:tang-disc} using that $F^\mr{or}$ has degree-wise finite-dimensional rational cohomology (as $B$ does, by assumption), and that $B\mr{Diff}_\partial(D^{2n})$ has degree-wise finite-dimensional rational cohomology by \cite[Theorem A]{kupersdisk}) and finite fundamental group. Secondly, the fundamental group of the base of \eqref{eq:connFibre} acts trivially on the cohomology of the fibre, as this fibration deloops. The rational cohomology Serre spectral sequence for this fibration and Lemma \ref{lem:TorEmbThetaAlgebraic} then gives the result.

\subsection{Proof of Theorem \ref{thm:nilp-tangential}}

Having established algebraicity for the cohomology of $B\mr{Tor}^{\theta}_\partial(W_{g,1};\ell_\partial)_\ell$, we now prove that it is a nilpotent space.

\begin{proposition} Let $2n \geq 6$ and $B$ be $n$-connected. Then $B\mr{TorEmb}^\theta_{\half \partial}(W_{g,1};\ell_{\half \partial})_\ell$ is nilpotent.\end{proposition}

\begin{proof}Its fundamental group $L^{\theta,\ell}_g$ is nilpotent by Lemma \ref{lem:l-alg} (in fact it is abelian), so it suffices to show that it acts nilpotently on the higher homotopy groups. To do so, let $\overline{ B\mr{TorEmb}^{\cong}_{\partial}(W_{g,1})}$ denote the covering space corresponding to the subgroup $J^{\theta,[[\ell]]}_g \leq J_g$, and consider the fibration sequence 
	\[\mr{Bun}_{\half \partial}(TW_{g,1},\theta^*\gamma;\ell_{\half \partial})_{\ell} \lra B\mr{TorEmb}^{\cong,\theta}_{\half \partial}(W_{g,1};\ell_{\half \partial})_\ell \lra \overline{B\mr{TorEmb}^{\cong}_{\partial}(W_{g,1})}.\]
As we remarked before the proof of Lemma \ref{lem:jtheta-triv}, the fibre is simple. Furthermore, the base is nilpotent by Proposition \ref{prop:TorEmbNilpotent} (and the fact that a cover of a nilpotent space is nilpotent). The fundamental group $J^{\theta,[[\ell]]}_g$ of the base acts on the fibre, and as the fibre is simple this gives an action on the homotopy groups of the fibre: it suffices to show that this action is nilpotent, but Lemma \ref{lem:jtheta-triv} shows that $J^{\theta,[[\ell]]}_g$ in fact acts trivially on the homotopy groups of the fibre.
\end{proof}

The proof of Theorem \ref{thm:nilp-tangential} is now analogous to that of Theorem \ref{thm:nilp}:

\begin{proof}[Proof of Theorem \ref{thm:nilp-tangential}] 
Delooping \eqref{eq:connFibre}, $B\mr{Tor}^\theta_\partial(W_{g,1};\ell_\partial)_\ell$ is the basepoint component of the homotopy fibre of a map
	\[\overline{B\mr{TorEmb}^\theta_{\half \partial}(W_{g,1};\ell_{\half \partial})_\ell} \lra B(B\mr{Diff}^\theta_\partial(D^{2n};\ell_{\partial_0})_{\ell_0}),\]
The domain is a nilpotent space by the previous proposition (again because a cover of a nilpotent space is nilpotent) so by Lemma \ref{lem:nilp-fib} the space $B\mr{Tor}^\theta_\partial(W_{g,1};\ell_\partial)_\ell$ is also nilpotent.
\end{proof}

\vspace{.5em}

\bibliographystyle{amsalpha}
\bibliography{../../cell}

\vspace{.5em}

\end{document}